%% file: abel-survey-merker.tex
\documentclass[graybox,leqno]{svmult}

\usepackage{amssymb,amsbsy,amsmath,amsfonts,amscd,times,
graphics,color,xypic,fancyhdr,multicol,fancybox,
graphicx,mathrsfs,rotating,ifthen,wasysym}

\usepackage{mathptmx}
\usepackage{helvet}
\usepackage{courier}
\usepackage{type1cm} 

\usepackage{bm}

\usepackage[bottom]{footmisc}

\newcommand{\C}{\mathbb{C}}
\renewcommand{\D}{\mathbb{D}}
\newcommand{\N}{\mathbb{N}}
\renewcommand{\P}{\mathbb{P}}
\newcommand{\Q}{\mathbb{Q}}
\newcommand{\R}{\mathbb{R}}
\newcommand{\Z}{\mathbb{Z}}

\begin{document}


\title*{Rationality in Differential Algebraic Geometry}

\author{Jo\"el Merker}

\institute{Jo\"el Merker 
\at 
D\'epartement de Math\'ematiques d'Orsay, 
B\^atiment 425,
Facult\'e des Sciences,
F-91405 Orsay Cedex, 
France,
\email{Joel.Merker@math.u-psud.fr}}

\maketitle


\vspace{-1cm}

\abstract{
Parametric Cartan theory of exterior differential systems, 
and explicit cohomology of projective manifolds reveal 
united rationality features of differential algebraic geometry.}

\section{Rationality}

The natural integer numbers:
\[
1,\,\,\,
2,\,\,\,
3,\,\,\,
4,\,\,\,
5,\,\,\,
6,\,\,\,
7,\,\,\,
8,\,\,\,
9,\,\,\,
10,\,\,\,
11,\,\,\,
12,\,\,\,
13,\,\,\,
\dots,\,\,\,
2013,\,\,\,
\dots\dots
\]
necessarily hint at some `{\sl invention}' of the Zero. While for the
Greeks, the actual `$\infty$' and the actual `$0$' did not `{\em
exist}', the Babylonians used the symbol `0' in numeration. In India
({\em cf.} {\em e.g.} Brahmagupta), the zero comes from
self-subtraction:
\[
0
\overset{\text{\rm def}}{\,:=\,\,}
{\bf a}
\,{\bf -}\,
{\bf a}.
\]
In rational numbers:
\[
\frac{\,\,p\,\,}{\,\,q\,\,}
\ \ \ \ \ \ \ \ \ \ \ \ \
{\scriptstyle{(q\,\neq\,0)}},
\] 
division by zero is and must be excluded.

The present paper aims at showing that {\em higher abstract
conceptions in advanced mathematics depend upon archetypical
rational computational phenomena}.

Several instances of deeper rationality facts
will hence be surveyed:

\medskip\noindent$\square$\,
In Cartan's theory of exterior differential systems;

\medskip\noindent$\square$\,
In Complex Algebraic Geometry.

\section{Equivalences of 5-dimensional CR manifolds}

\begin{quotation}
Despite their importance, until now, the invariants of strictly
pseudoconvex domains have been fully computed, to our knowledge, only
in the case of the unit ball $\mathbb{B}^{n+1}$, where they all vanish!
\hfill Sidney {\sc Webster} (\cite{Webster-2000}).
\end{quotation}

\noindent
Real analytic ($\mathcal{ C}^\omega$) {\sl CR-generic} submanifolds $M
\subset \C^{ n + c}$ of codimension $c \geqslant 0$ are those
satisfying $TM + J(TM) = T\C^{n+c} \big\vert_M$, where $J \colon
T\C^{n + c} \longrightarrow T\C^{n+c}$ denotes the standard complex
structure and then $TM \cap J(TM)$ has constant real dimension $2n =:
2\, {\sf CRdim}\, M$, while ${\sf dim}_\R M = 2n+c$; general $\mathcal{
C}^\omega$ CR submanifolds $M \subset \C^\nu$, {\em i.e.} those for
which ${\sf dim}\, \big( T_pM \cap J(T_pM)\big)$ is constant for $p \in M$,
are always locally CR-generic in some complex submanifold \cite{
Merker-Porten-2006}, hence CR-genericity is not a restriction.

\medskip\noindent{\bf Problem 1.} 
{\sl Classify local $\mathcal{ C}^\omega$
CR-generic submanifolds $M^{ 2n+c} 
\subset \C^{ n+c}$ under local biholomorphisms of $\C^{n+c}$ up
to dimension $2n+c \leqslant {\bf 5}$.} 

\medskip

If $c = 0$, then $M \cong \C^n$, where `$\cong$' means
`{\sl locally biholomorphic}'; if $n = 0$, then
$M \cong \R^c$. Assume therefore $c \geqslant 1$ and $n \geqslant 1$.
The possible CR dimensions and real codimensions are:
\[
\aligned
&
2n+c={\bf 3}
\ \ \ \ \ \
\Longrightarrow
\ \
\Big\{
n={\bf 1},\ \ \ \ \
c={\bf 1},
\\
&
2n+c={\bf 4}
\ \ \ \ \ \
\Longrightarrow
\ \
\Big\{
n={\bf 1},\ \ \ \ \
c={\bf 2},
\\
&
2n+c={\bf 5}
\ \ \ \ \ \
\Longrightarrow
\ \
\left\{
\aligned
&
n={\bf 1},\ \ \ \ \
c={\bf 3},
\\
&
n={\bf 2},\ \ \ \ \
c={\bf 1}.
\endaligned\right.
\endaligned
\]

In local coordinates $(z_1, \dots, z_n, w_1, \dots, w_c \big) = ( x_1
+ {\scriptstyle{\sqrt{-1}}}\, y_1, \dots, x_n+ {\scriptstyle{\sqrt{-1}}}\, y_n,\, u_1+ {\scriptstyle{\sqrt{-1}}}\, v_1, \dots,
u_c+ {\scriptstyle{\sqrt{-1}}}\,v_c )$, represent with graphing $\mathcal{ C}^\omega$
functions $\varphi_{\scriptscriptstyle{\bullet}}$:
\[
\aligned
&
M^3\subset\C^2
\colon
\ \ \ \ \ \,
\Big[
\,\,
v
=
\varphi(x,y,u),
\\
&
M^4\subset\C^3
\colon
\ \ \ \ \
\left[
\aligned
\,\,
v_1
&
=
\varphi_1(x,y,u_1,u_2),
\\
\,\,
v_2
&
=
\varphi_2(x,y,u_1,u_2),
\endaligned\right.
\\
& 
M^5\subset\C^4
\colon
\ \ \ \ \
\left[
\aligned
\,\,
v_1
&
=
\varphi_1(x,y,u_1,u_2,u_3),
\\
\,\,
v_2
&
=
\varphi_2(x,y,u_1,u_2,u_3), 
\\
\,\,
v_3
&
=
\varphi_3(x,y,u_1,u_2,u_3),
\endaligned\right.
\\
&
M^5\subset\C^3
\colon
\ \ \ \ \ \
\Big[
\,\,
v
=
\varphi(x_1,y_1,x_2,y_2,u).
\endaligned
\]
Before proceeding further, answer (partly) Webster's quote.

\subsection{Explicit characterization of sphericity}

Consider for instance a hypersurface $M^3 \subset \C^2$.
As its graphing function $\varphi$ is real analytic, 
$w$ can be (locally) solves (\cite{ Merker-Porten-2006, 
Merker-2011}):
\[
w
=
\Theta
\big(z,\overline{z},\overline{w}\big).
\]
Letting the `round' unit 3-sphere $S^3 \subset \C^2$ be:
\[
1
=
z\overline{z}
+
w\overline{w}
=
x^2+y^2+u^2+v^2,
\]
a Cayley transform (\cite{ Merker-2011}) 
maps $S^3 \backslash \{ p_\infty\}$ with $p_\infty := 
(0, -1)$ biholomorphically
onto the {\sl Heisenberg sphere}:
\[
w
=
\overline{w}+2{\scriptstyle{\sqrt{-1}}}\,z\overline{z}.
\]

An intrinsic local generator for the fundamental
subbundle:
\[
T^{1,0}M 
:= 
\big\{ 
X-{\scriptstyle{\sqrt{-1}}}\,J(X) 
\colon 
X\in TM\cap J(TM)
\big\}
\]
of $TM \otimes_\R \C$ is:
\[
L
=
\frac{\partial}{\partial z}.
\]
Also, an intrinsic generator for $T^{0, 1}M := \overline{
T^{1, 0} M}$ is:
\[
\overline{L}
:=
\frac{\partial}{\partial\overline{z}}
-
\frac{\Theta_{\overline{z}}(z,\overline{z},\overline{w})}{
\Theta_{\overline{w}}(z,\overline{z},\overline{w})}\,
\frac{\partial}{\partial\overline{w}}.
\]
In the Lie bracket:
\[
\big[L,\,\overline{L}\big]
=
\bigg[
\frac{\partial}{\partial z},\,
\frac{\partial}{\partial\overline{z}}
-
\frac{\Theta_{\overline{z}}}{\Theta_{\overline{w}}}\,
\frac{\partial}{\partial\overline{w}}
\bigg]
=
\bigg(
\frac{-\,\Theta_{\overline{w}}\,\Theta_{z\overline{z}}
+
\Theta_{\overline{z}}\,\Theta_{z\overline{w}}}{
\Theta_{\overline{w}}\,\Theta_{\overline{w}}}
\bigg)\,
\frac{\partial}{\partial\overline{w}},
\]
an explicit {\sl Levi factor in coordinates} appears:
\[
\frac{\pmb{-\,\Theta_{\overline{w}}\,\Theta_{z\overline{z}}
+
\Theta_{\overline{z}}\,\Theta_{z\overline{w}}}}{
\pmb{\Theta_{\overline{w}}\,\Theta_{\overline{w}}}}.
\]
The assumption that $M^3 \subset \C^2$ is smooth reads:
\[
0
\neq
\pmb{\Theta_{\overline{w}}}
\,\,\,\,
\text{\rm vanishes at no point}.
\]
The assumption that $M$ is Levi nondegenerate reads:
\[
0
\neq
\pmb{-\,\Theta_{\overline{w}}\,\Theta_{z\overline{z}}
+
\Theta_{\overline{z}}\,\Theta_{z\overline{w}}}
\,\,\,\,
\text{\rm also vanishes at no point}.
\]

\medskip\noindent{\bf General principle.} 
{\em Various geometric assumptions enter computations in 
denominator places}.

\medskip

Here is a first illustration. 

\begin{theorem}
\label{explicit-sphericity}
{\rm (\cite{Merker-2011})}
An arbitrary real analytic hypersurface $M^3 \subset \C^2$
which is Levi nondegenerate:
\[
w
=
\Theta\big(z,\,\overline{z},\,\overline{w}\big),
\]
is locally biholomorphically equivalent to the Heisenberg sphere if
and only if:
\[
\aligned
0
\equiv
\bigg(
\frac{-\,\Theta_{\overline{w}}}{
\Theta_{\overline{z}}\Theta_{z\overline{w}}
-\Theta_{\overline{w}}\Theta_{z\overline{z}}}\,
\frac{\partial}{\partial\overline{z}}
+
\frac{\Theta_{\overline{z}}}{
\Theta_{\overline{z}}\Theta_{z\overline{w}}
-\Theta_{\overline{w}}\Theta_{z\overline{z}}}\,
\frac{\partial}{\partial\overline{w}}
\bigg)^2
\big[{\sf AJ}^4(\Theta)\big]
\endaligned
\]
identically in $\C \big\{ z, \overline{ z}, \overline{ w} \big\}$,
where:
\[
\footnotesize
\aligned
{\sf AJ}^4(\Theta)
&
:=
\frac{1}{
\pmb{[\Theta_{\overline{z}}\Theta_{z\overline{w}}
-\Theta_{\overline{w}}\Theta_{z\overline{z}}]^3}}
\bigg\{
\Theta_{zz\overline{z}\overline{z}}
\bigg(
\Theta_{\overline{w}}\Theta_{\overline{w}}
\left\vert\!\!
\begin{array}{cc}
\Theta_{\overline{z}} & \Theta_{\overline{w}}
\\
\Theta_{z\overline{z}} & \Theta_{z\overline{w}}
\end{array}
\!\!\right\vert
\bigg)
-
\\
&
\ \ \ \ \
-\,
2\Theta_{zz\overline{z}\overline{w}}
\bigg(
\Theta_{\overline{z}}\Theta_{\overline{w}}
\left\vert\!\!
\begin{array}{cc}
\Theta_{\overline{z}} & \Theta_{\overline{w}}
\\
\Theta_{z\overline{z}} & \Theta_{z\overline{w}}
\end{array}
\!\!\right\vert
\bigg)
+
\Theta_{zz\overline{w}\overline{w}}
\bigg(
\Theta_{\overline{z}}\Theta_{\overline{z}}
\left\vert\!\!
\begin{array}{cc}
\Theta_{\overline{z}} & \Theta_{\overline{w}}
\\
\Theta_{z\overline{z}} & \Theta_{z\overline{w}}
\end{array}
\!\!\right\vert
\bigg)
+
\\
&
\ \ \ \ \
+
\Theta_{zz\overline{z}}
\bigg(
\Theta_{\overline{z}}\Theta_{\overline{z}}
\left\vert\!\!
\begin{array}{cc}
\Theta_{\overline{w}} & \Theta_{\overline{w}\overline{w}}
\\
\Theta_{z\overline{w}} & \Theta_{z\overline{w}\overline{w}}
\end{array}
\!\!\right\vert
-
2\Theta_{\overline{z}}\Theta_{\overline{w}}
\left\vert\!\!
\begin{array}{cc}
\Theta_{\overline{w}} & \Theta_{\overline{z}\overline{w}}
\\
\Theta_{z\overline{w}} & \Theta_{z\overline{z}\overline{w}}
\end{array}
\!\!\right\vert
+
\Theta_{\overline{w}}\Theta_{\overline{w}}
\left\vert\!\!
\begin{array}{cc}
\Theta_{\overline{w}} & \Theta_{\overline{z}\overline{z}}
\\
\Theta_{z\overline{w}} & \Theta_{z\overline{z}\overline{z}}
\end{array}
\!\!\right\vert
\bigg)
+
\\
&
\ \ \ \ \
+
\Theta_{zz\overline{w}}
\bigg(\!\!
-\Theta_{\overline{z}}\Theta_{\overline{z}}
\left\vert\!\!
\begin{array}{cc}
\Theta_{\overline{z}} & \Theta_{\overline{w}\overline{w}}
\\
\Theta_{z\overline{z}} & \Theta_{z\overline{w}\overline{w}}
\end{array}
\!\!\right\vert
+
2\Theta_{\overline{z}}\Theta_{\overline{w}}
\left\vert\!\!
\begin{array}{cc}
\Theta_{\overline{z}} & \Theta_{\overline{z}\overline{w}}
\\
\Theta_{z\overline{z}} & \Theta_{z\overline{z}\overline{w}}
\end{array}
\!\!\right\vert
-
\Theta_{\overline{w}}\Theta_{\overline{w}}
\left\vert\!\!
\begin{array}{cc}
\Theta_{\overline{z}} & \Theta_{\overline{z}\overline{z}}
\\
\Theta_{z\overline{z}} & \Theta_{z\overline{z}\overline{z}}
\end{array}
\!\!\right\vert
\bigg)
\bigg\}.
\endaligned
\]
\end{theorem}

In fact, $\pmb{ \Theta_{\overline{ w}}}$ {\em also} enters denominator,
but erases in the equation `$=0$'.

\subsection{Theorema Egregium of Gauss}

Briefly, here is a second illustration.
On an embedded surface $S^2 \subset \R^3 \ni (x, y, z)$, 
consider local curvilinear bidimensional coordinates 
$(u, v)$:

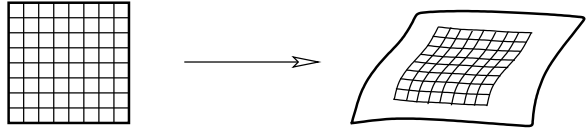
\begin{center}
\input u-v-x-y-z.pstex_t
\end{center}

\noindent
through parametric equations
\[
x
=
x(u,v),
\ \ \ \ \ \ \ \ \ \ \ 
y
=
y(u,v),
\ \ \ \ \ \ \ \ \ \ \ 
z
=
z(u,v).
\ \ \ \ \ \ \ \ \ \ \ 
\]
The flat Pythagorean metric $dx^2 + dy^2 + dz^2$
on $\R^3$ induces on $S^2$:
\[
\aligned
ds^2
=
&\,
\big\vert\!\!\big\vert
(du,dv)
\big\vert\!\!\big\vert^2
=
E\,du^2
+
2F\,dudv
+
G\,dv^2,
\\
\text{\rm with:}\ \ \ \
E
:=
&\,
x_u^2+y_u^2+z_u^2,
\ \ \ \ \ \ \ \ \ \ \ 
F
:=
x_ux_v+y_uy_v+z_uz_v,
\ \ \ \ \ \ \ \ \ \ \ 
G
:=
x_v^2+y_v^2+z_v^2.
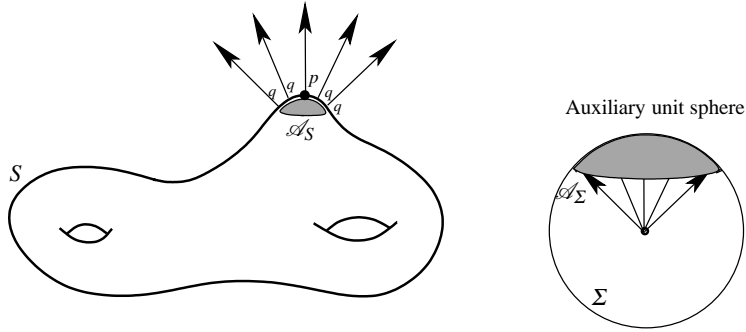
\endaligned
\]

\begin{center}
\input application-de-Gauss.pstex_t
\end{center}

The {\sl Gaussian curvature}
of $S$ at one of its points $p$ is:
\[
{\sf Curvature}(p)
:=
\lim_{\mathcal{A}_S\,\to\,p}\!\!
\frac{\text{ \rm 
area of the region}\
\mathcal{A}_\Sigma\
\text{\rm on the auxiliary unit sphere}
}{
\text{\rm area of the region}\ 
\mathcal{A}_S\ \text{\rm on the surface}}.
\]
When $S$ is graphed as $z = \varphi ( x, y)$, a first formula is:
\[
{\sf Curvature}
=
\frac{
\varphi_{xx}\,\varphi_{yy}-\varphi_{xy}\,\varphi_{xy}}{
1+\varphi_x^2+\varphi_y^2}.
\]
A splendid computation by Gauss provided its intrinsic meaning:
\[
\footnotesize
\aligned
&
{\sf Curvature}
= 
\frac{1}{\,\pmb{(EG-F^2)^2}}
\left\{ E \left[
\frac{ \partial E}{\partial v} \cdot
\frac{ \partial G}{\partial v} - 2\, 
\frac{ \partial F}{\partial u} \cdot 
\frac{ \partial G}{\partial v} +
\left(
\frac{\partial G}{\partial u}
\right)^2
\right] + \right. 
\\
& 
\ \ \ \ \ \ \ \ \ \ \ \ \ \ \ \ \ \ \ \ \
+ \left.
F \left[
\frac{\partial E}{\partial u} \cdot
\frac{\partial G}{\partial v}-
\frac{\partial E}{\partial v} \cdot
\frac{\partial G}{\partial u} -2\, 
\frac{\partial E}{\partial v} \cdot
\frac{\partial F}{\partial v}+ 4\, 
\frac{\partial F}{\partial u} \cdot
\frac{\partial F}{\partial v} - 2\, 
\frac{\partial F}{\partial u} \cdot
\frac{\partial G}{\partial u}
\right]
\right. 
+
\\
& 
\ \ \ \ \ \ \ \ \ \ \ \ \ \ \ \ \ \ \ \ \
+ \left.
G \left[
\frac{ \partial E}{\partial u} \cdot
\frac{ \partial G}{\partial u} - 2\, 
\frac{ \partial E}{\partial u} \cdot 
\frac{ \partial F}{\partial v} +
\left(
\frac{\partial E}{\partial v}
\right)^2
\right]
\right. 
- 
\\
& 
\ \ \ \ \ \ \ \ \ \ \ \ \ \ \ \ \ \ \ \ \
- 
\left.
2\, \left( EG-F^2 \right) \left[
\frac{\partial^2 E}{\partial v^2} -2\,
\frac{\partial^2 F}{\partial u \partial v } + 
\frac{\partial^2 G}{\partial u^2}
\right]
\right\},
\endaligned
\]
and, in the denominator appears $E\,G - F^2$ which is $> 0$ in any
metric.

\subsection{Propagation of sphericity across Levi degenerate points}

Here is one application of {\em explicit rational expressions} as
above. Developing Pinchuk's techniques of extension along Segre
varieties (\cite{ Pinchuk-1975}), Kossovskiy-Shafikov (\cite{
Kossovskiy-Shafikov-2012}) showed that local biholomorphic
equivalence to a model {\sl $(k, n-k)$-pseudo-sphere:}
\[
w
=
\overline{w}
+
2i\big(
-\,z_1\overline{z}_1
-\cdots-
z_k\overline{z}_k
+
z_{k+1}\overline{z}_{k+1}
+\cdots+
z_n\overline{z}_n
\big),
\] 
propagates on any connected real analytic hypersurface $M \subset \C^{
n+1}$ which is Levi nondegenerate outside some $n$-dimensional
complex hypersurface $\Sigma \subset M$. A more general statement,
not known with Segre varieties techniques, is:

\begin{theorem}
\label{propagation-sphericity}
{\rm (\cite{ Merker-2013-transfer})}
If a connected $\mathcal{ C}^\omega$
hypersurface $M \subset \C^{ n + 1}$
is locally biholomorphic, in a neighborhood of one of its
points $p$, to some $(k, n-k)$-pseudo-sphere, then
locally at every other Levi nondegenerate point $q \in M
\big\backslash \Sigma_{\sf LD}$, this hypersurface $M$ is also locally
biholomorphic to some Heisenberg $(l, n - l)$-pseudo-sphere, with
(\cite{ Kossovskiy-Shafikov-2012}),
possibly $l \neq k$.
\end{theorem}

The proof, suggested by Beloshapka, consists first for $n = 1$ in
observing that after expansion, Theorem~\ref{explicit-sphericity}
characterizes sphericity as:
\[
0
\equiv
\frac{
{\sf polynomial}
\big(
\big(
\Theta_{z^j\overline{z}^k\overline{w}^l}
\big)_{1\leqslant j+k+l\leqslant 6}
\big)}{
\big[\Theta_{\overline{z}}\,\Theta_{z\overline{w}}
-
\Theta_{\overline{w}}\,\Theta_{z\overline{z}}\big]^7},
\]
at every Levi nondegenerate point $(z_p, w_p) \in M$ at which the
denominator is $\neq 0$. But this means that the numerator is $\equiv
0$ near $(z_p, \overline{ z}_p, \overline{ w}_p)$, and at every other
Levi nondegenerate point $(z_q, w_q) \in M$ close to $(z_p, w_p)$, the
numerator is also locally $\equiv 0$ by analytic continuation. 
Small translations of
coordinates are needed; the complete arguments
appear 
in~\cite{ Merker-2013-transfer}.

In dimensions $n \geqslant 2$, the explicit characterization of $(k,
n-k)$-pseudo-sphericity is also {\em rational}. Indeed, 
in local holomorphic coordinates:
\[
t=(z,w)
\,\in\,
\C^n\times\C,
\]
represent similarly a $\mathcal{ C}^\omega$ hypersurface $M^{
2n + 1} \subset \C^{ n+1}$ as:
\[
w 
= 
\Theta\big(z,\,\overline{z},\overline{w})
=
\Theta\big(z,\,\overline{t}\big).
\]
Introduce the Levi form Jacobian-like determinant: 
\[
\Delta
:=
\left\vert
\begin{array}{cccc}
\Theta_{\overline{z}_1} & \cdots & \Theta_{\overline{z}_n} 
& \Theta_{\overline{w}}
\\
\Theta_{z_1\overline{z}_1} & \cdots & \Theta_{z_1\overline{z}_n} 
& \Theta_{z_1\overline{w}}
\\
\cdot\cdot & \cdots & \cdot\cdot & \cdot\cdot
\\
\Theta_{z_n\overline{z}_1} & \cdots & \Theta_{z_n\overline{z}_n} 
& \Theta_{z_n\overline{w}}
\end{array}
\right\vert
=
\left\vert
\begin{array}{cccc}
\Theta_{\overline{t}_1} & \cdots & \Theta_{\overline{t}_n} 
& \Theta_{\overline{t}_{n+1}}
\\
\Theta_{z_1\overline{t}_1} & \cdots & \Theta_{z_1\overline{t}_n} 
& \Theta_{z_1\overline{t}_{n+1}}
\\
\cdot\cdot & \cdots & \cdot\cdot & \cdot\cdot
\\
\Theta_{z_n\overline{t}_1} & \cdots & \Theta_{z_n\overline{t}_n} 
& \Theta_{z_n\overline{t}_{n+1}}
\end{array}
\right\vert.
\]
It is nonzero at a point $p = (z_p, \overline{ t}_p)$ if and only
if $M$ is Levi nondegenerate at $p$.
For any index $\mu \in \{ 1, \dots, n, n+1\}$ and for any index $\ell
\in \{ 1, \dots, n\}$, let also $\Delta_{ [ 0_{ 1+\ell}]}^\mu$ denote
the same determinant, but with its $\mu$-th column replaced by the
transpose of the line $(0 \cdots 1 \cdots 0)$ with $1$ at the
$(1+\ell)$-th place, and $0$ elsewhere, its other columns being
untouched.
Similarly, for any indices $\mu, \nu,
\tau \in \{ 1, \dots, n, n+1\}$, denote by $\Delta_{ [ \overline{
t}^\mu \overline{ t}^\nu]}^\tau$ the same determinant as $\Delta$, but
with only its $\tau$-th column replaced by the transpose of the line:
\[
\big(
\Theta_{\overline{t}^\mu\overline{t}^\nu}\ \ 
\Theta_{z_1\overline{t}^\mu\overline{t}^\nu}\ \
\cdots\ \
\Theta_{z_n\overline{t}^\mu\overline{t}^\nu}
\big),
\]
other columns being again untouched. All these determinants $\Delta$,
$\Delta_{ [ 0_{ 1+\ell}]}^\mu$, $\Delta_{ [ \overline{ t}^\mu
\overline{ t}^\nu]}^\tau$ depend 
upon the third-order jet $J_{ z, \overline{
z}, \overline{ w}}^3 \Theta$.

\begin{theorem}
\label{k-n-k-pseudo-spherical}
\text{\rm (\cite{ Merker-Hachtroudi, Merker-2013-transfer})}
A $\mathcal{ C}^\omega$ hypersurface
$M \subset \C^{ n+1}$ with $n \geqslant 2$ which is Levi
nondegenerate at some point $p = (z_p, \overline{ z}_p,
\overline{ w}_p)$ is $(k, n-k)$-pseudo-spherical
at $p$ if and only if, identically for $(z, \overline{ z}, 
\overline{ w})$ near $(z_p, \overline{ z}_p,
\overline{ w}_p)$:
\[
\footnotesize
\aligned
0
&
\equiv
\frac{1}{\Delta^3}
\bigg[\,
\sum_{\mu=1}^{n+1}\,\sum_{\nu=1}^{n+1}
\bigg[
\Delta_{[0_{1+\ell_1}]}^\mu
\cdot
\Delta_{[0_{1+\ell_2}]}^\nu
\bigg\{
\Delta
\cdot
\frac{\partial^4\Theta}{
\partial z_{k_1}\partial z_{k_2}
\partial\overline{t}_\mu\partial\overline{t}_\nu}
-
\sum_{\tau=1}^{n+1}\,
\Delta_{[\overline{t}^\mu\overline{t}^\nu]}^\tau
\cdot
\frac{\partial^3\Theta}{
\partial z_{k_1}\partial z_{k_2}\partial\overline{t}^\tau}
\bigg\}
-
\\
&
-
{\textstyle{\frac{\delta_{k_1,\ell_1}}{n+2}}}\,
\sum_{\ell'=1}^n\,
\Delta_{[0_{1+\ell'}]}^\mu
\cdot
\Delta_{[0_{1+\ell_2}]}^\nu
\bigg\{
\Delta
\cdot
\frac{\partial^4\Theta}{
\partial z_{\ell'}\partial z_{k_2}
\partial\overline{t}_\mu\partial\overline{t}_\nu}
-
\sum_{\tau=1}^{n+1}\,
\Delta_{[\overline{t}^\mu\overline{t}^\nu]}^\tau
\cdot
\frac{\partial^3\Theta}{
\partial z_{\ell'}\partial z_{k_2}\partial\overline{t}^\tau}
\bigg\}
-
\\
&
-
{\textstyle{\frac{\delta_{k_1,\ell_2}}{n+2}}}\,
\sum_{\ell'=1}^n\,
\Delta_{[0_{1+\ell_1}]}^\mu
\cdot
\Delta_{[0_{1+\ell'}]}^\nu
\bigg\{
\Delta
\cdot
\frac{\partial^4\Theta}{
\partial z_{\ell'}\partial z_{k_2}
\partial\overline{t}_\mu\partial\overline{t}_\nu}
-
\sum_{\tau=1}^{n+1}\,
\Delta_{[\overline{t}^\mu\overline{t}^\nu]}^\tau
\cdot
\frac{\partial^3\Theta}{
\partial z_{\ell'}\partial z_{k_2}\partial\overline{t}^\tau}
\bigg\}
-
\\
&
-
{\textstyle{\frac{\delta_{k_2,\ell_1}}{n+2}}}\,
\sum_{\ell'=1}^n\,
\Delta_{[0_{1+\ell'}]}^\mu
\cdot
\Delta_{[0_{1+\ell_2}]}^\nu
\bigg\{
\Delta
\cdot
\frac{\partial^4\Theta}{
\partial z_{k_1}\partial z_{\ell'}
\partial\overline{t}_\mu\partial\overline{t}_\nu}
-
\sum_{\tau=1}^{n+1}\,
\Delta_{[\overline{t}^\mu\overline{t}^\nu]}^\tau
\cdot
\frac{\partial^3\Theta}{
\partial z_{k_1}\partial z_{\ell'}\partial\overline{t}^\tau}
\bigg\}
-
\\
&
-
{\textstyle{\frac{\delta_{k_2,\ell_2}}{n+2}}}\,
\sum_{\ell'=1}^n\,
\Delta_{[0_{1+\ell_1}]}^\mu
\cdot
\Delta_{[0_{1+\ell'}]}^\nu
\bigg\{
\Delta
\cdot
\frac{\partial^4\Theta}{
\partial z_{k_1}\partial z_{\ell'}
\partial\overline{t}_\mu\partial\overline{t}_\nu}
-
\sum_{\tau=1}^{n+1}\,
\Delta_{[\overline{t}^\mu\overline{t}^\nu]}^\tau
\cdot
\frac{\partial^3\Theta}{
\partial z_{k_1}\partial z_{\ell'}\partial\overline{t}^\tau}
\bigg\}
+
\\
&
\ \ \ \ \
+
{\textstyle{\frac{1}{(n+1)(n+2)}}}
\cdot
\big[
\delta_{k_1,\ell_1}\delta_{k_2,\ell_2}
+
\delta_{k_2,\ell_1}\delta_{k_1,\ell_2}
\big]
\cdot
\\
&
\ \ \ \ \
\cdot
\sum_{\ell'=1}^n\,\sum_{\ell''=1}^n\,
\Delta_{[0_{1+\ell'}]}^\mu
\cdot
\Delta_{[0_{1+\ell''}]}^\nu
\bigg\{
\Delta
\cdot
\frac{\partial^4\Theta}{
\partial z_{\ell'}\partial z_{\ell''}
\partial\overline{t}_\mu\partial\overline{t}_\nu}
-
\sum_{\tau=1}^{n+1}\,
\Delta_{[\overline{t}^\mu\overline{t}^\nu]}^\tau
\cdot
\frac{\partial^3\Theta}{
\partial z_{\ell'}\partial z_{\ell''}\partial\overline{t}^\tau}
\bigg],
\\
&
\ \ \ \ \ \ \ \ \ \ \ \ \ \ \ \ \ \ \ \ \ \ \ \ \ \ \ \ \ \ 
\ \ \ \ \ \ \ \ \ \ \ \ \ 
(1\,\leqslant k_1,\,k_2\,\leqslant\,n;\,\,\,
1\,\leqslant \ell_1,\,\ell_2\,\leqslant\,n).
\endaligned
\] 
\end{theorem}

Then as in the case $n = 1$, propagation of pseudo-sphericity `jumps'
across Levi degenerate points, because above,
the denominator $\Delta^3$ locates Levi nondegenerate points. This
explicit expression is a translation of Hachtroudi's characterization
(\cite{ Hachtroudi-1937}) of equivalence
to $w_{z_{k_1}'z_{k_2}'}'(z')
= 0$ of completely integrable {\sc pde}
systems:
\[
w_{z_{k_1}z_{k_2}}(z)
=
\Phi_{k_1,k_2}
\big(
z,\,w(z),\,w_{z_1}(z),\dots,w_{z_n}(z)
\big)
\ \ \ \ \ \ \ \ \ \ \ \ \ \ \ \ \ \
{\scriptstyle{(1\,\leqslant\,k_1,\,\,k_2\,\leqslant\,n)}}.
\]

\medskip\noindent{\bf Question still open.} 
{\sl Compute explicity the Chern-Moser-Webster $1$-forms 
and curvatures 
(\cite{ Chern-Moser-1974, Webster-1978})
in terms of a local graphing function
for a Levi nondegenerate $M^{2n +1} \subset \C^{n+1}$
(rigid and tube cases are treated in~\cite{ Isaev-2011}).}

\medskip

This would, in particular, provide an alternative proof of
Theorem~\ref{k-n-k-pseudo-spherical}.

\subsection{Zariski-generic $\mathcal{ C}^\omega$ CR manifolds of dimension 
$\leqslant 5$}

Coming back to $M^{ 2n + c} \subset \C^{n+c}$ of dimension $2n + c
\leqslant 5$, and calling {\sl Zariski-open} any complement $M
\backslash \Sigma$ of some {\em proper} real analytic subset $\Sigma
\subsetneqq M$, treat at first the:

\medskip\noindent{\bf Problem 2.} 
(Accessible subquestion of Problem 1)
{\sl Set up all possible initial geometries of connected $\mathcal{ C}^\omega$
CR-generic submanifolds $M^{ 2n+c} \subset \C^{ n+c}$ at
Zariski-generic points.}

\medskip

For general $M^{ 2n + c} \subset \C^{n+c}$, recall that
the fundamental invariant bundle is:
\[
T^{1,0}M 
:= 
\big\{ 
X-{\scriptstyle{\sqrt{-1}}}\,J(X) 
\colon 
X\in TM\cap J(TM)
\big\}.
\]

\begin{lemma}
\label{L-explicit} 
{\rm (\cite{Merker-5-CR})}
If a CR-generic $M^{2n+c} \subset \C^{n+c}$ is locally graphed as:
\[
v_j
=
\varphi_j
\big(x_1,\dots,x_n,y_1,\dots,y_n,u_1,\dots,u_c\big)
\ \ \ \ \ \ \ \ \ \ \ \ \
{\scriptstyle{(1\,\leqslant\,j\,\leqslant\,c)}},
\]
a local frame $\big\{ \mathcal{L}_1,\dots,\mathcal{L}_n \big\}$ for
$T^{ 1, 0}M$ consists of the $n$ vector fields:
\[
\mathcal{L}_i
=
\frac{\partial}{\partial z_k}
+
A_i^1\big(x_{\scriptscriptstyle{\bullet}},y_{\scriptscriptstyle{\bullet}},u_{\scriptscriptstyle{\bullet}}\big)\,
\frac{\partial}{\partial u_1}
+\cdots\cdots+
A_i^c\big(x_{\scriptscriptstyle{\bullet}},y_{\scriptscriptstyle{\bullet}},u_{\scriptscriptstyle{\bullet}}\big)\,
\frac{\partial}{\partial u_c}
\ \ \ \ \ \ \ \ \ \ \ \ \
{\scriptstyle{(1\,\leqslant\,i\,\leqslant\,n)}},
\]
having {\em rational} coefficient-functions:
\[
\footnotesize
\aligned
\!\!\!\!\!\!\!\!\!\!\!\!\!\!\!\!\!\!\!\!\!
A_i^1
\,=\,
\frac{
\left\vert\!\!
\begin{array}{cccc}
-\,\varphi_{1,z_i} & \varphi_{1,u_2} & \cdots & \varphi_{1,u_c}
\\
-\,\varphi_{2,z_i} & {\scriptstyle{\sqrt{-1}}}+\varphi_{2,u_2} & \cdots & \varphi_{2,u_c}
\\
\vdots & \vdots & \ddots & \vdots
\\
-\,\varphi_{c,z_i} & \varphi_{c,u_2} & \cdots & {\scriptstyle{\sqrt{-1}}}+\varphi_{c,u_c}
\end{array}
\!\!\right\vert
}{
\left\vert\!\!
\begin{array}{ccc}
{\scriptstyle{\sqrt{-1}}}+\varphi_{1,u_1} & \cdots & \varphi_{1,u_c}
\\
\varphi_{2,u_1} & \cdots & \varphi_{2,u_c}
\\
\vdots & \ddots & \vdots
\\
\varphi_{c,u_1} & \cdots & {\scriptstyle{\sqrt{-1}}}+\varphi_{c,u_c}
\end{array}
\!\!\right\vert
},\,\,\,
\dots\dots\dots,\,\,\,
A_i^c
\,=\,
\frac{
\left\vert\!\!
\begin{array}{cccc}
{\scriptstyle{\sqrt{-1}}}+\varphi_{1,u_1} & \cdots & -\,\varphi_{1,z_i}
\\
\varphi_{2,u_1} & \cdots & -\,\varphi_{2,z_i}
\\
\vdots & \ddots & \vdots
\\
\varphi_{c,u_1} & \cdots & -\,\varphi_{c,z_i}
\end{array}
\!\!\right\vert
}{
\left\vert\!\!
\begin{array}{ccc}
{\scriptstyle{\sqrt{-1}}}+\varphi_{1,u_1} & \cdots & \varphi_{1,u_c}
\\
\varphi_{2,u_1} & \cdots & \varphi_{2,u_c}
\\
\vdots & \ddots & \vdots
\\
\varphi_{c,u_1} & \cdots & {\scriptstyle{\sqrt{-1}}}+\varphi_{c,u_c}
\end{array}
\!\!\right\vert
}.
\endaligned
\]
\end{lemma}

Nonvanishing of the denominator is equivalent to CR-genericity of $M$.

\smallskip

Here is how $M$ transfers through local biholomorphisms.

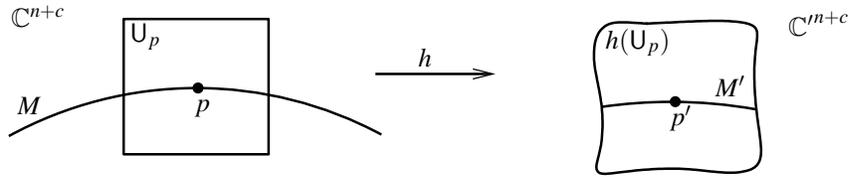
\begin{center}
\input M-h-polydiscs.pstex_t
\end{center}

\begin{lemma}
{\rm (\cite{Merker-5-CR})}
Given a connected $\mathcal{ C}^\omega$ CR-generic submanifold
$M^{2n+c} \subset \C^{n+c}$ and a local biholomorphism between open
subsets:
\[
h\colon\ \ \
{\sf U}_p
\overset{\sim}{\,\longrightarrow\,}
h({\sf U}_p)
=
{\sf U}_{p'}'
\subset
{\C'}^{n+c},
\]
with $p \in M$, $p' = h (p)$, setting:
\[
M'
:=
h(M)
\,\subset\,
{\C'}^{n+c}
\ \ \ \ \ \ \ \ \ \ \ \ \
{\scriptstyle{(c\,=\,{\sf codim}\,M',\,\,\,
n\,=\,{\sf CRdim}\,M')}},
\]
then for any two local frames:
\[
\big\{\mathcal{L}_1,\dots,\mathcal{L}_n\big\}
\ \ \ \ \
\text{for}\ \
T^{1,0}M
\ \ \ \ \ \ \ \ \ \ \ \ \
\text{and}
\ \ \ \ \ \ \ \ \ \ \ \ \
\big\{\mathcal{L}_1',\dots,\mathcal{L}_n'\big\}
\ \ \ \ \
\text{for}\ \
T^{1,0}M',
\]
there exist
uniquely defined $\mathcal{ C}^\omega$ local
coefficient-functions:
\[
a_{i_1i_2}'\colon\ \ \
M'
\,\longrightarrow\,
\C
\ \ \ \ \ \ \ \ \ \ \ \ \
{\scriptstyle{(1\,\leqslant\,i_1,\,i_2\,\leqslant\,n)}},
\]
satisfying:
\[
\aligned
h_*\big(\mathcal{L}_1\big)
&
=
a_{11}'\,\mathcal{L}_1'
+\cdots+
a_{n1}'\,\mathcal{L}_n',
\\
\cdots\cdots\cdot
&
\cdots\cdots\cdots\cdots\cdots\cdots\cdots\cdots
\\
h_*\big(\mathcal{L}_n\big)
&
=
a_{1n}'\,\mathcal{L}_1'
+\cdots+
a_{nn}'\,\mathcal{L}_n'.
\endaligned
\]
\end{lemma}

\begin{definition}
Taking any local $1$-form $\rho_0 \colon TM \to \R$ whose
extension to $\C \otimes_\R TM$ satisfies:
\[
T^{1,0}M\oplus T^{0,1}M
=
\big\{\rho_0=0\big\},
\]
the Hermitian matrix of the {\sl Levi form} of $M$ at various points 
$p \in M$ is:
\[
\left(\!
\begin{array}{ccc}
\rho_0\big({\scriptstyle{\sqrt{-1}}}\,\big[\mathcal{L}_1,\overline{\mathcal{L}}_1\big]\big)
& \cdots &
\rho_0\big({\scriptstyle{\sqrt{-1}}}\,\big[\mathcal{L}_n,\overline{\mathcal{L}}_1\big]\big)
\\
\vdots & \ddots & \vdots
\\
\rho_0\big({\scriptstyle{\sqrt{-1}}}\,\big[\mathcal{L}_1,\overline{\mathcal{L}}_n\big]\big)
& \cdots &
\rho_0\big({\scriptstyle{\sqrt{-1}}}\,\big[\mathcal{L}_n,\overline{\mathcal{L}}_n\big]\big)
\end{array}
\!\right)
(p),
\]
the extra factor ${\scriptstyle{\sqrt{-1}}}$ being present in order
to counterbalance the change of sign:
\[
\overline{\big[\mathcal{L}_j,\,\overline{\mathcal{L}}_k\big]}
=
-\,\big[\mathcal{L}_k,\,\overline{\mathcal{L}}_j\big].
\] 
\end{definition}

As an application, show the invariance of Levi nondegeneracy. For
$M^3 \subset \C^2$ equivalent to ${M'}^3 \subset
{\C'}^2$, whence $n = n' = 1$, introduce local vector field
generators:
\[
\mathcal{L}
\ \ 
\text{\rm for}\ \
T^{1,0}M
\ \ \ \ \ \ \ \ \ \ \ \ \
\text{\rm and}
\ \ \ \ \ \ \ \ \ \ \ \ \
\mathcal{L}'
\ \ 
\text{\rm for}\ \
T^{1,0}M'.
\]

\begin{lemma}
\label{3-rank}
{\em At every point $q \in M$ near $p$:}
\[
\aligned
{\bf 3}
&
=
{\sf rank}_\C
\Big(
\mathcal{L}\big\vert_q,\,\,
\overline{\mathcal{L}}\big\vert_q,\,\,
\big[\mathcal{L},\overline{\mathcal{L}}\big]\big\vert_q
\Big)
\\
&
\ \ \ \ \
\Updownarrow
\\
{\bf 3}
&
=
{\sf rank}_\C
\Big(
\mathcal{L}'\big\vert_{h(q)},\,\,
\overline{\mathcal{L}}'\big\vert_{h(q)},\,\,
\big[\mathcal{L}',\overline{\mathcal{L}}'\big]\big\vert_{h(q)}
\Big).
\endaligned
\]
\end{lemma}

\begin{proof}
\smartqed
By what precedes, there exists a function
$a' \colon M'\longrightarrow\C \backslash\{ 0\}$ with:
\[
h_*(\mathcal{L})
=
a'\,\mathcal{L}'
\ \ \ \ \ \ \ \ \ \ \ \ \
\text{\rm and}
\ \ \ \ \ \ \ \ \ \ \ \ \
h_*\big(\overline{\mathcal{L}}\big)
=
\overline{a}'\,\overline{\mathcal{L}}'.
\]
Consequently:
\[
\aligned
h_*\big(\big[\mathcal{L},\overline{\mathcal{L}}\big]\big)
&
=
\big[h_*(\mathcal{L}),\,h_*\big(
\overline{\mathcal{L}}\big)\big]
\\
&
=
\big[a'\mathcal{L}',\,\overline{a}'\overline{\mathcal{L}}'\big]
\\
&
=
a'\overline{a}'\,
\big[\mathcal{L}',\,\overline{\mathcal{L}}'\big]
+
a'\,\mathcal{L}'(\overline{a}')\cdot\overline{\mathcal{L}}'
-
\overline{a}'\,\overline{\mathcal{L}}'(a')\cdot\mathcal{L}'.
\endaligned
\] 
Dropping the mention of $h_*$, because the {\sl change of frame} matrix:
\[
\left(\!
\begin{array}{c}
\mathcal{L}
\\
\overline{\mathcal{L}}
\\
\big[\mathcal{L},\overline{\mathcal{L}}\big]
\end{array}
\!\right)
\,=\,
\left(\!
\begin{array}{ccc}
a' & 0 & 0
\\
0 & \overline{a}' & 0
\\
\ast & \ast & a'\overline{a}'
\end{array}
\!\right)
\left(\!
\begin{array}{c}
\mathcal{L}'
\\
\overline{\mathcal{L}}'
\\
\big[\mathcal{L}',\overline{\mathcal{L}}'\big]
\end{array}
\!\right)
\]
is visibly of rank $3$, the result follows.
\qed
\end{proof}

Since $T^{1, 0} M$ and $T^{0, 1} M$ are Frobenius-involutive
(\cite{Merker-5-CR}), only iterated Lie brackets between $T^{1, 0}M$
and $T^{0, 1}M$ are nontrivial. For a $\mathcal{ C}^\omega$ connected
$M^{2n+c} \subset \C^{n+c}$ with $T^{1, 0}M$ having local generators
$\mathcal{L}_1, \dots, \mathcal{L}_n$, set:
\[
\aligned
\mathbb{L}_{\mathcal{L},\overline{\mathcal{L}}}^1
&
:=
\mathcal{C}^\omega\text{\rm -linear combinations of}\,\,
\mathcal{L}_1,\dots,\mathcal{L}_n,\,\,
\overline{\mathcal{L}}_1,\dots,\overline{\mathcal{L}}_n,
\\
\cdots\cdots
&
\cdots\cdots\cdots\cdots\cdots\cdots\cdots\cdots\cdots\cdots\cdots\cdots
\cdots\cdots\cdots\cdots\cdots\cdots
\\
\mathbb{L}_{\mathcal{L},\overline{\mathcal{L}}}^{\nu+1}
&
:=
\mathcal{C}^\omega\text{\rm -linear combinations of vector fields}\,\,
\mathcal{M}^\nu\in\mathbb{L}_{\mathcal{L},\overline{\mathcal{L}}}^\nu
\\
&
\ \ \ \ \ \ \ \ \ \ \ \ \ 
\text{\rm and of brackets}\,\,
\big[\mathcal{L}_k,\mathcal{M}^\nu\big],\,\,
\big[\overline{\mathcal{L}}_k,\mathcal{M}^\nu\big].
\endaligned
\]
Set:
\[
\mathbb{L}_{\mathcal{L},\overline{\mathcal{L}}}^{\sf Lie}
:=
\bigcup_{\nu\geqslant 1}\,
\mathbb{L}_{\mathcal{L},\overline{\mathcal{L}}}^\nu.
\]

By real analyticity (\cite{ Merker-Porten-2006}), there exists an
integer $c_M$ with $0\leqslant c_M\leqslant c$ and a proper real
analytic subset $\Sigma\, \subsetneqq\, M$ such that at every point $q
\in M\backslash\Sigma$:
\[
{\sf dim}_\C\,
\Big(
\mathbb{L}_{\mathcal{L},\overline{\mathcal{L}}}^{\sf Lie}(q)
\Big)
=
2n+c_M.
\]

\begin{theorem}
\text{\rm (Known, \cite{ Merker-Porten-2006})}
Every point $q\,\in\,M\backslash\Sigma$
has a small open neighborhood
${\sf U}_q
\subset
\C^{n+c}$
in which:
\[
M^{2n+c}
\cong\,
\underline{M}^{2n+c}
\]
biholomorphically, with a CR and $\mathcal{ C}^\omega$:
\[
\underline{M}^{2n+c}
\subset
\C^{n+c_M}
\times
\R^{c-c_M}.
\]
\end{theorem}

The case $c_M \leqslant c - 1$ must hence be considered as {\sl
degeneration}, excluded in classification of initial geometries at
Zariski-generic points.

A well known fact is that, at a Zariski-generic point, a connected
$\mathcal{ C}^\omega$ hypersurface $M^3 \subset \C^2$ is either $\cong
\C \times \R$ or is Levi nondegenerate:
\[
3
=
{\sf rank}_\C
\big(
T^{1,0}M,\,\,
T^{0,1}M,\,\,
\big[T^{1,0}M,\,T^{0,1}M\big]
\big).
\]

\begin{theorem}
{\rm (\cite{ Merker-5-CR})}
Excluding degenerate CR manifolds, there are precisely {\bf six}
general classes of nondegenerate connected
$M^{ 2n +c} \subset \C^{ n+c}$ having dimension:
\[
2n+c
\,\leqslant\,
{\bf 5},
\]
hence having CR dimension $n = 1$ or $n = 2$, namely if:
\[
\big\{
\mathcal{L}
\big\}
\ \ \ \ \ \ \ \ \ \ \ \ \ \
\text{\rm or}
\ \ \ \ \ \ \ \ \ \ \ \ \ \
\big\{
\mathcal{L}_1,\mathcal{L}_2
\big\},
\]
denotes any local frame for $T^{ 1, 0} M$:

\smallskip\noindent$\bullet$\,
{\bf General Class $\text{\sf I}$:} Hypersurfaces
$M^3\subset\C^2$ with $\big\{\mathcal{L},\, \overline{\mathcal{L}},
\,\, \big[\mathcal{L}, \overline{ \mathcal{L}}\big]\big\}$
constituting a frame for $\C \otimes_\R TM$, with:
\[
{\text{\small\sf Model I:}}
\ \ \ \ \ \ \ \ \ 
v
=
z\overline{z},
\]

\smallskip\noindent$\bullet$\,
{\bf General Class $\text{\sf II}$:} CR-generic $M^4\subset\C^3$
with $\big\{ \mathcal{L},\, \overline{ \mathcal{L}},\,\,
\big[\mathcal{L}, \overline{ \mathcal{L}} \big],\,\,
\big[\mathcal{L},\, \big[\mathcal{L}, \overline{\mathcal{L}} \big]\big]
\big\}$ constituting a frame for $\C\otimes_\R TM$, with:
\[
{\text{\small\sf Model II:}}
\ \ \ \ \ \ \ \ \ 
v_1
=
z\overline{z},
\ \ \ \ \ \ \ \ \ 
v_2
=
z^2\overline{z}+z\overline{z}^2,
\]

\smallskip\noindent$\bullet$\,
{\bf General Class $\text{\sf III}_{\text{\sf 1}}$:} CR-generic
$M^5\subset\C^4$ with
$\big\{\mathcal{L},\, \overline{\mathcal{L}},\,\,
\big[\mathcal{L}, \overline{\mathcal{L}}\big],\,\,
\big[\mathcal{L},\,
\big[\mathcal{L}, \overline{\mathcal{L}}\big]\big],\,\,
\big[ \overline{\mathcal{L}}$,
$\big[\mathcal{L}, \overline{\mathcal{L}}\big]\big]
\big\}$ constituting a frame for
$\C\otimes_\R TM$, with:
\[
{\text{\small\sf Model 
$\text{\sf III}_{\text{\sf 1}}$:}}
\ \ \ \ \ \ \ \ \ 
v_1
=
z\overline{z},
\ \ \ \ \ \ \ \ \
v_2
=
z^2\overline{z}+z\overline{z}^2,
\ \ \ \ \ \ \ \ \
v_3
=
{\scriptstyle{\sqrt{-1}}}\,\big(z^2\overline{z}-z\overline{z}^2\big),
\]

\smallskip\noindent$\bullet$\,
{\bf General Class $\text{\sf III}_{\text{\sf 2}}$:} CR generic
$M^5\subset\C^4$ with
$\big\{\mathcal{L},\, \overline{\mathcal{L}},\,\,
\big[\mathcal{L},\overline{\mathcal{L}} \big],\,\,
\big[\mathcal{L},\,
\big[\mathcal{L}, \overline{\mathcal{L}} \big]\big],\,\,
\big[\mathcal{L}$, 
$\big[\mathcal{L},\,
\big[\mathcal{L}, \overline{\mathcal{L}} \big]\big]\big]
\big\}$ constituting a frame for
$\C\otimes_\R TM$,
while
${\bf 4}
=
{\sf rank}_\C
\Big(\mathcal{L},\overline{\mathcal{L}},\,
\big[\mathcal{L},\overline{\mathcal{L}}\big]$,
$\big[\mathcal{L},\,\big[\mathcal{L},\overline{\mathcal{L}}\big]\big],\,\,
\big[\overline{\mathcal{L}},\,
\big[\mathcal{L},\overline{\mathcal{L}}\big]\big]\Big)$, with:
\[
{\text{\small\sf Model
$\text{\sf III}_{\text\sf 2}$:}}
\ \ \ \ \ \ \ \ \ 
v_1
=
z\overline{z},
\ \ \ \ \ \ \ \ \ 
v_2
=
z^2\overline{z}+z\overline{z}^2,
\ \ \ \ \ \ \ \ \ 
v_3
=
2\,z^3\overline{z}+2\,z\overline{z}^3
+
3\,z^2\overline{z}^2,
\]

\smallskip\noindent$\bullet$\,
{\bf General Class $\text{\sf IV}_{\text{\sf 1}}$:} Hypersurfaces
$M^5\subset\C^3$ with
$\big\{\mathcal{L}_1,\, \mathcal{L}_2,\,
\overline{\mathcal{L}}_1,\, \overline{\mathcal{L}}_2,\,\,
\big[\mathcal{L}_1, \overline{\mathcal{L}}_1\big]
\big\}$
constituting a frame for
$\C\otimes_\R TM$, and with the Levi-Form of $M$
being of rank $2$ at every point
$p \in M$, with:
\[
{\text{\small\sf Model(s) 
$\text{\sf IV}_{\text{\sf 1}}$:}}
\ \ \ \ \ \ \ \ \ 
v
=
z_1\overline{z}_1
\pm
z_2\overline{z}_2,
\]

\smallskip\noindent$\bullet$\,
{\bf General Class $\text{\sf IV}_{\text{\sf 2}}$:} Hypersurfaces
$M^5\subset\C^3$ with
$\big\{\mathcal{L}_1,\, \mathcal{L}_2,\,
\overline{\mathcal{L}}_1,\, \overline{\mathcal{L}}_2,\,\,
\big[\mathcal{L}_1, \overline{\mathcal{L}}_1\big]
\big\}$ constituting a frame for
$\C\otimes_\R TM$, with the Levi-Form
being of rank $1$ at every point $p\in M$
while the Freeman-Form (defined below)
is nondegenerate at every point, with:
\[
{\text{\small\sf Model 
$\text{\sf IV}_{\text{\sf 2}}$:}}
\ \ \ \ \ \ \ \ \ 
v
=
\frac{z_1\overline{z}_1+
\frac{1}{2}\,z_1z_1\overline{z}_2
+
\frac{1}{2}\,z_2\overline{z}_1\overline{z}_1}{
1-z_2\overline{z}_2}.
\]
\end{theorem}

The models $\text{\sf II}$,
$\text{\sf III}_{\text{\sf 1}}$ appear in the works of Beloshapka
(\cite{Beloshapka-1998, Beloshapka-2002}) which exhibit a wealth
of higher dimensional models widening the biholomorphic
equivalence problem.
Before proceeding, explain (only) how the General Classes
$\text{\sf IV}_{\text{\sf 1}}$ and
$\text{\sf IV}_{\text{\sf 2}}$ occur.

\subsection{Concept of Freeman form}
If $M^5 \subset \C^3$ is connected
$\mathcal{ C}^\omega$ and belongs to Class 
$\text{\sf IV}_{\text{\sf 2}}$, so that:
\[
1
=
{\sf rank}_\C\big(
{\sf Levi}\text{-}{\sf Form}^M(p)
\big)
\ \ \ \ \ \ \ \ \ \ \ \ \
{\scriptstyle{(\forall\,p\,\in\,M)}},
\] 
then there exists a unique rank $1$ complex vector subbundle:
\[
K^{1,0}M
\subset
T^{1,0}M
\]
such that, at every point $p \in M$:
\[
K_p^{1,0}M
\,\ni\,
\mathcal{K}_p
\,\,\Longleftrightarrow\,\,
\mathcal{K}_p
\,\in\,
{\sf Kernel}
\big(
{\sf Levi}\text{-}{\sf Form}^M(p)
\big).
\]
Local trivializations of $K^{ 1, 0} M$ match up on intersections of balls
(\cite{ Merker-5-CR}).

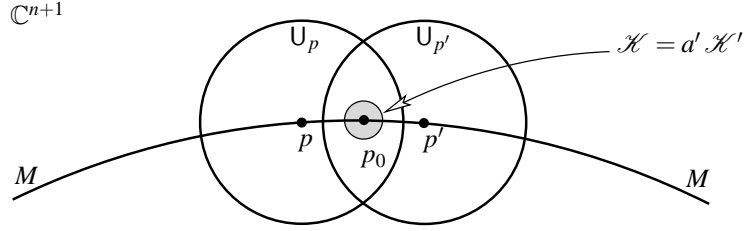
\begin{center}
\input 3-balls.pstex_t
\end{center}

Furthermore, three known involutiveness conditions hold
(\cite{ Merker-5-CR}):
\begin{equation}
\label{K-involutiveness}
\aligned
\big[K^{1,0}M,\,K^{1,0}M\big]
&
\,\subset\,
K^{1,0}M,
\\
\big[K^{0,1}M,\,K^{0,1}M\big]
&
\,\subset\,
K^{0,1}M,
\\
\big[K^{1,0}M,\,K^{0,1}M\big]
&
\,\subset\,
K^{1,0}M
\oplus
K^{0,1}M.
\endaligned
\end{equation}

In local coordinates, $M^5 \subset \C^3$ is graphed as:
\[
v
=
\varphi\big(x_1,x_2,y_1,y_2,u\big),
\]
and two local generators of $T^{1, 0}M$ with their conjugates are:
\[
\aligned
\mathcal{L}_1
&
=
\frac{\partial}{\partial z_1}
+
A_1\,\frac{\partial}{\partial u},
\ \ \ \ \ \ \ \ \ \ \ \ \ \ \ \ \ \ \ \ \ \ \ \ 
\overline{\mathcal{L}}_1
=
\frac{\partial}{\partial\overline{z}_1}
+
\overline{A_1}\,\frac{\partial}{\partial u},
\\
\mathcal{L}_2
&
=
\frac{\partial}{\partial z_2}
+
A_2\,\frac{\partial}{\partial u},
\ \ \ \ \ \ \ \ \ \ \ \ \ \ \ \ \ \ \ \ \ \ \ \ 
\overline{\mathcal{L}}_2
=
\frac{\partial}{\partial\overline{z}_2}
+
\overline{A_2}\,\frac{\partial}{\partial u},
\endaligned
\]
with:
\[
\aligned
A_1
:=
-\,
\frac{\varphi_{z_1}}{{\scriptstyle{\sqrt{-1}}}+\varphi_u},
\\
A_2
:=
-\,
\frac{\varphi_{z_2}}{{\scriptstyle{\sqrt{-1}}}+\varphi_u}.
\endaligned
\]
Taking as a $1$-form $\rho_0$ with $\{ \rho_0 = 0\} = T^{1, 0} M 
\oplus T^{0, 1}M$:
\[
\rho_0
:=
-\,A_1\,dz_1
-
A_2\,dz_2
-
\overline{A_1}\,d\overline{z}_1
-
\overline{A_2}\,d\overline{z}_2
+
du,
\]
the top-left entry of the:
\begin{equation}
\label{levi-matrix}
\aligned
{\sf Levi}\text{-}{\sf Matrix}
&
=
\left(\!
\begin{array}{cc}
\rho_0\big({\scriptstyle{\sqrt{-1}}}\big[\mathcal{L}_1,\overline{\mathcal{L}}_1\big]\big)
&
\rho_0\big({\scriptstyle{\sqrt{-1}}}\big[\mathcal{L}_2,\overline{\mathcal{L}}_1\big]\big)
\\
\rho_0\big({\scriptstyle{\sqrt{-1}}}\big[\mathcal{L}_1,\overline{\mathcal{L}}_2\big]\big)
&
\rho_0\big({\scriptstyle{\sqrt{-1}}}\big[\mathcal{L}_2,\overline{\mathcal{L}}_2\big]\big)
\end{array}
\!\right)
\\
&
=
\left(\!
\begin{array}{cc}
{\scriptstyle{\sqrt{-1}}}\big(
\mathcal{L}_1\big(\overline{A_1}\big)
-
\overline{\mathcal{L}}_1\big(A_1\big)
\big)
&
{\scriptstyle{\sqrt{-1}}}\big(
\mathcal{L}_2\big(\overline{A_1}\big)
-
\overline{\mathcal{L}}_1\big(A_2\big)
\big)
\\
{\scriptstyle{\sqrt{-1}}}\big(
\mathcal{L}_1\big(\overline{A_2}\big)
-
\overline{\mathcal{L}}_2\big(A_1\big)
\big)
&
{\scriptstyle{\sqrt{-1}}}\big(
\mathcal{L}_2\big(\overline{A_2}\big)
-
\overline{\mathcal{L}}_2\big(A_2\big)
\big)
\end{array}
\!\right),
\endaligned
\end{equation}
expresses as:
\[
\aligned
{\scriptstyle{\sqrt{-1}}}\,
\big(\mathcal{L}_1(\overline{A_1}\big)
-
\overline{\mathcal{L}}_1\big(A_1\big)\big)
&
=
\frac{1}{({\scriptstyle{\sqrt{-1}}}+\varphi_u)^2(-{\scriptstyle{\sqrt{-1}}}+\varphi_u)^2}
\bigg\{
2\,\varphi_{z_1\overline{z}_1}
+
2\,\varphi_{z_1\overline{z}_1}\varphi_u\varphi_u
-
\\
&
\ \ \ \ \
-\,
2{\scriptstyle{\sqrt{-1}}}\,\varphi_{\overline{z}_1}\varphi_{z_1u}
-
2\,\varphi_{\overline{z}_1}\varphi_{z_1u}\varphi_u
+
2{\scriptstyle{\sqrt{-1}}}\,\varphi_{z_1}\varphi_{\overline{z}_1u}
+
\\
&
\ \ \ \ \ \
+
2\,\varphi_{z_1}\varphi_{\overline{z}_1}\varphi_{uu}
-
2\,\varphi_{z_1}\varphi_{\overline{z}_1u}\varphi_u
\bigg\},
\endaligned
\]
with quite similar expressions for the remaining three entries.
As $M$ belongs to Class $\text{\sf IV}_{\text{\sf 2}}$:
\[
\aligned
0
&
\equiv
\det\,
\left(\!
\begin{array}{cc}
{\scriptstyle{\sqrt{-1}}}
\big(
\overline{A_1}_{z_1}
+
A_1\,\overline{A_1}_u
-
{A_1}_{\overline{z_1}}
-
\overline{A_1}\,
{A_1}_u
\big)
&
{\scriptstyle{\sqrt{-1}}}
\big(
\overline{A_1}_{z_2}
+
A_2\,\overline{A_1}_u
-
{A_2}_{\overline{z_1}}
-
\overline{A_1}\,
{A_2}_u
\big)
\\
{\scriptstyle{\sqrt{-1}}}
\big(
\overline{A_2}_{z_1}
+
A_1\,\overline{A_2}_u
-
{A_1}_{\overline{z_2}}
-
\overline{A_2}\,
{A_1}_u
\big)
&
{\scriptstyle{\sqrt{-1}}}
\big(
\overline{A_2}_{z_2}
+
A_2\,\overline{A_2}_u
-
{A_2}_{\overline{z_2}}
-
\overline{A_2}\,
{A_2}_u
\big)
\end{array}
\!\right),
\endaligned
\]
that is to say in terms of $\varphi$:
\begin{equation}
\label{Levi-determinant}
\aligned
0
&
\equiv
\frac{4}{({\scriptstyle{\sqrt{-1}}}+\varphi_u)^3(-{\scriptstyle{\sqrt{-1}}}+\varphi_u)^3}
\bigg\{
\varphi_{z_2\overline{z}_2}\varphi_{z_1\overline{z}_1}
-
\varphi_{z_2\overline{z}_1}\varphi_{z_1\overline{z}_2}
+
\\
&
\ \ \ \ \ 
+
\varphi_{z_2\overline{z}_1}\varphi_{\overline{z}_2}\varphi_{z_1u}
\varphi_u
-
\varphi_{z_2\overline{z}_1}\varphi_{\overline{z}_2}\varphi_{z_1}
\varphi_{uu}
-
\varphi_{\overline{z}_1}\varphi_{z_2u}\varphi_{z_1}
\varphi_{\overline{z}_2u}
+
\varphi_{\overline{z}_1}\varphi_{z_2u}\varphi_u
\varphi_{z_1\overline{z}_2}
-
\\
&
\ \ \ \ \ 
-\,
\varphi_{z_2}\varphi_{\overline{z}_1u}\varphi_{\overline{z}_2}
\varphi_{z_1u}
-
\varphi_{z_2}\varphi_{\overline{z}_1}\varphi_{uu}
\varphi_{z_1\overline{z}_2}
+
\varphi_{z_2}\varphi_{\overline{z}_1u}
\varphi_u\varphi_{z_1\overline{z}_2}
-
\varphi_{z_2\overline{z}_2}\varphi_{\overline{z}_1}
\varphi_{z_1u}\varphi_u
+
\\
&
\ \ \ \ \ 
+
\varphi_{z_2\overline{z}_2}\varphi_{z_1}
\varphi_{\overline{z}_1}\varphi_{uu}
-
\varphi_{z_2\overline{z}_2}\varphi_{z_1}
\varphi_{\overline{z}_1u}\varphi_u
+
\varphi_{z_2\overline{z}_1}\varphi_{z_1}
\varphi_{\overline{z}_2u}\varphi_u
+
\varphi_{z_2}\varphi_{\overline{z}_2u}
\varphi_{\overline{z}_1}\varphi_{z_1u}
-
\\
&
\ \ \ \ \ 
-\,
\varphi_{z_2}\varphi_{\overline{z}_2u}
\varphi_{z_1\overline{z}_1}\varphi_u
+
\varphi_{\overline{z}_2}\varphi_{z_2u}
\varphi_{z_1}\varphi_{\overline{z}_1u}
-
\varphi_{\overline{z}_2}\varphi_{z_2u}
\varphi_u\varphi_{z_1\overline{z}_1}
+
\varphi_{\overline{z}_2}\varphi_{z_2}
\varphi_{uu}\varphi_{z_1\overline{z}_1}
+
\\
&
\ \ \ \ \ 
+
{\scriptstyle{\sqrt{-1}}}
\big(
\varphi_{z_2\overline{z}_2}\varphi_{z_1}
\varphi_{\overline{z}_1u}
+
\varphi_{\overline{z}_1}\varphi_{z_2u}
\varphi_{z_1\overline{z}_2}
+
\varphi_{z_2\overline{z}_1}\varphi_{\overline{z}_2}
\varphi_{z_1u}
+
\varphi_{z_2}\varphi_{\overline{z}_2u}
\varphi_{z_1\overline{z}_1}
\big)
-
\\
&
\ \ \ \ \ 
-\,{\scriptstyle{\sqrt{-1}}}\big(
\varphi_{\overline{z}_2}\varphi_{z_2u}
\varphi_{z_1\overline{z}_1}
+
\varphi_{z_2\overline{z}_1}\varphi_{z_1}
\varphi_{\overline{z}_2u}
+
\varphi_{z_2}\varphi_{\overline{z}_1u}
\varphi_{z_1\overline{z}_2}
+
\varphi_{z_2\overline{z}_2}\varphi_{\overline{z}_1}
\varphi_{z_1u}
\big)
-
\\
&
\ \ \ \ \ 
-\,
\varphi_{z_2\overline{z}_1}\varphi_{z_1\overline{z}_2}
\varphi_u\varphi_u
+
\varphi_{z_2\overline{z}_2}\varphi_{z_1\overline{z}_1}
\varphi_u\varphi_u
\bigg\}.\,\,
\endaligned
\end{equation}
Since the rank of the Levi Matrix~\thetag{
\ref{levi-matrix}} equals $1$ everywhere, after permutation,
its top-left entry vanishes nowhere. Hence a local 
generator for $K^{1, 0} M$ is:
\[
\mathcal{K}
=
k\,\mathcal{L}_1
+
\mathcal{L}_2,
\]
with:
\[
k
=
-\,
\frac{
\mathcal{L}_2\big(\overline{A_1}\big)
-
\overline{\mathcal{L}}_1\big(A_2\big)
}{
\mathcal{L}_1\big(\overline{A_1}\big)
-
\overline{\mathcal{L}}_1\big(A_1\big)
},
\]
namely:
\[
\!\!\!\!\!\!\!\!\!\!\!\!\!\!\!\!\!\!
\aligned
k
&
=
\frac{\varphi_{z_2\overline{z}_1}
+
\varphi_{z_2\overline{z}_1}\varphi_u\varphi_u
-
{\scriptstyle{\sqrt{-1}}}\varphi_{\overline{z}_1}\varphi_{z_2u}
-
\varphi_{\overline{z}_1}\varphi_{z_2u}\varphi_u
+
{\scriptstyle{\sqrt{-1}}}\varphi_{z_2}\varphi_{\overline{z}_1u}
+
\varphi_{z_2}\varphi_{\overline{z}_1}\varphi_{uu}
-
\varphi_{z_2}\varphi_{\overline{z}_1u}\varphi_u}{
-\,\varphi_{z_1\overline{z}_1}
-
\varphi_{z_1\overline{z}_1}\varphi_u\varphi_u
+
{\scriptstyle{\sqrt{-1}}}\varphi_{\overline{z}_1}\varphi_{z_1u}
+
\varphi_{\overline{z}_1}\varphi_{z_1u}\varphi_u
-
{\scriptstyle{\sqrt{-1}}}\varphi_{z_1}\varphi_{\overline{z}_1u}
-
\varphi_{z_1}\varphi_{\overline{z}_1}\varphi_{uu}
+
\varphi_{z_1}\varphi_{\overline{z}_1u}\varphi_u
},
\endaligned
\]
and there is a surprising computational fact that this function $k$
happens to be also equal to the other two quotients
(\cite{ Merker-5-CR}, II, p.~82):
\begin{equation}
\label{surprising-computational-fact}
k
=
-\,
\frac{\mathcal{L}_2\big(\overline{A_1}\big)}{
\mathcal{L}_1\big(\overline{A_1}\big)}
=
-\,
\frac{-\,\overline{\mathcal{L}}_1\big(A_2\big)}{
-\,\overline{\mathcal{L}}_1\big(A_1\big)}.
\end{equation}
Heuristically, this fact becomes transparent when $\varphi = \varphi(x_1, 
x_2, y_1, y_2)$ is independent of $u$, whence the Levi matrix becomes:
\[
\left(\!
\begin{array}{cc}
2\,\varphi_{z_1\overline{z}_1} 
&
2\,\varphi_{z_2\overline{z}_1}
\\
2\,\varphi_{z_1\overline{z}_2}
&
2\,\varphi_{z_2\overline{z}_2}
\end{array}
\!\right),
\]
and clearly:
\[
k
=
-\,\frac{2\,\varphi_{z_2\overline{z}_1}}{
2\,\varphi_{z_1\overline{z}_1}}
=
-\,\frac{\varphi_{z_2\overline{z}_1}}{
\varphi_{z_1\overline{z}_1}}
=
-\,\frac{\mathcal{L}_2(\varphi_{\overline{z}_1})}{
\mathcal{L}_1(\varphi_{\overline{z}_1})}
=
-\,
\frac{-\,\overline{\mathcal{L}}_1(\varphi_{z_2})}{
-\,\overline{\mathcal{L}}_1(\varphi_{z_1})}.
\]

\begin{proposition}
\label{L1-K-kappa-0}
{\rm (\cite{ Merker-5-CR})}
In any system of holomorphic coordinates, 
for any choice of Levi-kernel adapted local $T^{1,0}M$-frame
$\big\{\mathcal{L}_1,\mathcal{K}\big\}$
satisfying:
\[
K^{1,0}M
=
\C\,\mathcal{K},
\]
and for any choice of differential $1$-forms
$\big\{
\rho_0,\kappa_0,\zeta_0
\big\}$
satisfying:
\[
\aligned
\big\{0=\rho_0\big\}
&
=
T^{1,0}M\oplus T^{0,1}M,
\\
\big\{0=\rho_0=\kappa_0=\overline{\kappa}_0=\overline{\zeta}_0\big\}
&
=
K^{1,0}M,
\endaligned
\]
the quantity:
\[
\kappa_0
\big(
\big[\mathcal{K},\,\overline{\mathcal{L}}_1\big]
\big),
\]
is, at one fixed point $p \in M$, either zero
or nonzero, independently of any choice.
\end{proposition}

\begin{definition}
The {\sl Freeman form} at a point $p \in M$ is the value of $\kappa_0
\big( \big[\mathcal{K},\,\overline{\mathcal{L}}_1\big] \big) (p)$, and
it depends only on $\kappa_0(p)$, $\mathcal{ K} \big\vert_p$,
$\overline{ \mathcal{ L}}_1 \big\vert_p$.
\end{definition}

With a $T^{1, 0}M$-frame $\{ \mathcal{ K}, \mathcal{ L}_1 \}$ 
satisfying $K^{1, 0} M = \C \mathcal{ K}$, define quite equivalently:
\[
\aligned
{\sf Freeman}\text{-}{\sf Form}^M(p)
\colon
\left[
\aligned
K_p^{1,0}M
\times
\big(
T_p^{1,0}M\!\!\!\!
\mod\,K_p^{1,0}M
\big)
&
\,\longrightarrow\,
\C
\\
\big(\mathcal{K}_p,\mathcal{L}_{1p}\big)
&
\,\longmapsto\,
\big[\mathcal{K},\,\overline{\mathcal{L}}_1\big](p)
\\
&
\ \ \ \ \ \ \
\mod\,\,
\big(
K^{1,0}M
\oplus
T^{0,1}M
\big),
\endaligned\right.
\endaligned
\]
the result being independent of vector field extensions $\mathcal{
K}\big\vert_p = \mathcal{ K}_p$ and $\mathcal{ L}_1 \big\vert_p =
\mathcal{ L}_{ 1p}$.

\medskip\noindent{\bf General Classes $\text{\sf IV}_{\text{\sf 1}}$,
$\text{\sf IV}_{\text{\sf 2}}$.} 
For a connected $\mathcal{ C}^\omega$ hypersurface $M^5 \subset \C^3$,
if the Levi form is of rank $2$ at one point, it is
of rank $2$ at every
Zariski-generic point. Excluding Levi degenerate points, this brings
$\text{\sf IV}_{\text{\sf 1}}$.

\smallskip

If the Levi form is identically zero, then as is known $M \cong \C^2
\times \R$.

\smallskip

If the Levi form is of rank $1$, the Freeman form creates
bifurcation:

\begin{proposition}
\text{\rm (\cite{ Merker-5-CR})}
A $\mathcal{ C}^\omega$ hypersurface $M^5
\subset
\C^3$ 
having at every point $p$:
\[
{\sf rank}_\C
\big(
{\sf Levi}\text{-}{\sf Form}^M(p)
\big)
=
{\bf 1}
\]
has an identically vanishing:
\[
{\sf Freeman}\text{-}{\sf Form}^M(p)
\equiv
{\bf 0},
\]
if and only if it is locally biholomorphic to a product:
\[
M^5
\,\cong\,
M^3
\times
\C
\]
with a $\mathcal{ C}^\omega$ hypersurface $M^3 \subset \C^2$.
\end{proposition}

\begin{proof}
\smartqed
With:
\[
\mathcal{K}
=
k\,\mathcal{L}_1
+
\mathcal{L}_2
=
k\,\frac{\partial}{\partial z_1}
+
\frac{\partial}{\partial z_2}
+
\big(k\,A_1+A_2\big)\,
\frac{\partial}{\partial u},
\]
the involutiveness~\thetag{ \ref{K-involutiveness}}:
\[
\big[\mathcal{K},\,\overline{\mathcal{K}}\big]
=
\text{function}
\cdot
\mathcal{K}
+
\text{function}
\cdot
\overline{\mathcal{K}},
\]
and the fact that this bracket does not contain either $\partial \big/
\partial z_2$ or $\partial \big/ \partial \overline{ z}_2$:
\[
\aligned
\big[
\mathcal{K},\,
\overline{\mathcal{K}}
\big]
&
=
\bigg[
k\,\frac{\partial}{\partial z_1}
+
\frac{\partial}{\partial z_2}
+
\big(k\,A_1+A_2\big)\,
\frac{\partial}{\partial u},\,\,\,
\overline{k}\,\frac{\partial}{\partial\overline{z}_1}
+
\frac{\partial}{\partial\overline{z}_2}
+
\big(\overline{k}\,\overline{A}_1+\overline{A}_2\big)\,
\frac{\partial}{\partial u}
\bigg]
\\
&
=
\mathcal{K}\big(\overline{k}\big)\,
\frac{\partial}{\partial\overline{z}_1}
-
\overline{\mathcal{K}}\big(k\big)\,
\frac{\partial}{\partial z_1}
+
\Big(
\mathcal{K}\big(\overline{k}\,\overline{A}_1+\overline{A}_2\big)
-
\overline{\mathcal{K}}\big(k\,A_1+A_2\big)
\Big)\,
\frac{\partial}{\partial u},
\endaligned
\]
entail:
\[
0
\equiv
\overline{\mathcal{K}}(k)
\equiv
\mathcal{K}\big(\overline{k}\big).
\]
Next, identical vanishing of the Freeman form:
\[
\big[\mathcal{K},\,\overline{\mathcal{L}}_1\big]
\,\equiv\,
0
\ \ \ \ \ 
\mod\,\,
\big(
\mathcal{K},\,\overline{\mathcal{K}},\,\overline{\mathcal{L}}_1
\big),
\]
with:

\[
\aligned
\big[\mathcal{K},\,\overline{\mathcal{L}}_1\big]
&
=
\bigg[
k\,\frac{\partial}{\partial z_1}
+
\frac{\partial}{\partial z_2}
+
\big(k\,A_1+A_2\big)\,
\frac{\partial}{\partial u},\,\,\,
\frac{\partial}{\partial\overline{z}_1}
+
\overline{A}_1\,\frac{\partial}{\partial u}
\bigg]
\\
&
=
\overline{\mathcal{L}}_1(k)\,
\frac{\partial}{\partial z_1}
+
\text{something}\,
\frac{\partial}{\partial u},
\endaligned
\]
reads as:
\[
0
\equiv
\overline{\mathcal{L}}_1(k),
\]
and since $\big\{ \overline{ \mathcal{ K}},
\overline{ \mathcal{L}}_1 \big\}$ is a $T^{0, 1}M$-frame, 
\[
\text{\sl The $\mathcal{ C}^\omega$ slanting function $k$ is 
a CR function!}
\]
Moreover, the last coefficient-function of:
\[
\mathcal{K}
=
k\,\frac{\partial}{\partial z_1}
+
\frac{\partial}{\partial z_2}
+
\big(k\,A_1+A_2\big)\,
\frac{\partial}{\partial u}
\]
is {\em also} a CR function, namely it is annihilated by
$\overline{ \mathcal{ L}}_1$, $\overline{ \mathcal{ L}}_2$,
because firstly:
\[
0
\overset{?}{\,=\,}
\overline{\mathcal{L}}_1
\big(
k\,A_1+A_2\big)
\,=\,
k\,\overline{\mathcal{L}}_1(A_1)
+
\overline{\mathcal{L}}_1(A_2)
\overset{\text{ok}}{\,=\,}
0,
\]
thanks to~\thetag{ 
\ref{surprising-computational-fact}}, and secondly because
a direct computation gives:
\[
\aligned
\!\!\!\!\!\!\!\!\!\!\!\!\!\!\!\!\!\!\!\!\!\!\!\!\!\!\!\!\!\!
0
&
\overset{?}{\,=\,}
\overline{\mathcal{L}}_2
\big(
k\,A_1+A_2\big)
\\
\!\!\!\!\!\!\!\!\!\!\!\!\!\!\!\!\!\!\!\!\!\!\!\!\!\!\!\!\!\!
&
\,=\,
k\,\overline{\mathcal{L}}_2(A_1)
+
\overline{\mathcal{L}}_2(A_2)
\\
\!\!\!\!\!\!\!\!\!\!\!\!\!\!\!\!\!\!\!\!\!\!\!\!\!\!\!\!\!\!
&
\,=\,
\frac{-\,\text{numerator of the Levi determinant}}{
({\scriptstyle{\sqrt{-1}}}+\varphi_u)\,
[\varphi_{z_1\overline{z}_1}+\varphi_{z_1\overline{z}_1}\varphi_u\varphi_u
-
{\scriptstyle{\sqrt{-1}}}\,\varphi_{\overline{z}_1}\varphi_{z_1u}
-
\varphi_{\overline{z}_1}\varphi_{z_1u}\varphi_u
+
{\scriptstyle{\sqrt{-1}}}\,\varphi_{z_1}\varphi_{\overline{z}_1u}\varphi_u
+
\varphi_{z_1}\varphi_{\overline{z}_1}\varphi_{uu}]}
\\
&
\,\equiv\,
0.
\endaligned
\]
In conclusion, the $(1, 0)$ field $\mathcal{ K}$ has $\mathcal{
C}^\omega$ CR coefficient-functions, hence $\mathcal{ K}$ is locally
extendable to a neighborhood of $M$ in $\C^3$ as a $(1, 0)$ field
having holomorphic coefficients, and a straightening $\mathcal{ K}
= \partial \big/ \partial z_2$ yields $M^5 \cong M^3 \times \C$.
\qed
\end{proof}

Excluding therefore such a degeneration, and focusing attention on a
Zariski-generic initial classification, it therefore remains only the
General Class $\text{\sf IV}_{\text{\sf 2}}$.

\subsection{Existence of the six General Classes $\text{\sf
I}$, $\text{\sf II}$, $\text{\sf III}_{\text{\sf 1}}$, 
$\text{\sf III}_{\text{\sf 2}}$, 
$\text{\sf IV}_{\text{\sf 1}}$,
$\text{\sf IV}_{\text{\sf 2}}$}

Graphing functions are essentially free and arbitrary.

\begin{proposition}
\text{\rm (\cite{ Merker-5-CR})}
Every CR-generic submanifold belonging to the four classes $\text{\sf
I}$, $\text{\sf II}$, $\text{\sf III}_{\text{\sf 1}}$, $\text{\sf
IV}_{\text{\sf 1}}$ may be represented in suitable local holomorphic
coordinates as:
\[
\underline{\text{\sf (I):}}
\ \ \ \ \
\left[
\aligned
v
=
z\overline{z}
+
z\overline{z}\,{\rm O}_1\big(z,\overline{z}\big)
+
z\overline{z}\,{\rm O}_1(u),
\endaligned\right.
\]
\[
\underline{\text{\sf (II):}}
\ \ \ \ \
\left[
\aligned
v_1
&
=
z\overline{z}
\ \ \ \ \ \ \ \ \ \ \ \ 
+
z\overline{z}\,{\rm O}_2\big(z,\overline{z}\big)
+
z\overline{z}\,{\rm O}_1(u_1)
+
z\overline{z}\,{\rm O}_1(u_2),\,\,
\\
v_2
&
=
z^2\overline{z}
+
z\overline{z}^2
+
z\overline{z}\,{\rm O}_2\big(z,\overline{z}\big)
+
z\overline{z}\,{\rm O}_1(u_1)
+
z\overline{z}\,{\rm O}_1(u_2),
\endaligned\right.
\]
\[
\underline{\text{\sf 
$\text{\sf (III)}_{\text{\sf 1}}$:}}
\ \ \ \ \
\left[
\aligned
v_1
&
=
z\overline{z}
\ \ \ \ \ \ \ \ \ \ \ \ \ \ \ \ \ \ \ \ \ \ \
+
z\overline{z}\,{\rm O}_2\big(z,\overline{z}\big)
+
z\overline{z}\,{\rm O}_1(u_1)
+
z\overline{z}\,{\rm O}_1(u_2)
+
z\overline{z}\,{\rm O}_1(u_3),\,\,
\\
v_2
&
=
z^2\overline{z}
+
z\overline{z}^2
\ \ \ \ \ \ \ \ \ \ \,
+
z\overline{z}\,{\rm O}_2\big(z,\overline{z}\big)
+
z\overline{z}\,{\rm O}_1(u_1)
+
z\overline{z}\,{\rm O}_1(u_2)
+
z\overline{z}\,{\rm O}_1(u_3),\,\,
\\
v_3
&
=
{\scriptstyle{\sqrt{-1}}}\,\big(z^2\overline{z}
-
z\overline{z}^2\big)
+
z\overline{z}\,{\rm O}_2\big(z,\overline{z}\big)
+
z\overline{z}\,{\rm O}_1(u_1)
+
z\overline{z}\,{\rm O}_1(u_2)
+
z\overline{z}\,{\rm O}_1(u_3),
\endaligned\right.
\]
\[
\underline{\text{\sf 
$\text{\sf (IV)}_{\text{\sf 1}}$:}}
\ \ \ \ \
\left[
\aligned
v
=
z_1\overline{z}_1
\pm
z_2\overline{z}_2
+
{\rm O}_3\big(z_1,z_2,\overline{z}_1,\overline{z}_2,u\big),
\endaligned\right.
\]
with arbitrary remainder functions. For class:
\[
\!\!\!\!\!\!\!\!\!\!\!\!\!\!\!\!\!\!\!\!
\underline{\text{\footnotesize\sf (III)}_{
\text{\tiny\sf 2}}:}
\ \ \ \ \
\aligned
v_1
&
=
z\overline{z}
+
c_1\,z^2\overline{z}^2
+
z\overline{z}\,{\rm O}_3\big(z,\overline{z}\big)
+
z\overline{z}\,u_1\,{\rm O}_1\big(z,\overline{z},u_1\big)
+
\\
&
\ \ \ \ \ \ \ \ \ \ \ \ \ \ \ \ \ \ \ \ \ \ 
+
z\overline{z}\,u_2\,{\rm O}_1\big(z,\overline{z},u_1,u_2\big)
+
z\overline{z}\,u_3\,{\rm O}_1\big(z,\overline{z},u_1,u_2,u_3\big),
\\
v_2
&
=
z^2\overline{z}
+
z\overline{z}^2
+
z\overline{z}\,{\rm O}_3\big(z,\overline{z}\big)
+
z\overline{z}\,u_1\,{\rm O}_1\big(z,\overline{z},u_1\big)
+
\\
&
\ \ \ \ \ \ \ \ \ \ \ \ \ \ \ \ \ \ \ \ \ \ 
+
z\overline{z}\,u_2\,{\rm O}_1\big(z,\overline{z},u_1,u_2\big)
+
z\overline{z}\,u_3\,{\rm O}_1\big(z,\overline{z},u_1,u_2,u_3\big),
\\
v_3
&
=
2\,z^3\overline{z}
+
2\,z\overline{z}^3
+
3\,z^2\overline{z}^2
+
z\overline{z}\,{\rm O}_3\big(z,\overline{z}\big)
+
z\overline{z}\,u_1\,{\rm O}_1\big(z,\overline{z},u_1\big)
+
\\
&
\ \ \ \ \ \ \ \ \ \ \ \ \ \ \ \ \ \ \ \ \ \ \ \ \ \ \ \ \ \ \ \ \ \ \ 
\ \ \ \ \ \ 
+
z\overline{z}\,u_2\,{\rm O}_1\big(z,\overline{z},u_1,u_2\big)
+
z\overline{z}\,u_3\,{\rm O}_1\big(z,\overline{z},u_1,u_2,u_3\big),
\endaligned
\rule[-5pt]{0pt}{27pt}
\]
the $3$ graphing functions $\varphi_1$, $\varphi_2$, $\varphi_3$
are subjected to the identical vanishing:
\[
\aligned
0
\equiv
\left\vert\!
\begin{array}{ccc}
\substack{
\mathcal{L}(\overline{A}_1)-\overline{\mathcal{L}}(A_1)}
&
\substack{
\mathcal{L}(\overline{A}_2)-\overline{\mathcal{L}}(A_2)}
&
\substack{
\mathcal{L}(\overline{A}_3)-\overline{\mathcal{L}}(A_3)}\bigskip
\\
\substack{
\mathcal{L}(\mathcal{L}(\overline{A}_1))
-2\mathcal{L}(\overline{\mathcal{L}}(A_1))+\\
+\overline{\mathcal{L}}(\mathcal{L}(A_1))}
&
\substack{
\mathcal{L}(\mathcal{L}(\overline{A}_2))
-2\mathcal{L}(\overline{\mathcal{L}}(A_2))+\\
+\overline{\mathcal{L}}(\mathcal{L}(A_2))}
&
\substack{
\mathcal{L}(\mathcal{L}(\overline{A}_3))
-2\mathcal{L}(\overline{\mathcal{L}}(A_3))+\\
+\overline{\mathcal{L}}(\mathcal{L}(A_3))}\bigskip
\\
\substack{
-\overline{\mathcal{L}}(\overline{\mathcal{L}}(A_1))
+2\overline{\mathcal{L}}(\mathcal{L}(\overline{A}_1))-\\
-\mathcal{L}(\overline{\mathcal{L}}(\overline{A}_1))}
&
\substack{
-\overline{\mathcal{L}}(\overline{\mathcal{L}}(A_2))
+2\overline{\mathcal{L}}(\mathcal{L}(\overline{A}_2))-\\
-\mathcal{L}(\overline{\mathcal{L}}(\overline{A}_2))}
&
\substack{
-\overline{\mathcal{L}}(\overline{\mathcal{L}}(A_3))
+2\overline{\mathcal{L}}(\mathcal{L}(\overline{A}_3))-\\
-\mathcal{L}(\overline{\mathcal{L}}(\overline{A}_3))}
\end{array}
\!\right\vert.
\endaligned
\]
Lastly, for class:
\[
\underline{\text{\small\sf 
$\text{\sf (IV)}_{\text{\sf 2}}$:}}
\ \ \ \ \
\left[
\aligned
v
=
z_1\overline{z}_1
+
{\textstyle{\frac{1}{2}}}\,
z_1z_1\overline{z}_2
+
{\textstyle{\frac{1}{2}}}\,
z_2\overline{z}_1\overline{z}_1
+
{\rm O}_4\big(z_1,z_2,\overline{z}_1,\overline{z}_2\big)
+
u\,
{\rm O}_2\big(z_1,z_2,\overline{z}_1,\overline{z}_2,u\big),
\endaligned\right.
\]
the graphing function $\varphi$ is subjected to the identical
vanishing of the Levi determinant~\thetag{ 
\ref{Levi-determinant}}.

\end{proposition}

\subsection{Cartan equivalences and curvatures}

\medskip\noindent{\bf Problem 3.} 
(Subproblem of Problem 1)
{\sl Perform Cartan's local equivalence procedure for these six general
classes $\text{\sf I}$, $\text{\sf II}$, $\text{\sf III}_{\text{\sf
1}}$, $\text{\sf III}_{\text{\sf 2}}$, $\text{\sf IV}_{\text{\sf
1}}$, $\text{\sf IV}_{\text{\sf 2}}$ of nondegenerate CR
manifolds up to dimension {\bf 5}.}

\medskip

\medskip\noindent{\bf Class $\text{\sf I}$ equivalences.}
Consider firstly Class $\text{\sf I}$ hypersurfaces $M^3 \subset \C^2$
graphed as $v = \varphi(x, y, u)$ with:
\[
\bigg\{
\mathcal{L}
=
\frac{\partial}{\partial z}
-
\frac{\varphi_z}{{\scriptstyle{\sqrt{-1}}}+\varphi_u}\,
\frac{\partial}{\partial u},
\ \ \ \ \
\overline{\mathcal{L}}
=
\frac{\partial}{\partial\overline{z}}
-
\frac{\varphi_{\overline{z}}}{
-\,{\scriptstyle{\sqrt{-1}}}+\varphi_u}\,
\frac{\partial}{\partial u},
\ \ \ \ \
\mathcal{T}
:=
{\scriptstyle{\sqrt{-1}}}\,\big[
\mathcal{L},\overline{\mathcal{L}}
\big]
=
\ell\,\frac{\partial}{\partial u}
\bigg\}
\]
making a frame for $\C \otimes_\R TM$, where the Levi-factor
(nonvanishing) function:
\begin{equation}
\label{ell-denominator}
\aligned
\ell
:= 
&\,
{\scriptstyle{\sqrt{-1}}}\,
\Big(
\mathcal{L}\big(\overline{A}\big)
-
\overline{\mathcal{L}}(A)
\Big)
\\
=
&\,
\frac{1}{({\scriptstyle{\sqrt{-1}}}+\varphi_u)^2\,
(-{\scriptstyle{\sqrt{-1}}}+\varphi_u)^2}\,
\bigg\{
2\,\varphi_{z\overline{z}}
+
2\,\varphi_{z\overline{z}}\,\varphi_u\,\varphi_u
-
2\,{\scriptstyle{\sqrt{-1}}}\,\varphi_{\overline{z}}\,\varphi_{zu}
-
2\,\varphi_{\overline{z}}\,\varphi_{zu}\,\varphi_u
\,+
\\
&
\ \ \ \ \ \ \ \ \ \ \ \ \ \ \ \ \ \ \ \ \ \ \ \ \ \ \ \ \ \ \ \ \ \
\ \ \ \ \ \ \ \ \ \ \ \ \ \ \ \ \ \ \
+
2\,{\scriptstyle{\sqrt{-1}}}\,\varphi_z\,\varphi_{\overline{z}u}
+
2\,\varphi_z\,\varphi_{\overline{z}}\,\varphi_{uu}
-
2\,\varphi_z\,\varphi_{\overline{z}u}\,\varphi_u
\bigg\}
\endaligned
\end{equation}
{\em will, notably, enter computations in denominator place}.

Introduce also the dual coframe for $\C \otimes_\R T^*M$:
\[
\big\{
\rho_0,\,
\overline{\zeta}_0,\,
\zeta_0
\big\},
\]
satisfying:
\[
\begin{array}{ccc}
\rho_0(\mathcal{T})=1
\ \ \ & \ \ \ 
\rho_0(\overline{\mathcal{L}})=0
\ \ \ & \ \ \ 
\rho_0(\mathcal{L})=0,
\\
\overline{\zeta}_0(\mathcal{T})=0
\ \ \ & \ \ \ 
\overline{\zeta_0}(\overline{\mathcal{L}})=1
\ \ \ & \ \ \ 
\overline{\zeta}_0\big(\mathcal{L}\big)=0,
\\
\zeta_0(\mathcal{T})=0
\ \ \ & \ \ \ 
\zeta_0(\overline{\mathcal{L}})=0
\ \ \ & \ \ \ 
\zeta_0\big(\mathcal{L}\big)=1.
\end{array}
\]
Since:
\[
\big[\mathcal{L},\,\mathcal{T}\big]
=
\bigg[
\frac{\partial}{\partial z}
+
A\,\frac{\partial}{\partial u},\,\,
\ell\,
\frac{\partial}{\partial u}
\bigg]
=
\Big(
\ell_z
+
A\,\ell_u
-
\ell\,A_u
\Big)\,
\frac{\partial}{\partial u}
=
\overbrace{
\frac{\ell_z+A\,\ell_u-\ell\,A_u}{\ell}}^{=:\,P}\,
\mathcal{T},
\]
the {\sl initial Darboux structure} reads dually as:
\[
\aligned
d\rho_0
&
=
P\,\rho_0\wedge\zeta_0
+ 
\overline{P}\,\rho_0\wedge\overline{\zeta}_0 
+
{\scriptstyle{\sqrt{-1}}}\,\zeta_0\wedge\overline{\zeta}_0,
\\
d\overline{\zeta}_0
&
=
0,
\\
d\zeta_0
&
=
0,
\endaligned
\]
with a {\em single} fundamental function $P$. \'Elie Cartan in 1932
performed his equivalence procedure (well presented in~\cite{
Olver-1995}) for such $M^3 \subset \C^2$, but the completely
explicit aspects must be endeavoured once again for systematic
treatment of Problem~3. Within Tanaka's theory,
Ezhov-McLaughlin-Schmalz (\cite{ Ezhov-McLaughlin-Schmalz-2011})
already constructed a Cartan connection on a certain principal bundle
$N^8 \to M^3$, whose effective aspects have been explored further
in~\cite{Merker-Sabzevari-2012, Merker-Sabzevari-M3-C2}.

Indeed, as already observed in Lemma~\ref{3-rank} ({\em see} also~\cite{
Merker-5-CR, Merker-Sabzevari-M3-C2}), the initial
{\sl $G$-structure} for (local) biholomorphic
equivalences of such hypersurfaces is:
\[
{\sf G}_{\text{\sf IV}_{\text{\sf 2}}}^{\sf initial}
\,:=\,
\left\{
\left(\!
\begin{array}{ccc} 
{\sf a}\overline{\sf a} & 0 & 0 \\ 
\overline{\sf b} & \overline{\sf a} & 0 \\ 
{\sf b} & 0 & {\sf a} \\ 
\end{array}
\!\right)
\,\in\,{\sf GL}_3(\C)
\colon 
\ \ \ 
{\sf a}\in\C\backslash\{0\},\,\,
{\sf b}\in\C
\right\}.
\]
Cartan's method is to deal with the so-called {\sl lifted coframe:}
\[
\left(\!\!
\begin{array}{c}
\rho
\\
\overline{\zeta}
\\
\zeta
\end{array}
\!\!\right)
:=
\left(\!\!
\begin{array}{ccc}
{\sf a}\overline{\sf a} & 0 & 0
\\
\overline{\sf b} & \overline{\sf a} & 0
\\
{\sf b} & 0 & {\sf a}
\end{array}
\!\!\right)
\left(\!\!
\begin{array}{c}
\rho_0
\\
\overline{\zeta}_0
\\
\zeta_0
\end{array}
\!\!\right),
\]
in the space of $\big(x, y, u, {\sf a}, \overline{\sf a}, {\sf b},
\overline{\sf b}\big)$. After two absorbtions-normalizations and after
one prolongation:

\begin{theorem}
{\rm (M.-{\sc Sabzevari}, \cite{ Merker-Sabzevari-M3-C2})}
The biholomorphic equivalence problem for $M^3 \subset \C^2$ 
transforms some explicit eight-dimensional coframe:
\[
\big\{
\rho,\overline{\zeta},\zeta,\alpha,
\beta,\overline{\alpha},\overline{\beta},\delta 
\big\},
\]
on a certain manifold $N^8 \longrightarrow M^3$
having $\{ e\}$-structure equations:
\begin{eqnarray*} 
\begin{array}{ll}\footnotesize 
d\rho=\alpha\wedge\rho+\overline\alpha\wedge\rho
+
{\scriptstyle{\sqrt{-1}}}\,\,\zeta\wedge\overline\zeta, 
& 
\\ 
d\overline\zeta=\overline\beta\wedge\rho+\overline\alpha\wedge\overline\zeta, 
& 
\\ 
d\zeta=\beta\wedge\rho+\alpha\wedge\zeta, & 
\\ 
d\alpha=\delta\wedge\rho+2\,{\scriptstyle{\sqrt{-1}}}\,\zeta\wedge\overline{\beta}
+
{\scriptstyle{\sqrt{-1}}}\,\overline{\zeta}\wedge\beta,

& 
\\ 
d\beta=\delta\wedge\zeta+\beta\wedge\overline{\alpha}+{\frak 
I}\,\overline\zeta\wedge\rho, 
& 
\\ 
d\overline\alpha=\delta\wedge\rho-2\,{\scriptstyle{\sqrt{-1}}}\,
\overline\zeta\wedge{\beta}-{\scriptstyle{\sqrt{-1}}}\,{\zeta}\wedge\overline\beta,
& 
\\ 
d\overline\beta=\delta\wedge\overline\zeta
+
\overline\beta\wedge{\alpha}+\overline{\frak
I}\,\zeta\wedge\rho, 
& 
\\ 
d\delta=\delta\wedge\alpha+\delta\wedge\overline\alpha
+
{\scriptstyle{\sqrt{-1}}}\,\beta\wedge\overline\beta+{\frak
T}\,\rho\wedge\zeta+\overline{\frak T}\,\rho\wedge\overline\zeta, & 
\end{array} 
\end{eqnarray*} 
with the single primary complex invariant: 
\[
\aligned 
\mathfrak{I}
&
:=
\frac{1}{6}\,
\frac{1}{{\sf a}\overline{\sf a}^3}\,
\Big(
-\,2\,
\overline{\mathcal{L}}\big(\mathcal{L}\big(
\overline{\mathcal{L}}\big(\overline{P}\big)\big)\big)
+
3\,\overline{\mathcal{L}}\big(
\overline{\mathcal{L}}\big(
\mathcal{L}\big(\overline{P}\big)\big)\big)
-\,
7\,\overline{P}\,\,
\overline{\mathcal{L}}\big(\mathcal{L}\big(\overline{P}\big)\big)
+
\\
&
\ \ \ \ \ \ \ \ \ \ \ \ \ \ \ \ \ \ \ \,
+
4\,\overline{P}\,
\mathcal{L}\big(\overline{\mathcal{L}}\big(
\overline{P}\big)\big)
-
\mathcal{L}\big(\overline{P}\big)\,
\overline{\mathcal{L}}\big(\overline{P}\big)
+
2\,\overline{P}\,\overline{P}\,\mathcal{L}
\big(\overline{P}\big)
\Big),
\endaligned
\]
and with one secondary invariant:
\[ 
\mathfrak{T}
=
\frac{1}{\overline{\sf a}}\,\bigg(\overline{\mathcal 
L}(\overline{\frak I})-\overline P\,\overline{\frak 
I}\bigg)-{\scriptstyle{\sqrt{-1}}}\,\frac{\sf b}{{\sf a}\overline{\sf a}}\,\overline{\frak I}. 
\] 
\end{theorem}

\medskip\noindent{\bf Explicitness obstacle.}
In terms of the function $P$, the formulas for $\mathfrak{ I}$, for
$\mathfrak{ T}$ and for the $1$-forms constituting the $\{
e\}$-structure are writable, but when expressing everything in
terms of the graphing function $\varphi$, because $P$ involves the
Levi factor $\ell$ in denominator place, formulas `explode'.

Indeed, the real and imaginary parts $\Delta_1$ and $\Delta_2$ in:
\[
\mathfrak{I}
=
\frac{4}{{\sf a}\overline{\sf 
a}^3}\big(\pmb{\Delta}_1
+
{\scriptstyle{\sqrt{-1}}}\,\pmb{\Delta}_4\big)
\]
have numerators containing respectively
(\cite{ Merker-Sabzevari-M3-C2}):
\[
{\bf 1\,553\,198}
\ \ \ \ \ \ \ \ \ \ \ \ \ \
\text{\rm and}
\ \ \ \ \ \ \ \ \ \ \ \ \ \
{\bf 1\,634\,457}
\]
monomials in the differential ring in $\binom{ 6 + 3}{ 3} - 1 = 83$
variables:
\[
\mathbb{Z}\big[
\varphi_x,\varphi_y,
\varphi_{x^2},\varphi_{y^2},\varphi_{u^2},
\varphi_{xy},\varphi_{xu},\varphi_{yu},\dots\dots,
\varphi_{x^6},\varphi_{y^6},\varphi_{u^6},\dots
\big].
\]
Strikingly, though, in the so-called {\sl rigid case} (often useful as
a case of study-exploration) where $\varphi = \varphi ( x, y)$ is
independent of $u$ so that:
\[
P
=
\frac{\varphi_{zz\overline{z}}}{\varphi_{z\overline{z}}},
\ \ \ \ \ \ \ \ \ \
\text{\rm Levi factor at denominator}
=
\ell
=
\varphi_{z\overline{z}},
\]
$\mathfrak{ I}$ is easily writable:
\[
\aligned
\mathfrak{I}
\bigg\vert_{
\substack{
\text{\sf rigid}\\
\text{\sf case}}}
&
=
\frac{1}{6}\,
\frac{1}{{\sf a}\overline{\sf a}^3}
\bigg(
\frac{\varphi_{z^2\overline{z}^4}}{\varphi_{z\overline{z}}}
-
6\,
\frac{\varphi_{z^2\overline{z}^3}\,\varphi_{z\overline{z}^2}}{
(\varphi_{z\overline{z}})^2}
-
\frac{\varphi_{z\overline{z}^4}\,\varphi_{z^2\overline{z}}}{
(\varphi_{z\overline{z}})^2}
-
4\,
\frac{\varphi_{z\overline{z}^3}\,\varphi_{z^2\overline{z}^2}}{
(\varphi_{z\overline{z}})^2}
+
\\
&
\ \ \ \ \ \ \ \ \ \ \ \ \ \ \ \ \ \ \
+
10\,
\frac{\varphi_{z\overline{z}^3}\,\varphi_{z^2\overline{z}}\,
\varphi_{z\overline{z}^2}}{
(\varphi_{z\overline{z}})^3}
+
15\,
\frac{(\varphi_{z\overline{z}^2})^2\,
\varphi_{z^2\overline{z}^2}}{
(\varphi_{z\overline{z}})^3}
-
15\,
\frac{(\varphi_{z\overline{z}^2})^3\,\varphi_{z^2\overline{z}}}{
(\varphi_{z\overline{z}})^4}
\bigg),
\endaligned
\]
and this therefore shows that {\em there is a spectacular contrast of
computational complexity when passing from the rigid case to the
general case}. The reason of this contrast mainly comes from the size of 
the Levi factor $\ell$ in~\thetag{ \ref{ell-denominator}} 
{\em appearing at denominator place in subsequent differentiations}.

\medskip

\medskip\noindent{\bf Class $\text{\sf IV}_{\text{\sf 2}}$ equivalences.}
Consider secondly a Class $\text{\sf IV}_{\text{\sf 2}}$ hypersurface 
$M^5 \subset \C^3$, and let a Levi-kernel adapted frame
for $\C \otimes_\R TM$ be:
\[
\aligned
&
\big\{
\mathcal{T},\,
\overline{\mathcal{L}}_1,\,
\overline{\mathcal{K}},\,
\mathcal{L}_1,\,
\mathcal{K}
\big\},
\ \ \ \ \ \ \ \ \ \
\ \ \ \ \ \ \ \ \ \
\overline{\mathcal{T}}
=
\mathcal{T},
\\
&
\mathcal{T}
:=
{\scriptstyle{\sqrt{-1}}}\,
\big[\mathcal{L}_1,\overline{\mathcal{L}}_1\big]
=
\ell\,\frac{\partial}{\partial u},
\ \ \ \ \ \ \ \ \ \
\ell
:=
{\scriptstyle{\sqrt{-1}}}\,\big(
\mathcal{L}_1\big(\overline{A}_1\big)
-
\overline{\mathcal{L}}_1(A_1)
\big).
\endaligned
\]

\begin{lemma}
{\rm (\cite{ Merker-5-CR, Pocchiola-2013})}
The initial Lie structure of this frame consists of $10 =
{\textstyle{\binom{5}{2}}}$ brackets:
\[
\aligned
\big[\mathcal{T},\overline{\mathcal{L}}_1\big]
&
=
-\,\overline{P}\cdot\mathcal{T},
\\
\big[\mathcal{T},\overline{\mathcal{K}}\big]
&
=
\overline{\mathcal{L}}_1\big(\overline{k}\big)
\cdot
\mathcal{T}
+
\mathcal{T}
\big(\overline{k}\big)
\cdot
\overline{\mathcal{L}}_1,
\\
\big[\mathcal{T},\mathcal{L}_1\big]
&
=
-\,P\cdot\mathcal{T},
\\
\big[\mathcal{T},\mathcal{K}\big]
&
=
\mathcal{L}_1(k)\cdot\mathcal{T}
+
\mathcal{T}(k)\cdot\mathcal{L}_1,
\\
\big[\overline{\mathcal{L}}_1,\overline{\mathcal{K}}\big]
&
=
\overline{\mathcal{L}}_1\big(\overline{k}\big)\cdot
\overline{\mathcal{L}}_1,
\\
\big[\overline{\mathcal{L}}_1,\mathcal{L}_1\big]
&
=
{\scriptstyle{\sqrt{-1}}}\,\mathcal{T},
\\
\big[\overline{\mathcal{L}}_1,\mathcal{K}\big]
&
=
\overline{\mathcal{L}}_1(k)
\cdot\mathcal{L}_1,
\\
\big[\overline{\mathcal{K}},\mathcal{L}_1\big]
&
=
-\,\mathcal{L}_1\big(\overline{k}\big)
\cdot\overline{\mathcal{L}}_1,
\\
\big[\overline{\mathcal{K}},\mathcal{K}\big]
&
=
0,
\\
\big[\mathcal{L}_1,\mathcal{K}\big]
&
=
\mathcal{L}_1(k)\cdot\mathcal{L}_1,
\endaligned
\]
in terms of the $2$ fundamental functions:
\[
k:=
-\,
\frac{
\mathcal{L}_2\big(\overline{A_1}\big)
-
\overline{\mathcal{L}}_1\big(A_2\big)
}{
\mathcal{L}_1\big(\overline{A_1}\big)
-
\overline{\mathcal{L}}_1\big(A_1\big)
},
\ \ \ \ \ \ \ \ \ \ \ \ \
P
:=
\frac{\ell_{z_1}+A_1\,\ell_u-\ell\,A_{1,u}}{
{\scriptstyle{\sqrt{-1}}}\,\big(
\mathcal{L}_1\big(\overline{A}_1\big)
-
\overline{\mathcal{L}}_1\big(A_1\big)
\big)
}
\]
(in~\cite{ Pocchiola-2013}, $M^5$ is graphed as $u = F ( x_1, y_1, x_2, y_2, 
v)$ instead, hence $P$ changes).
\end{lemma}

Introduce then the coframe:
\[
\big\{
\rho_0,\,\overline{\kappa}_0,\,\overline{\zeta}_0,\,
\kappa_0,\,\zeta_0
\big\}
\]
which is dual to the frame:
\[
\big\{
\mathcal{T},\,
\overline{\mathcal{L}}_1,\,
\overline{\mathcal{K}},\,
\mathcal{L}_1,\,
\mathcal{K}
\big\},
\]
the notations being the same as in Proposition~\ref{L1-K-kappa-0}, 
that is to say:
\[
\aligned
\rho_0
&
=
\frac{du-A_1dz_1-A_2dz_2
-\overline{A}_1d\overline{z}_1
-\overline{A}_2d\overline{z}_2}{\ell},
\\
\kappa_0
&
=
dz_1-k\,dz_2,
\\
\zeta_0
&
=
dz_2.
\endaligned
\]
The initial Darboux structure is:
\[
\aligned 
d\rho_0
&
=
\overline{P}
\cdot
\rho_0\wedge\overline{\kappa}_0
-
\overline{\mathcal{L}}_1\big(\overline{k}\big)
\cdot
\rho_0\wedge\overline{\zeta}_0
+
P
\cdot
\rho_0\wedge\kappa_0
-
\mathcal{L}_1(k)
\cdot
\rho_0\wedge\zeta_0
+
{\scriptstyle{\sqrt{-1}}}\,
\kappa_0\wedge\overline{\kappa}_0,
\\
d\overline{\kappa}_0
&
=
-\,\mathcal{T}\big(\overline{k}\big)
\cdot
\rho_0\wedge\overline{\zeta}_0
-
\overline{\mathcal{L}}_1\big(\overline{k}\big)
\cdot
\overline{\kappa}_0
\wedge
\overline{\zeta}_0
+
\mathcal{L}_1\big(\overline{k}\big)
\cdot
\overline{\zeta}_0
\wedge
\kappa_0,
\\
d\overline{\zeta}_0
&
=
0,
\\
d\kappa_0
&
=
-\,\mathcal{T}(k)
\cdot
\rho_0\wedge\zeta_0
-
\overline{\mathcal{L}}_1(k)
\cdot
\overline{\kappa}_0\wedge\zeta_0
-
\mathcal{L}_1(k)
\cdot
\kappa_0\wedge\zeta_0,
\\
d\zeta_0
&
=
0.
\endaligned
\]
The initial associated $G$-structure is:
\[
{\sf G}_{\text{\sf IV}_{\text{\sf 2}}}^{\sf initial}
\,:=\,
\left\{
\left(\!\!
\begin{array}{ccccc}
{\sf c} & 0 & 0 & 0 & 0
\\
{\sf b} & {\sf a} & 0 & 0 & 0
\\
0 & 0 & \overline{\sf c} & 0 & 0
\\
0 & 0 & \overline{\sf b} & \overline{\sf a} & 0
\\
{\sf e} & {\sf d} & \overline{\sf e} & \overline{\sf d} & 
{\sf a}\overline{\sf a}
\end{array}
\!\!\right)
\,\in\,
\mathcal{M}_{5\times 5}(\C)\,\colon\,\,
{\sf a},\,{\sf c}\,\in\,\C\backslash\{0\},\,\,
{\sf b},\,{\sf d},\,{\sf e}
\,\in\,\C
\right\}.
\]

\begin{theorem}
\label{Pocchiola-IV-2}
{\rm ({\sc Pocchiola}, \cite{ Pocchiola-2013})} Two fundamental explicit
invariants both having denominators related to the nondegeneracy 
of the Freeman form:
\[
\aligned
W
&
:=
\frac{2}{3}\,
\frac{\mathcal{L}_1\big(\overline{\mathcal{L}}_1(k)\big)}
{\overline{\mathcal{L}}_1(k)}
+
\frac{2}{3}\,
\frac{\mathcal{L}_1\big(\mathcal{L}_1(\overline{k})\big)}
{\mathcal{L}_1(\overline{k})}
+
\\
&
\ \ \ \ \
+
\frac{1}{3}\,
\frac{\overline{\mathcal{L}}_1\big(\overline{\mathcal{L}}_1(k)\big)\,
\mathcal{K}\big(\overline{\mathcal{L}}_1(k)\big)}
{\overline{\mathcal{L}}_1(k)^3}
-
\frac{1}{3}\,
\frac{\mathcal{K}\big(\overline{\mathcal{L}}_1
\big(\overline{\mathcal{L}}_1(k)\big)\big)}
{\overline{\mathcal{L}}_1(k)^2}
+
\frac{i}{3}\,
\frac{\mathcal{T}(k)}
{\overline{\mathcal{L}}_1(k)},
\endaligned
\]
and:
\[
\aligned
J
&
:=
\frac{5}{18}\,
\frac{\mathcal{L}_1\big(\mathcal{L}_1(\overline{k})\big)^2}
{\mathcal{L}_1(\overline{k})^2}
+
\frac{1}{3}\,
P\,\mathcal{L}_1(P)
-
\frac{1}{9}\,P^2\,
\frac{\mathcal{L}_1\big(\mathcal{L}_1(\overline{k})\big)}
{\mathcal{L}_1(\overline{k})}
+
\frac{20}{27}\,
\frac{\mathcal{L}_1\big(\mathcal{L}_1(\overline{k})\big)^3}
{\mathcal{L}_1(\overline{k})^3}
\,-
\\
&
\ \ \ \ \
-\,
\frac{5}{6}\,
\frac{\mathcal{L}_1\big(\mathcal{L}_1(\overline{k})\big)\,
\mathcal{L}_1\big(\mathcal{L}_1\big(\mathcal{L}_1
(\overline{k})\big)\big)}
{\mathcal{L}_1(\overline{k})^2}
+
\frac{1}{6}\,
\frac{\mathcal{L}_1\big(\mathcal{L}_1(\overline{k})\big)\,
\mathcal{L}_1(P)}
{\mathcal{L}_1(\overline{k})}
\,-
\\
&
\ \ \ \ \
-\,
\frac{1}{6}\,
P\,
\frac{\mathcal{L}_1\big(\mathcal{L}_1\big(\mathcal{L}_1
(\overline{k})\big)\big)}
{\mathcal{L}_1(\overline{k})}
\,-
\frac{2}{27}\,P^3
-
\frac{1}{6}\,
\mathcal{L}_1\big(\mathcal{L}_1(P)\big)
+
\frac{1}{6}\,
\frac{\mathcal{L}_1\big(\mathcal{L}_1\big(\mathcal{L}_1\big(\mathcal{L}_1
(\overline{k})\big)\big)\big)}
{\mathcal{L}_1(\overline{k})},
\endaligned
\]
occur in the local biholomorphic equivalence problem for Class $\text{\sf
IV}_{\text{\sf 2}}$ real analytic hypersurfaces $M^5 \subset \C^3$.
Such a $M^5$ is locally biholomorphically equivalent to the light cone
model:
\def\theequation{${\sf LC}$}
\begin{equation}
u
=
\frac{z_1\overline{z}_1+
\frac{1}{2}\,z_1z_1\overline{z}_2
+
\frac{1}{2}\,z_2\overline{z}_1\overline{z}_1}{
1-z_2\overline{z}_2}
\underset{\text{\rm locally}\atop
\text{\rm by~\cite{Fels-Kaup-2007}}}{\cong}
({\sf Re}\,z_1')^2
-
({\sf Re}\,z_2')^2
-
({\sf Re}\,z_3')^2,
\end{equation}
having $10$-dimensional symmetry group ${\sf Aut}_{ CR} 
({\sf LC}) \cong {\sf Sp}( 4, \R)$, 
if and only if:
\[
0
\equiv
W
\equiv
J.
\]
Moreover, if either $W \not\equiv 0$ or $J \not\equiv 0$, 
an absolute parallelism is constructed on $M$ (after
relocalization), and in this
case, the local Lie group of $\mathcal{ C}^\omega$ CR automorphisms
of $M$ always has:
\[
{\sf dim}\,
{\sf Aut}_{CR}(M)
\,\leqslant\,
5.
\]
\end{theorem}

The latter dimension drop was obtained by Fels-Kaup (\cite{
Fels-Kaup-2008}) under the assumption that ${\sf Aut}_{ CR} (M)$ is
locally transitive, while Cartan's method
embraces {\em all} Class $\text{\sf IV}_{\text{\sf 2}}$ 
hypersurfaces $M^5 \subset \C^3$. 

Reduction to an absolute parallelism on a $10$-dimensional bundle $N^{
10} \to M$ has been obtained previously by Isaev-Zaitsev (\cite{
Isaev-Zaitsev-2013}) and Medori-Spiro (\cite{Medori-Spiro-2012}).
The explicitness of $W$ and $J$, the equivalence bifurcation $W \not
\equiv 0$ or $J \not \equiv 0$ and the dimension drop $10 \to 5$
provide a complementary aspect. Furthermore, here is an application
of the explicit rational expressions of $J$ and $W$ in the spirit of
Theorem~\ref{propagation-sphericity}.

\begin{corollary}
Let $M^5 \subset \C^3$ be a connected $\mathcal{ C}^\omega$
hypersurface whose Levi form is of rank $1$ at Zariski-generic points,
possibly of rank $0$ somewhere, and whose Freeman form is also
nondegenerate at Zariski-generic points. If $M$ is locally
biholomorphic to the light cone model $({\sf LC})$ at some Freeman
nondegenerate point, then $M$ is also locally biholomorphic to
$({\sf LC})$ at every other Freeman nondegenerate
point.
\end{corollary}

\medskip\noindent{\bf Equivalences of remaining Classes $\text{\sf II}$,
$\text{\sf III}_{\text{\sf 1}}$, $\text{\sf III}_{\text{\sf 2}}$,
$\text{\sf IV}_{\text{\sf 1}}$.} Class $\text{\sf II}$ has been
treated optimally by Beloshapka-Ezhov-Schmalz (\cite{
Beloshapka-Ezhov-Schmalz-2007}) who directly constructed a Cartan
connection on a principal bundle $N^5 \to M^4$ with fiber $\cong \R$.
Recently, Pocchiola (\cite{ Pocchiola-2014a}) provided
an alternative construction the elements of which are
computed deeper. 

\smallskip

Class $\text{\sf IV}_{\text{\sf 1}}$ is reduced to an absolute
parallelism with a Cartan connection
by Chern-Moser (\cite{ Chern-Moser-1974}) inspired by
Hachtroudi (\cite{ Hachtroudi-1937}), though not explicity in terms
of a local graphing function (question still open).

\smallskip

Recently, jointly with Sabzevari, Class $\text{\sf III}_{\text{\sf 1}}$
has been recently settled. Beloshapka ({\em see} also~\cite{
Improved-2013}), proved that the Lie algebra $\mathfrak{ aut}_{ CR}
= 2\, {\rm Re}\, \mathfrak{hol}$ of ${\rm Aut}_{ CR}$ of the cubic:
\[
{\text{\small\sf Model 
$\text{\sf III}_{\text{\sf 1}}$:}}
\ \ \ \ \ \ \ \ \ 
v_1
=
z\overline{z},
\ \ \ \ \ \ \ \ \
v_2
=
z^2\overline{z}+z\overline{z}^2,
\ \ \ \ \ \ \ \ \
v_3
=
{\scriptstyle{\sqrt{-1}}}\,\big(z^2\overline{z}-z\overline{z}^2\big),
\]
is $7$-dimensional generated by:
\[
\aligned T & :=
\partial_{w_1},
\\
S_1 & :=
\partial_{w_2},
\\
S_2 & :=
\partial_{w_3},
\\
 L_1 & :=
\partial_z
+ (2{\scriptstyle{\sqrt{-1}}}\, z)\,\partial_{w_1} + (2{\scriptstyle{\sqrt{-1}}}\,
z^2+4w_1)\,\partial_{w_2} +
2z^2\,\partial_{w_3},
\\
L_2 & := {\scriptstyle{\sqrt{-1}}}\,\partial_z + (2\,z)\,\partial_{w_1} +
(2z^2)\,\partial_{w_2} - (2{\scriptstyle{\sqrt{-1}}}\,z^2-4w_1)\,\partial_{w_3},
\\
D & := z\,\partial_z + 2w_1\,\partial_{w_1} + 3w_2\,\partial_{w_2} +
3w_3\,\partial_{w_3},
\\
R & := {\scriptstyle{\sqrt{-1}}}\,z\,\partial_z - w_3\,\partial_{w_2} + w_2\,\partial_{w_3}.
\endaligned
\]

Hence for a CR-generic
$M^5 \subset \C^4$ belonging to Class
$\text{\sf III}_{\text{\sf 1}}$ graphed as: 
\[
v_1
=
\varphi_1(x,y,u_1,u_2,u_3),
\ \ \ \ \ \ \ \ \ \ \ \ \ \ \
v_2
=
\varphi_2(x,y,u_1,u_2,u_3),
\ \ \ \ \ \ \ \ \ \ \ \ \ \ \
v_3
=
\varphi_3(x,y,u_1,u_2,u_3),
\]
reduction to an absolute parallelism on a certain bundle $N^7 \to M^5$
can be expected, and in fact, similarly as in
Theorem~\ref{Pocchiola-IV-2}, finer equivalence bifurcations will
occur.

Lemma~\ref{L-explicit} showed that a generator for $T^{1,0} M$ is:
\[
\mathcal{L}
=
\frac{\partial}{\partial z}
+
\frac{\Lambda_1}{\Delta}\,\frac{\partial}{\partial u_1}
+
\frac{\Lambda_2}{\Delta}\,\frac{\partial}{\partial u_2}
+
\frac{\Lambda_3}{\Delta}\,\frac{\partial}{\partial u_3},
\]
where:
\[
\Delta
:=
\left\vert\!
\begin{array}{ccc}
{\scriptstyle{\sqrt{-1}}}+\varphi_{1,u_1} & \varphi_{1,u_2} & \varphi_{1,u_3}\\
\varphi_{2,u_1} & {\scriptstyle{\sqrt{-1}}}+\varphi_{2,u_2} & \varphi_{2,u_3}\\
\varphi_{3,u_1} & \varphi_{3,u_2} & {\scriptstyle{\sqrt{-1}}}+\varphi_{3,u_3}
\end{array}
\!\right\vert,
\]
and where:
\[
\Lambda_1
:=
\left\vert\!
\begin{array}{ccc}
-\,\varphi_{1,z} & \varphi_{1,u_2} & \varphi_{1,u_3}\\
-\,\varphi_{2,z} & {\scriptstyle{\sqrt{-1}}}+\varphi_{2,u_2} & \varphi_{2,u_3}\\
-\,\varphi_{3,z} & \varphi_{3,u_2} & {\scriptstyle{\sqrt{-1}}}+\varphi_{3,u_3}
\end{array}
\!\right\vert,
\]
with similar $\Lambda_2$, $\Lambda_3$. By definition, on a Class
$\text{\sf III}_{\text{\sf 1}}$ CR-generic $M^5 \subset \C^4$ the
fields:
\[
\Big\{ \overline{\mathcal{S}},\, \mathcal{S},\, \mathcal{T},\,
\overline{\mathcal{L}},\, \mathcal{L} \Big\},
\]
where:
\[
\mathcal{T} 
:=
{\scriptstyle{\sqrt{-1}}}\,\big[\mathcal{L},\overline{\mathcal{L}}\big],
\ \ \ \ \ \ \ \ \ \ \ \
\mathcal{S} 
:= 
\big[\mathcal{L},\,\mathcal{T}\big],
\ \ \ \ \ \ \ \ \ \ \ \
\overline{\mathcal{S}} 
:=
\big[\overline{\mathcal{L}},\,\mathcal{T}\big],
\]
make up a frame for $\C \otimes_\R TM^5$.
Computing these (iterated) Lie brackets, there are certain
coefficient-polynomials:
\[
\aligned
\Upsilon_i
&
=
\Upsilon_i
\Big(
\varphi_{1,x^jy^ku_1^{l_1}u_2^{l_2}u_3^{l_3}},\,\,
\varphi_{2,x^jy^ku_1^{l_1}u_2^{l_2}u_3^{l_3}},\,\,
\varphi_{3,x^jy^ku_1^{l_1}u_2^{l_2}u_3^{l_3}}
\Big)_{1\leqslant j+k+l_1+l_2+l_3\leqslant 2}
\ \ \ \ \ \ \ 
{\scriptstyle{(i\,=\,1,\,2,\,3)}}
,
\\
\Pi_i
&
=
\Pi_i
\Big(
\varphi_{1,x^jy^ku_1^{l_1}u_2^{l_2}u_3^{l_3}},\,\,
\varphi_{2,x^jy^ku_1^{l_1}u_2^{l_2}u_3^{l_3}},\,\,
\varphi_{3,x^jy^ku_1^{l_1}u_2^{l_2}u_3^{l_3}}
\Big)_{1\leqslant j+k+l_1+l_2+l_3\leqslant 3}
\ \ \ \ \ \ \ 
{\scriptstyle{(i\,=\,1,\,2,\,3)}},
\endaligned
\]
so that (mind exponents in denominators):
\[
\aligned
\mathcal{T}
&
=
\frac{\Upsilon_1}{\Delta^2\,\overline{\Delta}^2}\,
\frac{\partial}{\partial u_1}
+
\frac{\Upsilon_2}{\Delta^2\,\overline{\Delta}^2}\,
\frac{\partial}{\partial u_2}
+
\frac{\Upsilon_3}{\Delta^2\,\overline{\Delta}^2}\,
\frac{\partial}{\partial u_3},
\\
\mathcal{S}
&
=
\frac{\Pi_1}{\Delta^4\,\overline{\Delta}^3}\,
\frac{\partial}{\partial u_1}
+
\frac{\Pi_2}{\Delta^4\,\overline{\Delta}^3}\,
\frac{\partial}{\partial u_2}
+
\frac{\Pi_3}{\Delta^4\,\overline{\Delta}^3}\,
\frac{\partial}{\partial u_3},
\\
\overline{\mathcal{S}}
&
=
\frac{\overline{\Pi}_1}{\Delta^3\,\overline{\Delta}^4}\,
\frac{\partial}{\partial u_1}
+
\frac{\overline{\Pi}_2}{\Delta^3\,\overline{\Delta}^4}\,
\frac{\partial}{\partial u_2}
+
\frac{\overline{\Pi}_3}{\Delta^3\,\overline{\Delta}^4}\,
\frac{\partial}{\partial u_3}.
\endaligned
\]

\medskip\noindent{\bf Explicitness obstacle.}
The expansions of $\Upsilon_1$, $\Upsilon_2$, $\Upsilon_3$ as
polynomials in their $3\cdot 20$ variables incorporate $41\,964$
monomials
while those of $\Pi_1$, $\Pi_2$, $\Pi_3$ as polynomials in their $3
\cdot 55$ variables would incorporate more than (no
computer succeeded) $100\,000\,000$ terms.

\smallskip

{\em Hence renouncement to full expliciteness is necessary.}

\smallskip

Between the $5$ fields
$\big\{ \overline{ \mathcal{ S}}, \mathcal{ S}, \mathcal{ T},
\overline{ \mathcal{ L}}, \mathcal{ L} \big\}$, there are $10 =
\binom{ 5}{ 2}$ Lie brackets. Assign therefore formal names to the
uncomputable appearing coefficient-functions.

\begin{lemma}
(\cite{ Merker-Sabzevari-5-cubic})
In terms of $5$ fundamental coefficient-functions:
\[
P,\ \ \ \ \
Q,\ \ \ \ \
R,\ \ \ \ \
A,\ \ \ \ \
B,
\]
the 10 Lie bracket relations write as:
\[
\aligned
\big[\overline{\mathcal{S}},\mathcal{S}\big]
&
=
\overline{K}_{\sf rpl}\cdot\overline{\mathcal{S}}
-
K_{\sf rpl}\cdot\mathcal{S}
-
{\scriptstyle{\sqrt{-1}}}\,J_{\sf rpl}\cdot\mathcal{T},
\\
\big[\overline{\mathcal{S}},\mathcal{T}\big]
&
=
-\,
\overline{F}_{\sf rpl}\cdot\overline{\mathcal{S}}
-
\overline{G}_{\sf rpl}\cdot\mathcal{S}
-
\overline{E}_{\sf rpl}\cdot\mathcal{T},
\\
\big[\overline{\mathcal{S}},\overline{\mathcal{L}}\big]
&
=
-\,
\overline{Q}\cdot\overline{\mathcal{S}}
-
\overline{R}\cdot\mathcal{S}
-
\overline{P}\cdot\mathcal{T},
\\
\big[\overline{\mathcal{S}},\mathcal{L}\big]
&
=
-\,
\overline{B}\cdot\overline{\mathcal{S}}
-
B\cdot\mathcal{S}
-
A\cdot\mathcal{T},
\\
\big[\mathcal{S},\,\mathcal{T}\big]
&
=
-\,
G_{\sf rpl}\cdot\overline{\mathcal{S}}
-
F_{\sf rpl}\cdot\mathcal{S}
-
E_{\sf rpl}\cdot\mathcal{T},
\\
\big[\mathcal{S},\,\overline{\mathcal{L}}\big]
&
=
-\,
\overline{B}\cdot\overline{\mathcal{S}}
-
B\cdot\mathcal{S}
-
A\cdot\mathcal{T},
\\
\big[\mathcal{S},\,\mathcal{L}\big]
&
=
-\,
R\cdot\overline{\mathcal{S}}
-
Q\cdot\mathcal{S}
-
P\cdot\mathcal{T},
\\
\big[\mathcal{T},\,\overline{\mathcal{L}}\big]
&
=
-\,
\overline{\mathcal{S}},
\\
\big[\mathcal{T},\,\mathcal{L}\big]
&
=
\,-\,\mathcal{S},
\\
\big[\overline{\mathcal{L}},\,\mathcal{L}\big]
&
=
{\scriptstyle{\sqrt{-1}}}\,\mathcal{T},\,\,
\endaligned
\]
the coefficient-functions $E_{\sf rpl}$, $G_{\sf rpl}$, $H_{\sf rpl}$,
$J_{\sf rpl}$, $K_{\sf rpl}$ being secondary:
\[
\aligned
E_{\sf rpl}
&
=
{\scriptstyle{\sqrt{-1}}}\,
\Big(
\mathcal{L}(A)
-
\overline{\mathcal{L}}(P)
+
A\overline{B}
+
BP
-
AQ
-
\overline{P}R
\Big),
\\
F_{\sf rpl}
&
=
{\scriptstyle{\sqrt{-1}}}\,
\Big(
\mathcal{L}(B)
-
\overline{\mathcal{L}}(Q)
+
A
+
B\overline{B}
-
R\overline{R}
\Big),
\\
G_{\sf rpl}
&
=
{\scriptstyle{\sqrt{-1}}}\,
\Big(
\mathcal{L}(\overline{B})
-
\overline{\mathcal{L}}(R)
+
\overline{B}\overline{B}
+
BR
-
P
-
\overline{B}Q
-
R\overline{Q}
\Big),
\endaligned
\]
with similar, longer expressions for $J_{\sf rpl}$, $K_{\sf rpl}$.
\end{lemma}

Introduce then the coframe:
\[
\big\{
\overline{\sigma_0},\,
\sigma_0,\,
\rho_0,\,
\overline{\zeta_0},\,
\zeta_0\big\},
\]
which is dual to the frame:
\[
\big\{
\overline{\mathcal{S}},\,
\mathcal{S},\,
\mathcal{T},\,
\overline{\mathcal{L}},\,
\mathcal{L}\big\}.
\]
Organize the ten Lie brackets as a convenient auxiliary array:
\[
\footnotesize
\begin{array}{cccccccccccc}
& & \overline{\mathcal{S}} & & \mathcal{S} & & \mathcal{T} & &
\overline{\mathcal{L}} & & \mathcal{L}
\\
& & \boxed{d\overline{\sigma_0}} & & \boxed{d\sigma_0} & &
\boxed{d\rho_0} & & \boxed{d\overline{\zeta_0}} & & \boxed{d\zeta_0}
\\
\big[\overline{\mathcal{S}},\,\mathcal{S}\big] & = &
\overline{K}_{\sf rpl}\cdot\overline{\mathcal{S}} & + &
-\,K_{\sf rpl}\cdot\mathcal{S} & + & -{\scriptstyle{\sqrt{-1}}}\,J_{\sf rpl}\cdot\mathcal{T} & + & 0 &
+ & 0 & \boxed{\overline{\sigma_0}\wedge\sigma_0}
\\
\big[\overline{\mathcal{S}},\,\mathcal{T}\big] & = &
-\,\overline{F}_{\sf rpl}\cdot\overline{\mathcal{S}} & + &
-\,\overline{G}_{\sf rpl}\cdot\mathcal{S} & + &
-\,\overline{E}_{\sf rpl}\cdot\mathcal{T} & + & 0 & + & 0 &
\boxed{\overline{\sigma_0}\wedge\rho_0}
\\
\big[\overline{\mathcal{S}},\,\overline{\mathcal{L}}\big] & = &
-\,\overline{Q}\cdot\overline{\mathcal{S}} & + &
-\,\overline{R}\cdot\mathcal{S} & + & -\,\overline{P}\cdot\mathcal{T}
& + & 0 & + & 0 & \boxed{\overline{\sigma_0}\wedge\overline{\zeta_0}}
\\
\big[\overline{\mathcal{S}},\,\mathcal{L}\big] & = &
-\,\overline{B}\cdot\overline{\mathcal{S}} & + &
-\,{B}\cdot\mathcal{S} & + & -A\cdot\mathcal{T} & + & 0 & + & 0 &
\boxed{\overline{\sigma_0}\wedge\zeta_0}
\\
\big[\mathcal{S},\,\mathcal{T}\big] & = &
-\,G_{\sf rpl}\cdot\overline{\mathcal{S}} & + & -\,F_{\sf rpl}\cdot 
\mathcal{S} & +
& -\,E_{\sf rpl}\cdot\mathcal{T} & + & 0 & + & 0 &
\boxed{\sigma_0\wedge\rho_0}
\\
\big[\mathcal{S},\,\overline{\mathcal{L}}\big] & = &
-\,\overline{B}\cdot\overline{\mathcal{S}} & + & -\,B\cdot
\mathcal{S} & + & -\,A\cdot\mathcal{T} & + & 0 & + & 0 &
\boxed{\sigma_0\wedge\overline{\zeta_0}}
\\
\big[\mathcal{S},\,\mathcal{L}\big] & = &
-\,R\cdot\overline{\mathcal{S}} & + & -\,Q\cdot\mathcal{S} & + &
-\,P\cdot\mathcal{T} & + & 0 & + & 0 & \boxed{\sigma_0\wedge\zeta_0}
\\
\big[\mathcal{T},\,\overline{\mathcal{L}}\big] & = &
-\,\overline{\mathcal{S}} & + & 0 & + & 0 & + & 0 & + & 0 &
\boxed{\rho_0\wedge\overline{\zeta_0}}
\\
\big[\mathcal{T},\,\mathcal{L}\big] & = & 0 & + & -\,\mathcal{S} & +
& 0 & + & 0 & + & 0 & \boxed{\rho_0\wedge\zeta_0}
\\
\big[\overline{\mathcal{L}},\,\mathcal{L}\big] & = & 0 & + & 0 & + &
{\scriptstyle{\sqrt{-1}}}\,\mathcal{T} & + & 0 & + & 0 &
\boxed{\overline{\zeta_0}\wedge\zeta_0}
\end{array}
\]

Read {\em vertically} and put an overall minus sign to
get the {\sl initial Darboux structure}:
\[
\aligned d\overline{\sigma}_0 & =
-\,\overline{K}_{\sf rpl} \cdot \overline{\sigma}_0\wedge\sigma_0 +
\overline{F}_{\sf rpl}\cdot \overline{\sigma}_0\wedge\rho_0 +
\overline{Q}\cdot \overline{\sigma}_0\wedge\overline{\zeta}_0 +
\overline{B}\cdot \overline{\sigma}_0\wedge\zeta_0 +
\\
& \ \ \ \ \ + G_{\sf rpl}\cdot \sigma_0\wedge\rho_0 + \overline{B}\cdot
\sigma_0\wedge\overline{\zeta}_0 + R\cdot \sigma_0\wedge\zeta_0 +
\rho_0\wedge\overline{\zeta}_0,
\\
d\sigma_0 & = K_{\sf rpl}\cdot \overline{\sigma}_0\wedge\sigma_0 +
\overline{G}_{\sf rpl}\cdot \overline{\sigma}_0\wedge\rho_0 +
\overline{R}\cdot \overline{\sigma}_0\wedge\overline{\zeta}_0 +
{B}\cdot \overline{\sigma}_0\wedge\zeta_0 +
\\
& \ \ \ \ \ + F_{\sf rpl}\cdot \sigma_0\wedge\rho_0 + B\cdot
\sigma_0\wedge\overline{\zeta}_0 + Q\cdot \sigma_0\wedge\zeta_0 +
\rho_0\wedge\zeta_0,
\\
d\rho_0 & = {\scriptstyle{\sqrt{-1}}}\,J_{\sf rpl}\cdot \overline{\sigma}_0\wedge\sigma_0 +
\overline{E}_{\sf rpl}\cdot \overline{\sigma}_0\wedge\rho_0 +
\overline{P}\cdot \overline{\sigma}_0\wedge\overline{\zeta}_0 +
A\cdot \overline{\sigma}_0\wedge\zeta_0 +
\\
& \ \ \ \ \ + E_{\sf rpl}\cdot \sigma_0\wedge\rho_0 + A\cdot
\sigma_0\wedge\overline{\zeta}_0 + P\cdot \sigma_0\wedge\zeta_0 -
{\scriptstyle{\sqrt{-1}}}\,\overline{\zeta}_0\wedge\zeta_0,
\\
d\overline{\zeta}_0 & = 0,
\\
d\zeta_0 & = 0.
\endaligned
\]

The initial $G$-structure is:
\[
{\sf G}_{\text{\sf III}_{\text{\sf 1}}}^{\sf initial}
\,:=\,
\left\{
\left(\!\!
\begin{array}{ccccc}
{\sf a} & 0 & 0 & 0 & 0
\\
0 & \overline{\sf a} & 0 & 0 & 0
\\
{\sf b} & \overline{\sf b} & {\sf a}\overline{\sf a} & 0 & 0
\\
{\sf e} & {\sf d} & {\sf c} & {\sf a}{\sf a}\overline{\sf a} & 0
\\
\overline{\sf d} & \overline{\sf e} & \overline{\sf c} & 0 &
{\sf a}\overline{\sf a}\overline{\sf a}
\end{array}
\!\!\right)
\,\in\,
\mathcal{M}_{5\times 5}(\C)\,\colon\,\,
{\sf a}\,\in\,\C\backslash\{0\},\,\,
{\sf b},\,{\sf c},\,{\sf d},\,{\sf e}\,\in\,\C
\right\}.
\]
The lifted coframe is:
\[
\aligned \left(\!\!
\begin{array}{c}
\overline{\sigma}
\\
\sigma
\\
\rho
\\
\overline{\zeta}
\\
\zeta
\end{array}
\!\!\right) 
:=
\left(\!\!
\begin{array}{ccccc}
{\sf a}\overline{\sf a}\overline{\sf a} & 0 & 0 & 0 & 0
\\
0 & {\sf a}{\sf a}\overline{\sf a} & 0 & 0 & 0
\\
\overline{\sf c} & {\sf c} & {\sf a}\overline{\sf a} & 0 & 0
\\
\overline{\sf e} & {\sf d} & \overline{\sf b} & \overline{\sf a} & 0
\\
\overline{\sf d} & {\sf e} & {\sf b} & 0 & {\sf a}
\end{array}
\!\!\right)
\left(\!\!
\begin{array}{c}
\overline{\sigma_0}
\\
\sigma_0
\\
\rho_0
\\
\overline{\zeta_0}
\\
\zeta_0
\end{array}
\!\!\right).
\endaligned
\]

Performing absorption of torsion and normalization of group variables
thanks to remaining essential torsion (\cite{ Olver-1995,
Merker-Sabzevari-5-cubic}), the coefficient-function $R$ is
an invariant which creates bifurcation. Even in terms of $P$, $Q$,
$R$, $A$, $B$, the expressions of some of the curvatures happen to be
large and the study of their mutual independencies 
requires to take account of iterated Jacobi identities, an aspect
of the subject which remains invisible in non-parametric Cartan method.

\begin{theorem}
{\rm ({\sc M.-Sabzevari}, \cite{ Merker-Sabzevari-5-cubic})}
Within the branch $R = 0$, the biholomorphic equivalence problem for
$M^5 \subset \C^4$ in Class $\text{\sf III}_{\text{\sf 1}}$ reduces to
various absolute parallelisms namely to $\{ e\}$-structures on certain
manifolds of dimension $6$, or directly on the $5$-dimensional basis
$M$, unless all existing essential curvatures vanish identically, in
which case $M$ is (locally) biholomorphic to the cubic model
with a characterization of such a condition being
explicit in terms of the five fundamental functions $P$, $Q$, $R$,
$A$, $B$.

Within the branch $R \not\equiv 0$, reduction to an absolute
parallelism on the $5$-dimensional basis $M$ always takes place,
whence:
\[
{\sf dim}\,{\sf Aut}_{CR}(M)
\,\leqslant\,
5.
\]
\end{theorem}

\smallskip

Class $\text{\sf III}_{\text{\sf 2}}$ was recently settled also.

\begin{theorem}
{\rm ({\sc Pocchiola}, \cite{ Pocchiola-2014b})}
If $M^5 \subset \C^4$ is a local $\mathcal{ C}^\omega$ CR-generic 
submanifold belonging to Class $\text{\sf III}_{\text{\sf 1}}$,
then there exists a $6$-dimensional principal bundle
$P^6 = M^5 \times \R^*$ and there exists a coframe on $P^6$:
\[
\varpi
:=
\big(
\lambda,\,\tau,\,\sigma,\,\rho,\,\zeta,\,\overline{\zeta}
\big)
\]
such that any local $\mathcal{ C}^\omega$
CR-diffeomorphism ${\sf H}_M \colon\, M \to M$ lifts as a bundle
isomorphism $\widehat{\sf H}_M \colon\, P \to P$ which satisfies
${\sf H}^* ( \varpi) = \varpi$. Moreover, the structure 
equations of $\varpi$ on $P$ are of the form:
\[
\aligned
d\tau
&
=
4\,\lambda\wedge\tau
+
I_1\,\tau\wedge\zeta
-
I_1\,\tau\wedge\overline{\zeta}
+
3\,I_1\,\sigma\wedge\rho
+
\sigma\wedge\zeta
+
\sigma\wedge\overline{\zeta},
\\
d\sigma
&
=
3\,\lambda\wedge\sigma
+
I_2\,\tau\wedge\rho
+
I_3\,\tau\wedge\zeta
+
\overline{I}_3\,\tau\wedge\overline{\zeta}
+
I_4\,\sigma\wedge\rho
\,-
\\
&
\ \ \ \ \ \ \ \ \ \ \ \ \ \ \ \ \ 
-\,
{\textstyle{\frac{1}{2}}}\,I_1\,
\sigma\wedge\zeta
+
{\textstyle{\frac{1}{2}}}\,I_1\,
\sigma\wedge\overline{\zeta}
+
\rho\wedge\zeta
+
\rho\wedge\overline{\zeta},
\\
d\rho
&
=
2\,\lambda\wedge\rho
+
I_5\,\tau\wedge\sigma
+
I_6\,\tau\wedge\rho
+
I_7\,\tau\wedge\zeta
+
\overline{I}_7\,\tau\wedge\overline{\zeta}
+
I_8\,\sigma\wedge\rho
\,+
\\
&
\ \ \ \ \ \ \ \ \ \ \ \ \ \ \ \ \ 
+
I_9\,\sigma\wedge\zeta
+
\overline{I}_9\,\sigma\wedge\overline{\zeta}
-
{\textstyle{\frac{1}{2}}}\,I_1\,
\rho\wedge\zeta
+
{\textstyle{\frac{1}{2}}}\,I_1\,
\rho\wedge\overline{\zeta}
+
{\scriptstyle{\sqrt{-1}}}\,
\,\zeta\wedge\overline{\zeta},
\\
d\zeta
&
=
\lambda\wedge\zeta
+
I_{10}\,\tau\wedge\sigma
+
I_{11}\,\tau\wedge\rho
+
I_{12}\,\tau\wedge\zeta
+
I_{13}\,\tau\wedge\overline{\zeta}
+
I_{14}\,\sigma\wedge\rho
+
I_{15}\,\sigma\wedge\zeta,
\endaligned
\]
for function $I_{\scriptscriptstyle{\bullet}}$, 
$J_{{\scriptscriptstyle{\bullet}}
{\scriptscriptstyle{\bullet}}}$ on $P$ together with:
\[
d\lambda
=
\sum_{\nu,\mu}\,J_{\nu\mu}\,
\nu\wedge\mu
\ \ \ \ \ \ \ \ \ \ \ \ \
{\scriptstyle{\left(\mu,\,\nu\,=\,\tau,\,
\sigma,\,\rho,\,\zeta,\,\overline{\zeta}\right)}}.
\]
\end{theorem}

For both Classes $\text{\sf III}_{\text{\sf 1}}$ and $\text{\sf
III}_{\text{\sf 2}}$, there also exist canonical Cartan connections
naturally related to the final $\{ e\}$-structures (\cite{
Pocchiola-2014b, Merker-Pocchiola-Sabzevari-2014}). 

\smallskip

{\em These works complete the program of performing 
{\em parametric} Cartan equivalences
for all CR manifolds up to dimension $5$.}

\smallskip

In dimension~6, Ezhov-Isaev-Schmalz (\cite{Ezhov-Isaev-Schmalz-1999})
treated elliptic and hyperbolic $M^6 \subset \C^4$. A wealth of higher
dimensional biholomorphic equivalence problems exists, {\em e.g.}
(\cite{ Mamai-2009}) for CR-generic $M^{2 + c} \subset \C^{ 1 + c}$ in
relation to classification of nilpotent Lie algebras (\cite{
Goze-Remm}).

\section{Kobayashi hyperbolicity}

\begin{quotation}
The dominant theme is the {\em interplay} between the extrinsic projective
geometry of algebraic subvarieties of $\P^n(\C)$ and their 
intrinsic geometric features.
Phillip~{\sc Griffiths}.
\end{quotation}

In local Cartan theory, as seen in what precedes, denominators
therefore play a central role in the differential ring generated by
derivatives of the fundamental graphing functions
$\varphi_j$. Similarly, in arithmetics of rational numbers $p /
q$, like {\em e.g.} in multizeta calculus
(\cite{Merker-2012-polyzetas}) involving a wealth of nested
Cramer-type determinants, a growing complexity, potentially infinite,
exists, and in fact, the complexity of rational numbers {\em also
enters} high order covariant derivatives of Cartan curvatures, as an
expression of the unity of mathematics. It is now time to show how
{\em explicit rationality} also concerns the core of global
algebraic geometry.

\smallskip

Let $X$ be a compact complex $n$-dimensional projective manifold that
is of {\sl general type}, namely whose canonical bundle $K_X =
\Lambda^n T_X^*$ is {\sl big} in the sense that ${\sf dim} \, H^0 \big( X,
(K_X)^{\otimes m} \big) \geqslant c\, m^n$ when $m \to \infty$ for
some constant $c > 0$. It is known that smooth hypersurfaces $X^n
\subset \mathbb{ P}^{ n+1} ( \mathbb{ C})$ are so if and only if their
degree $d$ is $\geqslant n+3$, since the {\sl adjunction
formula} shows that $K_X \cong \mathcal{ O}_X ( d - n - 2)$ (the
related rationality aspects will be discussed later).

A conjecture of Green-Griffiths-Lang expects that all nonconstant
entire holomorphic curves $f \colon \mathbb{ C} \to X$ should in fact
land (be `canalized') inside a certain proper subvariety $Y
\subsetneqq X$. The current state of the art is still quite (very) far
from reaching such a statement in this optimality. Furthermore, a
companion Picard-type conjecture dating back to Kobayashi 1970 expects
that all entire curves valued in Zariski-generic hypersurfaces $X^n
\subset \mathbb{ P}^{ n+1} ( \mathbb{ C})$ of degree $d\geqslant 2n+1$
should necessarily be {\em constant}, and the state of the art is also
still quite far from understanding properly when this occurs, even in
dimension $2$, because of the lack of an appropriate explicit rational
theory.

The current general method towards these conjectures consists as a
first step in setting up a great number of nonzero differential
equations $P \big( f, f', f'', \dots, f^{ (\kappa)} \big) = 0$ of some
order $\kappa$ (necessarily $\geqslant n$) satisfied by all
nonconstant $f \colon \mathbb{ C} \to X$, and then as a second step,
in trying to {\em eliminate} from such numerous differential equations
the true derivatives $f', f'', \dots, f^{ ( \kappa)}$ in order to
receive certain purely algebraic nonzero equations $Q ( f) = 0$
involving no derivatives anymore.

In 1979, by computing the Euler-Poincar\'e characteristic of a natural
vector bundle nowadays called the {\sl Green-Griffiths jet bundle}
$\mathcal{E}_{\kappa, m}^{\rm GG} T_X^* \to X$, and by
relying upon a $H^2$-cohomology vanishing theorem due to Bogomolov,
Green and Griffiths (\cite{ Green-Griffiths-1980}) 
showed the existence of differential equations satisfied by entire
curves valued in smooth {\em surfaces} $X^2 \subset \mathbb{ P}^3$ of
degree $d \geqslant 5$. In 1996, a breakthrough article by Siu-Yeung
(\cite{ Siu-Yeung-1996}) showed Kobayashi-hyperbolicity of complements
$\mathbb{ P}^2 \backslash X^1$ of generic curves of degree $\geqslant
10^{ 13}$ (rounding off). Around 2000, McQuillan \cite{
McQuillan-1998}, by importing ideas from (multi)foliation theory
considered entire maps valued in compact surfaces of general type
having Chern numbers ${\sf c}_1^2-{\sf c}_2 > 0$, which, for the case
of $X^2 \subset \mathbb{ P}^3$, improved very substantially the degree
bound to $d \geqslant 36$, and this was followed in a work of Demailly
and El Goul ({\em see}~\cite{ Demailly-1997}) by the improvement $d
\geqslant 21$. 
Later, using the Demailly-Semple bundle of jets that
are invariant under reparametrization of the source $\mathbb{ C}$,
Rousseau was the first to treat in great details 
threefolds $X^3 \subset \mathbb{ P}^4$
and he established algebraic
degeneracy of entire curves 
in degree $d \geqslant 593$ (\cite{ Rousseau-2007-Toulouse}).
Previously, 
in two conference proceedings of the first 2000 years
(ICM~\cite{ Siu-2002} and Abel Symposium~\cite{ Siu-2004}), Siu showed
the existence of differential equations on hypersurfaces $X^n \subset
\mathbb{ P}^{ n+1}$; the recent publication~\cite{ Siu-2012} of his
extended preprint of that time confirmed the validity and the strength
of his approach, which will be pursued {\em infra}.

Invariant jets used in~\cite{ Demailly-1997, Rousseau-2007-Toulouse,
Merker-2008-Symbolic, Diverio-Merker-Rousseau-2010} 
are in fact deeply connected to rationality.

Indeed, 
another instance of the key role of denominators appears in
computational invariant theory. Starting with an ideal $\mathcal{ I}
\subset \C[X_1,\dots,X_n]$ and with a nonzero $f \in
\C[X_1,\dots,X_n]$, the $f$-{\sl saturation} of $\mathcal{ I}$ is:
\[
\mathcal{I}^{\sf sat}
\,\,\equiv\,\,
\frac{\mathcal{I}}{f^\infty}
\overset{\text{\rm def}}{\,\,=\,\,}
\big\{
g\in\C[X]\colon\,\,
f^m\,g\,\in\,\mathcal{I}\,\,
\text{\rm for some}\,\,m\in\N
\big\},
\]
with increasing union stabilizing by noetherianity:
\[
\frac{\mathcal{I}}{f}
\,\subset\,
\frac{\mathcal{I}}{f^2}
\,\subset\,
\frac{\mathcal{I}}{f^3}
\,\subset\,
\cdots
\,\subset\,
\frac{\mathcal{I}}{f^m}
=
\frac{\mathcal{I}}{f^{m+1}}
=
\cdots.
\]
The {\sl Kernel algorithm}, discovered in the 
19\textsuperscript{th} Century, consists in:
\[
\aligned
\left<{\sf g}_0,\dots,{\sf g}_{n_0}\right>
&
=
{\sf Initial}\,\,{\sf ideal},
\\
\left<{\sf g}_0,\dots,{\sf g}_{n_0},\dots,
{\sf g}_{n_1}\right>
&
=
{\sf Saturation}_f
\left<{\sf g}_0,\dots,{\sf g}_{n_0}\right>,
\\
\left<{\sf g}_0,\dots,{\sf g}_{n_0},\dots,{\sf g}_{n_1},\dots,
{\sf g}_{n_2}\right>
&
=
{\sf Saturation}_f
\left<{\sf g}_0,\dots,{\sf g}_{n_0},\dots,{\sf g}_{n_1}\right>,
\ \ \ \ \
\text{\em etc.},
\endaligned
\]
and it has been applied in~\cite{ Merker-2008-Symbolic} to set up an
algorithm wich generates all polynomials in the $\kappa$-jet of a
local holomorphic map $\D \to \C^n$, $\zeta \longmapsto (f_1(\zeta),
\dots, f_n(\zeta))$ from the unit disc $\D = \{\vert \zeta \vert <
1\}$ that are invariant under all biholomorphic reparametrizations of
$\D$ with saturation with respect to the first derivative $f_1'$. The
explicit generators for $n = 4 = \kappa$ and for $n = 2$, $\kappa = 5$
given in~\cite{ Merker-2008-Symbolic} show well that
saturation (division) by $f_1'$ generates some unpredictable
complexity, a well known phenomenon in invariant theory.

Later, Berczi and Kirwan (\cite{ Berczi-Kirwan-2012}),
by developing concepts and
tools from reductive geometric invariant theory, showed that the
concerned algebra is always finitely generated. A
challenging still open question is to get information
about the number of generators and about the structure of
relations they share. In any case, the prohibitive complexity of
these algebras still prevents to hope for reaching arbitrary
dimension $n \geqslant 2$ and jet order $\kappa \geqslant n$ 
for Green-Griffiths and Kobayashi conjectures with 
invariant jets.

Around the same time, under the direction of Demailly and using an
algebraic version of holomorphic Morse inequalities delineated by
Trapani, Diverio studied a certain {\em subbundle} of the bundle of
invariant jets, already introduced
before by Demailly in~\cite{ Demailly-1997}. This,
for the first time after Siu, opened the door to arbitrary dimension
$n \geqslant 2$, though this was clearly not sufficient to reach the
first step towards the Green-Griffiths-Lang conjecture. 
In fact, an inspection of~\cite{
Diverio-Merker-Rousseau-2010} shows that on hypersurfaces $X^n
\subset \mathbb{ P}^{ n+1} ( \mathbb{ C})$ of degree $d$
approximately $\geqslant 2^{ n^5}$ (rounding off), many differential
equations exist with jet order $\kappa = n$ equal to the dimension,
but when increasing the jet order $\kappa = n+1, n+2, \dots$, an
unpleasant stabilization of the degree gain occurs, so that there is
absolutely no hope to reach the optimal $d \geqslant n+3$ for the
first step towards the Green-Griffiths-Lang conjecture (as did
Green-Griffiths in 1979 in dimension $n = 2$) with Diverio's technique
(even) for hypersurfaces $X^n \subset \mathbb{ P}^{ n+1} ( \mathbb{
C})$.

\smallskip

{\em For all of these reasons, it became undoubtedly clear that the whole
theory had to come back to the bundle of (plain) Green-Griffiths jets.}

\smallskip

On an $n$-dimensional complex manifod $X^n$, for a jet order $\kappa
\geqslant 1$ and for an homogeneous order $m \geqslant 1$, the {\sl
Green-Griffiths bundle} $\mathcal{E}_{\kappa, m}^{\sf GG} T_X^*
\longrightarrow X$ in a local chart $(z_1,\dots,z_n) \colon {\sf U}
\,\longrightarrow\, \C^n$ with ${\sf U} \subset X$ open, has general
local holomorphic sections which are polynomials in the derivatives
$z'$, $z''$, \dots, $z^{(\kappa)}$ of the $z_i$ (considered as
functions of a single variable $\zeta \in \D$) of the form:
\[
\sum_{\vert\alpha_1\vert+2\,\vert\alpha_2\vert
+\,\cdots\,+
\kappa\,\vert\alpha_\kappa\vert
\,=\,m}\,
P_{\alpha_1,\dots,\alpha_\kappa}(z)\,
\big(z'\big)^{\alpha_1}\,
\big(z''\big)^{\alpha_2}\,
\cdots\,
\big(z^{(\kappa)}\big)^{\alpha_\kappa},
\]
the $P_{\alpha_1, \dots, \alpha_\kappa}$ being holomorphic in ${\sf U}$
(this local definition ignores rationality features of the 
$P_{\alpha, \dots, \alpha_\kappa}$ which 
will be explored {\em infra}).

The memoir~\cite{ Merker-2010-Green-Griffiths}
established that on a hypersurface $X = X^n
\subset \mathbb{ P}^{ n+1} ( \mathbb{ C})$ of degree:
\[
d\geqslant n+3, 
\]
if $\mathcal{ A} \to X$ is any ample line
bundle\,\,---\,\,take {\em e.g.} simply $\mathcal{ A} := \mathcal{
O}_X ( 1)$\,\,---, then:
\[
\aligned
h^0
\big(X,\,\mathcal{E}_{\kappa,m}^{GG}T_X^*
\otimes 
\mathcal{A}^{-1}\big)
\geqslant
\frac{m^{(\kappa+1)n-1}}{
(\kappa!)^n\,((\kappa+1)n-1)!}
\Big\{
&
{\textstyle{\frac{(\log\kappa)^n}{n!}}}
d(d-n-2)^n
\\
&
-
{\sf Constant}_{n,d}\cdot(\log\kappa)^{n-1}
\Big\}
-
\\
&
-
{\sf Constant}_{n,d,\kappa}\cdot
m^{(\kappa+1)n-2},
\endaligned
\]
a formula in which the right-hand side minorant visibly tends to
$\infty$, as soon as both $\kappa \geqslant \kappa_{ n, d}^0$ and $m
\geqslant m_{ n, d, \kappa}^0$ do (no explicit expressions of the
constants was provided there). This, then, generalized to dimension $n
\geqslant 2$ the Green and Griffiths surface theorem, by estimating
the asymptotic quantititative behavior of weighted Young diagrams and
by applying partial (good enough) results of Br\"uckmann (\cite{
Bruckmann-1997}) concerning the cohomology of Schur bundles
$\mathcal{ S}^{ (\ell_1, \dots, \ell_n)} T_X^*$.

Also coming back to plain Green-Griffiths jets, but developing
completely different elaborate negative jet curvature estimates which
go back to an article of Cowen and Griffiths (\cite{
Cowen-Griffiths-1976}) and which had been `in the air' for some
time, though `blocked for deep reasons' by the untractable algebraic
complexity of invariant jets, Demailly (\cite{ Demailly-2011})
realized the next significant advance towards the conjecture by
establishing, under the sole assumption that $X$ be of general type
(not necessarily a hypersurface), that nonconstant entire holomorphic
curves $f \colon \mathbb{ C} \longrightarrow X$ always satisfy (many)
nonzero differential equations. The Bourbaki Seminar 1061 by Pa\u{u}n{
(\cite{ Paun-2012}) is a useful guide to enter the main concepts and
techniques of the topic.

However, according to~\cite{ Diverio-Rousseau-2013}, it is impossible
to reach the Green-Griffiths conjecture for {\em all} general type
compact complex manifolds by applying the jet differential
technique. This justifies to restrict attention to hypersurfaces or to
complete intersections in the projective space, and in this case, as
recently highlighted once again by Siu (\cite{ Siu-2012}), the only
convincing strategy towards a first solution to Kobayashi's conjecture
in arbitrary dimension $n \geqslant 1$ is to develope a new systematic
theory of explicit {\em rational} holomorphic sections of jet bundles.

\subsection{Holomorphic jet differentials}

Let $[X_0 \colon X_1 \colon \cdots \colon X_n ] \in \P^n ( \C)$
be homogenous coordinates. Recall that for $t \in \N$, 
holomorphic sections of $\mathcal{ O}_{ \P^n} (t)$
are represented on ${\sf U}_i = \{ X_i \neq 0\}$ as quotients:
\[
\ell_i([X])
:=
\frac{P(X_0\colon X_1\colon\cdots\colon X_n)}{(X_i)^t},
\]
for some polynomial $P \in \C[X]$ homogeneous 
of degree $t$, while, for $t \in \Z \backslash
\N$, meromorphic sections are:
\[
\ell_i([X])
:=
\frac{(X_i)^t}{P(X_0\colon X_1\colon\cdots\colon X_n)}.
\]

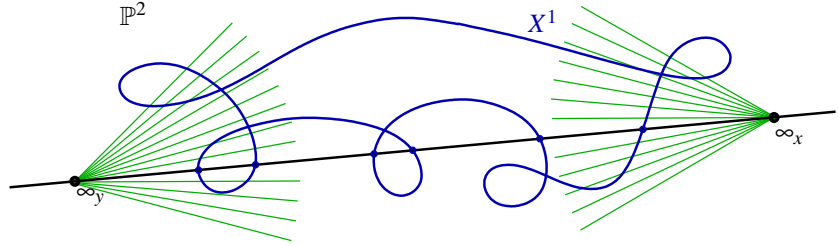
\begin{center}
\input faisceaux-infini-x-y.pstex_t
\end{center}

To review another global rationality phenomenon, consider a
complex algebraic curve $X^1$ smooth or with simple normal crossings
in $\P^2 ( \C)$ of degree $d \geqslant 1$, choose two points
$\infty_x, \infty_y \not \in X$ so that the line $\overline{\infty_x
\infty_y}$ intersects $X^1$ transversally in $d$ distinct points,
and adapt homogeneous coordinates $[T \colon X \colon Y] \in \P^2$
with affine $x := \frac{ X}{T}$ and $y := \frac{ Y}{T}$ so that
$\overline{\infty_x \infty_y} = \P_1^\infty = \{ [0\colon X \colon
Y]\}$, $\infty_x = [0 \colon 1 \colon 0]$, $\infty_y = [0 \colon 0
\colon 1]$, whence:
\[
X^1\cap\C_{(x,y)}^2
\,=\,
\big\{
(x,y)\in\C^2\colon\,
R(x,y)
=
0
\big\},
\]
for some polynomial $R \in \C[x,y]$ of degree $d$.

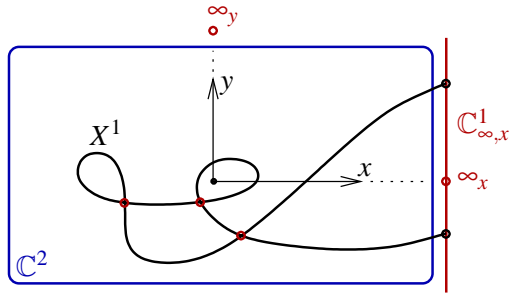
\begin{center}
\input plongement-croisements.pstex_t
\end{center}

Within the intrinsic theory, in terms of the ambient line bundles
$\mathcal{ O}_{ \P^2} ( t) \longrightarrow \P^2$
($t \in \Z$), the adjunction formula tells:
\[
T_X^*
\,\cong\,
\mathcal{O}_X(d-3)
\,\overset{\sf def}{\,\,=\,\,}
\mathcal{O}_{\P^2}(d-3)\big\vert_X.
\]

\begin{quotation}
The genus formula is of great importance, because it exposes the
relationship between the `{\em intrinsic}' topological invariant $g$
of the curve $X^1$ and the `{\em extrinsic}' quantity $d$.
Phillip {\sc Griffiths}.
\end{quotation}

\begin{theorem}
\text{\bf (Inspirational)}
On a smooth degree $d$ algebraic curve $X^1 \subset \P^2$:
\[
{\textstyle{\frac{(d-1)(d-2)}{2}}}
\,=\,
{\sf dim}\,H^0\big(X,T_X^*\big)
\,=\,
{\sf genus}(X)
\,=\,
g.
\]
\end{theorem}

But the extrinsic theory (\cite{ Griffiths-1989})
tells more. Differentiating once $0 = R(x,y)$:
\[
0
=
R_x\,dx
+
R_y\,dy
\ \ \ \ \ \ \ \ \
\Longleftrightarrow
\ \ \ \ \ \ \ \ \
\frac{dy}{R_x}
\,=\,
-\,\frac{dx}{R_y},
\]
{\em denominators must appear}. If $X^1$ is smooth, $X^1 \cap \C^2 =
\{ R_x \neq 0\} \cup \{ R_y \neq 0\}$, and global holomorphic sections
of $T_X^*$ are represented by multiplications:
\[
G(x,y)
\bigg(
\frac{dy}{R_x}
\,=\,
-\,\frac{dx}{R_y}
\bigg),
\]
with $G \in \C[x,y]$ having degree $\leqslant d-3$, 
the space of such $G$ being of dimension
$\frac{ (d-3 + 2)(d-3+1)}{2}$, 
since changing
affine chart in order to capture $\P_\infty^1 \backslash \{ \infty_x\}$:
\[
(x,y)
\,\longmapsto\,
\bigg(
\frac{x}{y},\,
\frac{1}{y}
\bigg)
=:
\big(x_2,y_2\big),
\ \ \ \ \ \ \ \ \ \ \ \ \ \
R_2(x_2,y_2)
\,:=\,
(y_2)^d\,
R\bigg(
\frac{x_2}{y_2},\,
\frac{1}{y_2}
\bigg),
\]
knowing $dy = -\, dy_2 / (y_2)^2$, the left side $G \, dy / R_x$
transfers to:
\[
G(x,y)\,
\frac{dy}{R_x}
\,=\,
G
\bigg(
\frac{x_2}{y_2},\,\frac{1}{y_2}
\bigg)\,
(y_2)^{d-3}\,
\frac{-\,dy_2}{R_{2,x_2}(x_2,y_2)},
\]
the denominator $R_{2, x_2} (x_2, y_2)$ being nonzero on $X$ at every
point of $\{y_2 = 0\} = \C_{\infty,y}^1$, while $(y_2)^{ d-3}$
compensates the poles of $G \big( \frac{ x_2}{ y_2}, \frac{ 1}{ y_2} \big)$
as soon as ${\sf deg}\, G \leqslant d - 3$.

\smallskip

Next, the {\sl intrinsic} Riemann-Roch theorem states that, 
given any divisor $D$ on a compact, abstract, Riemann surface
$S$, if $\mathcal{ O}_D$ denotes the sheaf of meromorphic functions
$f \in \Gamma (\mathcal{M}_S )$ with 
${\sf div}\, f \geqslant - D$, then:
\[
{\sf dim}\,H^0\big(S,\mathcal{O}_D\big)
-
{\sf dim}\,H^1\big(S,\mathcal{O}_D\big)
=
{\sf deg}\,D
-
{\sf genus}(S)
+
1.
\]

For compact Riemann surfaces $S$, there exists a satisfactory 
correspondence between intrinsic features and extrinsic
embeddings: all $S$ admit a representation as a curve $X^1 \subset \P^2$,
smooth or having normal crossings singularities.

Using Brill-Noether duality, the Riemann-Roch theorem can be proved
(\cite{ Griffiths-1989})
for such $X^1 \subset \P^2$ by means of two inequalities:
\[
\aligned
{\sf deg}\,D
-
{\sf g}(S)
+
1
&
\,\leqslant\,
{\sf dim}\,
\big\{
f\in\mathcal{M}(S)\colon\,
{\sf div}(f)
\geqslant
-\,D
\big\}
=
{\sf dim}\,H^0\big(S,\mathcal{O}_D\big),
\\
-\,
{\sf deg}\,D
+
{\sf g}(S)
-
1
&
\,\leqslant\,
{\sf dim}\,
\big\{
\omega\in\mathcal{M}T_X^*(S)\colon\,
{\sf div}(\omega)
\geqslant
+\,D
\big\}
=
{\sf dim}\,H^1\big(S,\mathcal{O}_D\big).
\endaligned
\]
For instance, the second inequality is proved by means
of {\em extrinsic} meromorphic differential forms:
\[
\frac{G}{H}\,
\bigg(
\frac{dy}{R_x}
=
-\,
\frac{dx}{R_y}
\bigg),
\]
with $G, H \in \C[ x, y]$ subjected to appropriate
conditions with respect to $D$.

\medskip\noindent{\bf Jets of order $2$.} Next, consider 
second order jets of holomorphic maps
$\D \to X^1 \subset \P^2$, 
use $x'$, $y'$ instead of $dx$, $dy$, and $x''$, $y''$. Differentiate
$0 \equiv R \big( x(\zeta), y(\zeta) \big)$ twice: 
\[
0
=
x'\,R_x
+
y'\,R_y,
\ \ \ \ \ \ \ \ \ \ \ \ \ \ 
0
=
x''\,R_x
+
y''\,R_y
+
(x')^2\,R_{xx}
+
2\,x'y'\,R_{xy}
+
(y')^2\,R_{yy},
\]
divide by $R_x\, R_y$, solve for $y''$, replace $y'$ on the right:
\[
\frac{y'}{R_x}
=
-\,\frac{x'}{R_y},
\ \ \ \ \ \ \ \ \ \ \ \ \ \ 
\frac{y''}{R_x}
=
-\,\frac{x''}{R_y}
-
\frac{(x')^2}{R_y}\,
\bigg[
\frac{R_{xx}}{\boxed{R_x}}
-
2\,\frac{R_{xy}}{R_y}
+
\frac{R_x}{R_y}\,
\frac{R_{yy}}{R_y}
\bigg].
\]
To erase the division by $R_x$, multiply the first equation
by $x'\, \frac{ R_{xx}}{ R_x}$ and subtract:
\[
\frac{y''}{R_x}
-
\frac{y'x'}{R_x}\,
\frac{R_{xx}}{R_x}
=
-\,\frac{x''}{R_y}
-
\frac{(x')^2}{R_y}\,
\bigg[
-
2\,\frac{R_{xy}}{R_y}
+
\frac{R_x}{R_y}\,
\frac{R_{yy}}{R_y}
\bigg].
\]
Little further, this expression can be symmetrized (\cite{ Merker-2014}):
\[
\aligned
\frac{y''}{R_x}
+
\frac{(y')^2}{R_x}
\bigg[\!
-
\frac{R_{xy}}{R_x}
+
\frac{R_y}{R_x}\,
\frac{R_{xx}}{R_x}
\bigg]
=
-\frac{x''}{R_y}
-
\frac{(x')^2}{R_y}
\bigg[\!
-\frac{R_{xy}}{R_y}
+
\frac{R_x}{R_y}\,
\frac{R_{yy}}{R_y}
\bigg],
\endaligned
\]
and this provides {\em second order holomorphic jet differentials} on
$X^1 \subset \P^2$ when $d \geqslant 4$, after checking holomorphicity on
the $\P_\infty^1$. For jets of order $3$ (\cite{ Merker-2014}):
\[
\aligned
&
\frac{y'''}{R_x}
+
\frac{y''y'}{R_x}
\bigg[
-\,3\,\frac{R_{xy}}{R_x}
+
3\,
\bigg(\!\frac{R_y}{R_x}\!\bigg)
\frac{R_{xx}}{R_x}
\bigg]
\,+
\\
&
\ \ \ \ \ \ 
+
\frac{(y')^3}{R_x}
\bigg[
-\,6\bigg(\!\frac{R_y}{R_x}\!\bigg)
\frac{R_{xy}}{R_x}\frac{R_{xx}}{R_x}
+
3\bigg(\!\frac{R_y}{R_x}\!\bigg)^2
\frac{R_{xx}}{R_x}\frac{R_{xx}}{R_x}
+
3\bigg(\!\frac{R_y}{R_x}\!\bigg)
\frac{R_{xxy}}{R_x}
-
\bigg(\!\frac{R_y}{R_x}\!\bigg)^2
\frac{R_{xxx}}{R_x}
\bigg]
=
\\
&
=
-\,\frac{x'''}{R_y}
-
\frac{x''x'}{R_y}
\bigg[
-\,3\,\frac{R_{xy}}{R_y}
+
3\bigg(\!\frac{R_x}{R_y}\!\bigg)
\frac{R_{yy}}{R_y}
\bigg]
\,-
\\
&
\ \ \ \ \ \ 
-\,
\frac{(x')^3}{R_y}
\bigg[
-\,6\bigg(\!\frac{R_x}{R_y}\!\bigg)
\frac{R_{xy}}{R_y}\frac{R_{yy}}{R_y}
+
3\bigg(\!\frac{R_x}{R_y}\!\bigg)^2
\frac{R_{yy}}{R_y}\frac{R_{yy}}{R_y}
+
3\bigg(\!\frac{R_x}{R_y}\!\bigg)
\frac{R_{xyy}}{R_y}
-
\bigg(\!\frac{R_x}{R_y}\!\bigg)^2
\frac{R_{yyy}}{R_y}
\bigg].
\endaligned
\] 
Strikingly, and quite interestingly, there appear explicit rational
expressions belonging to $\Z \big[ \frac{ R_{\cdot \cdot}}{ R_x}
\big]$ on the left, and to $\Z \big[ \frac{ R_{\cdot \cdot}}{ R_y}
\big]$ on the right.

\medskip\noindent{\bf Jets of arbitrary order.} The {\sl Green-Griffiths} bundle
$\mathcal{E}_{ \kappa, m}^{\rm GG} T_{X^1}^* 
\longrightarrow X^1 \subset \P^2$
consists, for a jet order $\kappa \geqslant 1$, in the $m$-homogeneous
polynomialization of the bundle $J^\kappa\big(\D,\,X^1\big)$, 
and is a {\em vector} bundle of:
\[
{\sf rank}\,
\big(\mathcal{E}_{\kappa,m}^{\rm GG}T_{X^1}^*\big)
\,=\,
{\sf Card}\,
\Big\{
\big(m_1,m_2,\dots,m_\kappa\big)
\in
\N^\kappa\colon\,
m_1+2\,m_2+\cdots+\kappa\,m_\kappa
=
m
\Big\}.
\]
It admits a natural filtration whose associated graded vector bundle is
(\cite{ Merker-2010-Green-Griffiths}):
\[
{\sf Gr}^\bullet
\mathcal{E}_{\kappa,m}^{\rm GG}T_{X^1}^*
\,\cong\,
\bigoplus_{m_1+\cdots+\kappa m_\kappa=m
\atop
m_1\geqslant 0,\,\dots,\,m_\kappa\geqslant 0}\,
\Big(
{\sf Sym}^{m_1}T_X^*
\otimes\cdots\otimes
{\sf Sym}^{m_\kappa}T_X^*
\Big),
\] 
whence:
\[
{\sf Gr}^\bullet
\mathcal{E}_{\kappa,m}^{\rm GG}T_{X^1}^*
\,\cong\,
\bigoplus_{m_1+\cdots+\kappa m_\kappa=m
\atop
m_1\geqslant 0,\,\dots,\,m_\kappa\geqslant 0}\,
\mathcal{O}_X
\Big(
(m_1+\cdots+m_\kappa)\,(d-3)
\Big).
\]
Knowing that for $t \geqslant d$:
\[
{\sf dim}\,H^0\big(X,\mathcal{O}_X(t)\big)
\,=\,
\binom{t+2}{2}
-
\binom{t-d+2}{2},
\]
and knowing $H^1 \big( X, \mathcal{ O}_X ( t) \big) = 0$, it follows:
\[
\aligned
{\sf dim}\,
H^0
\big(X,\,
\mathcal{E}_{\kappa,m}^{\rm GG}T_{X^1}^*
\big)
\,=\,
\sum_{m_1+\cdots+\kappa m_\kappa=m}\,
\bigg\{
&
\binom{(m_1+\cdots+m_\kappa)(d-3)+2}{2}
\,-
\\
&
\,-
\binom{(m_1+\cdots+m_\kappa)(d-3)-d+2}{2}
\bigg\},
\endaligned
\]
which, asymptotically, becomes:
\[
{\sf dim}\,
H^0
\big(X,\,
\mathcal{E}_{\kappa,m}^{\rm GG}T_X^*
\big)
\,\geqslant\,
\frac{m^\kappa}{\kappa!\,\kappa!}\,
\Big[
d^2\,\log\,\kappa
+
d^2\,{\rm O}(1)
+
{\rm O}(d)
\Big]
+
{\rm O}\big(m^{\kappa-1}\big).
\]

\begin{theorem}
\text{\rm (\cite{ Merker-2014})}
Given an arbitrary jet order $\kappa \geqslant 1$, for every
$1 \leqslant \lambda \leqslant \kappa$, if 
${\sf deg}\, R \geqslant \kappa + 3$, there exist
perfectly symmetric expressions: 
\[
\footnotesize
{\sf J}_R^\lambda
\,:=\,
\left\{
\aligned
&
\ \ \ \ \,
\frac{y^{(\lambda)}}{R_x}
+\!\!
\sum_{\mu_1+\cdots+(\lambda-1)\mu_{\lambda-1}=\lambda}
\!\!\!\!\!\!\!
\frac{\big(y'\big)^{\mu_1}\cdots
\big(y^{(\lambda-1)}\big)^{\mu_{\lambda-1}}}{
R_x}\,
\mathcal{J}_{\mu_1,\dots,\mu_{\lambda-1}}^\lambda\!
\left(
\frac{R_y}{R_x},
\bigg(
\frac{R_{x^i y^j}}{R_x}
\bigg)_{2\leqslant i+j\leqslant
\atop
\leqslant
-1+\mu_1+\cdots+\mu_{\lambda-1}}
\right),
\\
\!
&
-\,
\frac{x^{(\lambda)}}{R_y}
-\!\!
\sum_{\mu_1+\cdots+(\lambda-1)\mu_{\lambda-1}=\lambda}
\!\!\!\!\!\!\!
\frac{\big(x'\big)^{\mu_1}\cdots\big(x^{(\lambda-1)}\big)^{\mu_{\lambda-1}}}{
R_y}\,
\mathcal{J}_{\mu_1,\dots,\mu_{\lambda-1}}^\lambda\!
\left(
\frac{R_x}{R_y},
\bigg(
\frac{R_{y^ix^j}}{R_y}
\bigg)_{2\leqslant i+j\leqslant
\atop
\leqslant
-1+\mu_1+\cdots+\mu_{\lambda-1}}
\right),
\\
&
\ \ \ \ \ \ \ \ \ \ \ 
0
\ \ \ \ \ \ \ \ \ 
\text{on}\ \
X^1\cap\P_\infty^1,
\endaligned
\right.
\]
which define {\sl generating} holomorphic jet differentials
on the smooth curve $X^1 \subset \P^2$, notably on the two
open subsets:
\[
\aligned
&
\big\{R_x\neq 0\big\}
\ \
\text{where the bundle $J^\kappa(\D,X^1)$ has intrinsic
coordinates:}
\\
&
\ \ \ \ \ \ \ \ \ \ \ \ \ \ \ \ \ \ 
\big(y;y',y'',\dots,y^{(\kappa)}\big),
\\
&
\big\{R_y\neq 0\big\}
\ \
\text{where the bundle $J^\kappa(\D,X^1)$ has intrinsic
coordinates:}
\\
&
\ \ \ \ \ \ \ \ \ \ \ \ \ \ \ \ \ \ 
\big(x;x',x'',\dots,x^{(\kappa)}\big).
\endaligned
\]
All $J_R^\lambda$ vanish on the ample divisor
$X^1 \cap \P_\infty^1$, and involve universal polynomials:
\[
\mathcal{J}_{\mu_1,\dots,\mu_{\lambda-1}}^\lambda
=
\mathcal{J}_{\mu_1,\dots,\mu_{\lambda-1}}^\lambda
\bigg(
{\sf R}_{0,1},\,\,
\Big(
{\sf R}_{i,j}
\Big)_{2\leqslant i+j\leqslant
-1+\mu_1+\cdots+\mu_{\lambda-1}}
\bigg)
\]
with coefficients in $\Z$,
and in terms of these generating jet differentials, 
holomorphic sections of $\mathcal{E}_{ \kappa, m}^{\rm GG} T_{X^1}^*$
are generally represented as:
\[
\boxed{\,
\sum_{m_1+2m_2+\cdots+\kappa m_\kappa=m}\,
\big({\sf J}_R^1\big)^{m_1}
\big({\sf J}_R^2\big)^{m_2}
\,\cdots\,
\big({\sf J}_R^\kappa\big)^{m_\kappa}
\cdot
{\sf G}_{m_1,m_2,\dots,m_\kappa}(x,y),\,}
\]
with polynomials:
\[
{\sf G}_{m_1,m_2,\dots,m_\kappa}
=
{\sf G}_{m_1,m_2,\dots,m_\kappa}(x,y)
\]
of degrees:
\[
\deg\,
{\sf G}_{m_1,m_2,\dots,m_\kappa}
\,\leqslant\,
\underbrace{
m_1(d-3)
+
m_2(d-4)
+\cdots+
m_\kappa\big(d-\kappa-2\big)}_{
=:\,\delta},
\]
which belong to the quotient spaces:
\[
\C_\delta[x,y]
\Big/
R\cdot\C_{\delta-d}[x,y],
\]
the total number of such sections being equal to:
\[
\footnotesize
\aligned
\sum_{m_1+\cdots+\kappa m_\kappa=m}
\bigg\{
\binom{m_1(d-3)+\cdots+m_\kappa(d-\kappa-2)+2}{2}
-
\binom{m_1(d-3)+\cdots+m_\kappa(d-\kappa-2)-d+2}{2}
\bigg\},
\endaligned
\]
that is to say, also asymptotically equal to:
\[
\frac{m^\kappa}{\kappa!\,\kappa!}\,
\Big[
d^2\,\log\,\kappa
+
d^2\,{\rm O}(1)
+
{\rm O}(d)
\Big]
+
{\rm O}\big(m^{\kappa-1}\big).
\]
\end{theorem}

This generalizes directly to the case of complete
intersection {\em curves} $X^1 \subset \P^{1+c}(\C)$:
\[
0
=
R^1\big(z_1,\dots,z_c,z_{c+1}\big),\ \
\cdots\cdots\cdots,\ \ 
0
=
R^c\big(z_1,\dots,z_c,z_{c+1}\big),
\] 
minors of the Jacobian matrix naturally occupying denominator places:
\[
\frac{z_1'}{
\left\vert\!
\begin{array}{ccc}
R_{z_2}^1 & \cdots & R_{z_{c+1}}^1
\\
\cdot\cdot & \cdots & \cdot\cdot 
\\
R_{z_2}^c & \cdots & R_{z_{c+1}}^c
\end{array}
\!\right\vert}
\,=\,\cdots\cdots\,=\,
(-1)^c\,
\frac{z_{c+1}'}{
\left\vert\!
\begin{array}{ccc}
R_{z_1}^1 & \cdots & R_{z_c}^1
\\
\cdot\cdot & \cdots & \cdot\cdot 
\\
R_{z_1}^c & \cdots & R_{z_c}^c
\end{array}
\!\right\vert}.
\]

\medskip\noindent{\bf Holomorphic sections of the canonical bundle.}
For $X^n \subset \P^{ n+1} ( \C)$ a smooth hypersurface:
\[
0
=
R\big(z_1,\dots,z_n,z_{n+1}\big),
\] 
affinely represented as the zero-set of a degree $d \geqslant 1$ 
polynomial, whence:
\[
0
=
R_{z_1}\,dz_1
+\cdots+
R_{z_n}\,dz_n
+
R_{z_{n+1}}\,dz_{n+1},
\]
holomorphic sections of the {\sl canonical bundle} $K_X := \Lambda^n
T_X^*$ (which generalizes the cotangent $T_X^*$ of $X^1 \subset \P^2$), are
represented by the equalities:
\[
\frac{dz_1\wedge\cdots\wedge dz_n}
{R_{z_{n+1}}}
\,=\,
-\,
\frac{dz_1\wedge\cdots\wedge dz_{n-1}\wedge dz_{n+1}}
{R_{z_n}}
\,=\,
\cdots\cdots\cdots
\,=\,
(-1)^n\,
\frac{dz_2\wedge\cdots\wedge dz_{n+1}}
{R_{z_1}},
\]
that are always holomorphic on the $\P_\infty^n$ as soon as
$d \geqslant n+3$.

\medskip\noindent{\bf Question.}
{\em For $X^n \subset \P^{ n+c}$ a complete intersection $\{ 0 = R^1 =
\cdots = R^c\}$ of codimension $c$, are there explicit rational
holomorphic sections of $\mathcal{E}_{ \kappa, m}^{\rm GG} T_X^*$
having as denominators the appropriate minors of the Jacobian
matrix $\big(R_{z_k}^j\big)$ and numerators in $\Z \big[ R_{ 
z_1^{\alpha_1} \cdots z_{n+c}^{\alpha_{n+c}}}^j \big]$}?

\medskip

\medskip\noindent{\bf Surfaces $X^2 \subset \P^3$.}
Let $X^2 \subset \P^3$ be a smooth surface represented in affine
coordinates $(x, y, z) \in \C^3 \subset \P^3$ as:
\[
0
=
R(x,y,z),
\]
for some polynomial $R \in \C [x, y, z]$ of degree $d \geqslant 1$.
Differentiate this once:
\[
0
=
x'\,R_x
+
y'\,R_y
+
z'\,R_z.
\]
Three natural open sets $\{ R_x \neq 0 \}$, $\{ R_y \neq 0\}$,
$\{ R_z \neq 0\}$ cover $X^2$, by smoothness. On $X^2 \cap 
\{ R_x \neq 0\}$, coordinates are $(y, z)$, cotangent
(fiber) coordinates are $(y', z')$. The change of trivialization
for $T_X^* \cong J^1(\D, X)$
from above $X^2 \cap \{ R_x \neq 0\}$ to above
$X^2 \cap \{ R_y \neq 0\}$:
\[
\big(y,z,y',z'\big)
\,\,\longmapsto\,\,
\big(x,z,x',z'\big)
\]
amounts to just solving:
\[
y'
=
-\,x'\,
\frac{R_x}{R_y}
-
z'\,\frac{R_z}{R_y}.
\]
Inspired by what precedes for curves $X^1 \subset \P^2$,
seek global holomorphic sections of:
\[
\mathcal{E}_{1,m}^{\rm GG}T_X^*
\,\cong\,
{\sf Sym}^m\,T_X^*
\]
(symmetric differentials) under the form:
\[
\sum_{j+k=m}\,
{\sf coeff}_{j,k}
\cdot
(y')^j\,(z')^k
\,\,\,\,\,
\underset{\sf trivialization}{
\overset{{\sf change}\,\,{\sf of}}
{\longmapsto\rule[-3pt]{0pt}{11pt}}}
\,\,\,\,\,
\sum_{j+k=m}\,
{\sf coeff}_{j,k}^\sim
\cdot
(x')^j\,(z')^k,
\]
with all coefficient-functions ${\sf coeff}_{ j,k} ( y,z)$
being holomorphic on $X^2 \cap \{ R_x \neq 0\}$ and all 
${\sf coeff}_{j, k}^\sim ( x,z)$ being 
holomorphic on $X^2 \cap \{ R_y \neq 0\}$.
What sort of coefficients? A proposal of answer,
inspired by $X^1 \subset \P^2$, is that they belong to:
\[
\frac{1}{R_x}\,
\Z
\bigg[
\frac{R_y}{R_x},\,\frac{R_z}{R_x}
\bigg]
\ \ \ \ \ \ \ \ \ \ \ \ \
\text{\rm and to:}
\ \ \ \ \ \ \ \ \ \ \ \ \
\frac{1}{R_y}\,
\Z
\bigg[
\frac{R_x}{R_y},\,\frac{R_z}{R_y}
\bigg],
\]
because then, such jet differentials would vanish on the $\P_\infty^2$,
as soon as ${\sf deg}\, R \geqslant 2 m + 2$. Of course, 
in this special case, since the intrinsic theory
(\cite{ Diverio-Rousseau-2011, Brotbek-2011}) kowns that
whenever $\kappa < \frac{ n}{ c}$, 
on a complete intersection $X^n \subset \P^{ n + c}$:
\[
0
=
H^0\big(X,\,\mathcal{E}_{\kappa,m}^{\rm GG}T_X^*\big),
\]
whence with $n = 2$, $c = 1$ here, inexistence is expectable. 
As a confirmation:

\begin{proposition}
For every $m \geqslant 1$, if polynomials $\Pi_{j, k} \in \Z [
{\sf U}, {\sf V} ]$ are such that:
\[
\aligned
\sum_{j+k=m}\,
\frac{(y')^j\,(z')^k}{R_x}\,
\Pi_{j,k}
\bigg(
\frac{R_y}{R_x},\,\frac{R_z}{R_x}
\bigg)
&
\,=\,
\sum_{j+k=m}\,
\frac{\big(-\,x'\,\frac{R_x}{R_y}
-
z'\,\frac{R_z}{R_y}\big)^j\,(z')^k}{R_x}\,
\Pi_{j,k}
\bigg(
\frac{R_y}{R_x},\,\frac{R_z}{R_x}
\bigg)
\\
&
\,=\,
\sum_{j+k=m}\,
\frac{(x')^j\,(z')^k}{R_y}\,
\Pi_{j,k}^\sim\,
\bigg(
\frac{R_x}{R_y},\,\frac{R_z}{R_y}
\bigg)
\endaligned
\]
rewrites, after transitioning, only with $\frac{ 1}{ R_y}$-denominators, 
then all $\Pi_{j, k} \equiv 0$.
\end{proposition}

However, the intrinsic theory knows already (\cite{ Siu-Yeung-1996,
Demailly-1997, Siu-2002, Siu-2004, Rousseau-2007-Toulouse,
Diverio-Merker-Rousseau-2010, Diverio-Rousseau-2011, Brotbek-2011})
that for $X^n \subset \P^{n + c}$, global holomorphic sections of
$\mathcal{ E}_{ \kappa, m}^{\rm GG} T_X^*$ exist when $\kappa
\geqslant \frac{ n}{c}$.

Hence for $n = 2$, $c = 1$, $\kappa = 2$, differentiate 
{\em twice:}
\[
\aligned
0
&
=
x'\,R_x+y'\,R_y+z'\,R_z,
\\
0
&
=
x''\,R_x+y''\,R_y+z''\,R_z
+
(x')^2\,R_{xx}
+
(y')^2\,R_{yy}
+
(z')^2\,R_{zz}
+
2x'y'\,R_{xy}
+
2x'z'\,R_{xz}
+
2y'z'\,R_{yz}.
\endaligned
\]
The interesting question (to which no answer is known not up to date)
is whether there exist holomorphic
jet differentials of the rational form:
\[
\sum_{j_1+k_1+j_2+k_2=m}\,
\frac{(y')^{j_1}\,(z')^{k_1}\,(y'')^{j_2}\,(z'')^{k_2}}{R_x}
\cdot
\Pi_{j_1k_1j_2k_2}
\bigg(
\frac{R_y}{R_x},\,
\frac{R_z}{R_x},\,
\frac{R_{xx}}{R_x},\,
\frac{R_{yy}}{R_x},\,
\frac{R_{zz}}{R_x},\,
\frac{R_{xy}}{R_x},\,
\frac{R_{xz}}{R_x},\,
\frac{R_{yz}}{R_x}
\bigg),
\]
having the property that, after replacement of:
\[
\aligned
y'
&
=
-x'\,\frac{R_x}{R_y}
-
z'\,\frac{R_z}{R_y},
\\
y''
&
=
-x''\,\frac{R_x}{R_y}
-
z''\,\frac{R_z}{R_y}
-
(x')^2\,\frac{R_{xx}}{R_y}
-
(y')^2\,\frac{R_{yy}}{R_y}
-
(z')^2\,\frac{R_{zz}}{R_y}
-
2x'y'\,\frac{R_{xy}}{R_y}
-
2x'z'\,\frac{R_{xz}}{R_y}
-
2y'z'\,\frac{R_{yz}}{R_y},
\endaligned
\]
after expansion and after reorganization,
a similar jet-rational expression is got:
\[
\sum_{j_1+k_1+j_2+k_2=m}\,
\frac{(x')^{j_1}\,(z')^{k_1}\,(y'')^{j_2}\,(z'')^{k_2}}{R_y}
\cdot
\Pi_{j_1k_1j_2k_2}^\sim
\bigg(
\frac{R_x}{R_y},\,
\frac{R_z}{R_y},\,
\frac{R_{xx}}{R_y},\,
\frac{R_{yy}}{R_y},\,
\frac{R_{zz}}{R_y},\,
\frac{R_{xy}}{R_y},\,
\frac{R_{xz}}{R_y},\,
\frac{R_{yz}}{R_y}
\bigg)
\]
which involves division by only $R_y$.
The number of variables becomes $8$ (large):
\[
\aligned
\Pi_{j_1k_1j_2k_2}
&
\,\in\,
\frac{1}{R_x}
\cdot
\Z
\bigg[
\frac{R_y}{R_x},\,
\frac{R_z}{R_x},\,
\frac{R_{xx}}{R_x},\,
\frac{R_{yy}}{R_x},\,
\frac{R_{zz}}{R_x},\,
\frac{R_{xy}}{R_x},\,
\frac{R_{xz}}{R_x},\,
\frac{R_{yz}}{R_x}
\bigg],
\\
\Pi_{j_1k_1j_2k_2}^\sim
&
\,\in\,
\frac{1}{R_y}
\cdot
\Z
\bigg[
\frac{R_x}{R_y},\,
\frac{R_z}{R_y},\,
\frac{R_{xx}}{R_y},\,
\frac{R_{yy}}{R_y},\,
\frac{R_{zz}}{R_y},\,
\frac{R_{xy}}{R_y},\,
\frac{R_{xz}}{R_y},\,
\frac{R_{yz}}{R_y}
\bigg].
\endaligned
\]
By anticipation, for $X^n \subset \P^{n+1}$ and for jets of order
$\kappa = n$:
\[
\#
\Big(
\text{\rm partial derivatives}\,
R_{z_1^{\alpha_1}\cdots z_n^{\alpha_n}z_{n+1}^{\alpha_{n+1}}}
\Big)
\,=\,
\binom{n+1+n}{n}
\,\sim\,
2^{2n+1}\,
\frac{1}{\sqrt{\pi n}},
\]
hence something is {\em intimately exponential} in the subject.
For $X^2 \subset \P^3$ of degree $d \gg 1$, a patient cohomology
sequences chasing shows that there exist nonzero second-order
holomorphic jet differentials in $H^0 \big( X, \mathcal{ E}_{2,
m}^{\rm GG} T_X^* \big)$ only when:
\[
m
\geqslant
14
\]
(similarly, by \cite{ Brotbek-2011}, for $X^2 \subset \P^4$ of bidegrees
$d_1, d_2 \gg 1$, it is necessary that 
$m \geqslant 10$). Hence combinatorially,
there is a complexity obstacle, and moreover, 
an inspection
of what holds true for curves 
$X^1 \subset \P^2$ shows that it is quite probable
that the degrees of the $\Pi_{ j_1k_1j_2k_2}$ 
are about to be approximately equal to $m \geqslant
14$, whence the total number of monomials they involve:
\[
\binom{14+8}{8}
=
319\,770,
\]
would be already rather large to determine 
in a really effective way whether they exist. 

It happens to be a bit easier to work with the Wronskians:
\[
\aligned
\square
:=
\left\vert\!
\begin{array}{cc}
y' & z'
\\
y'' & z''
\end{array}
\!\right\vert
=
y'z''-z'y'',
\ \ \ \ \ \ \ \ \ \ \ \ \ \ \ \ \ \ \ \ \ 
\Delta
:=
\left\vert\!
\begin{array}{cc}
z' & x'
\\
z'' & x''
\end{array}
\!\right\vert
=
z'x''-x'z''.
\endaligned
\]
Two fundamental transition formulas are:
\[
\aligned
\frac{y'}{R_x}
&
\,\,=\,\,
-\,\frac{x'}{R_y}
-
\frac{z'}{R_y}\,
\frac{R_z}{R_x},
\\
\frac{
\left\vert\!
\begin{array}{cc}
y' & z'
\\
y'' & z''
\end{array}
\!\right\vert
}{R_x}
&
\,\,=\,\,
\frac{\left\vert\!
\begin{array}{cc}
z' & x'
\\
z'' & x''
\end{array}
\!\right\vert}{R_y}
-
\frac{(x')^2z'}{R_y}\,
\bigg[
\bigg(\!
\frac{R_x}{R_y}
\!\bigg)
\frac{R_{yy}}{R_y}
-
2\,\frac{R_{xy}}{R_y}
+
\frac{R_{xx}}{R_x}
\bigg]
-
\\
&
\ \ \ \ \ \ \ \ \ \ \ \ \ \ \
-
2\,\frac{x'(z')^2}{R_y}
\bigg[
\bigg(\!
\frac{R_z}{R_y}
\!\bigg)
\frac{R_{yy}}{R_y}
-
\frac{R_{yz}}{R_y}
+
\frac{R_{xz}}{R_x}
-
\bigg(\!
\frac{R_z}{R_y}
\!\bigg)
\frac{R_{xy}}{R_x}
\bigg]
-
\\
&
\ \ \ \ \ \ \ \ \ \ \ \ \ \ \ \ \
-
\frac{(z')^3}{R_y}
\bigg[
\bigg(\!
\frac{R_z}{R_y}
\!\bigg)^2
\frac{R_{yy}}{R_x}
-
2\bigg(\!
\frac{R_z}{R_y}
\!\bigg)
\frac{R_{yz}}{R_x}
+
\frac{R_{zz}}{R_x}
\bigg].
\endaligned
\]
Set as abbreviated new notations:
\[
\aligned
\frac{R_x}{R_y}
&
=:
r_x,
\ \ \ \ \ \ \ \ \ \ \ \ \ \ \ \ \
\frac{R_z}{R_y}
=:
r_z,
\\
\frac{R_{xx}}{R_y}
&
=:
r_{yy},
\ \ \ \ \ \ \ \ \ \ \ \ \ \ \ \ \
\frac{R_{xy}}{R_y}
=:
r_{xy},
\\
\frac{R_{yy}}{R_y}
&
=:
r_{zz},
\ \ \ \ \ \ \ \ \ \ \ \ \ \ \ \ \
\frac{R_{xz}}{R_y}
=:
r_{xz},
\\
\frac{R_{zz}}{R_y}
&
=:
r_{xx},
\ \ \ \ \ \ \ \ \ \ \ \ \ \ \ \ \
\frac{R_{yz}}{R_y}
=:
r_{yz}.
\endaligned
\]
Rewrite:
\[
\aligned
y'
&
=
-x'\,r_x
-
z'\,r_z,
\\
\square
&
=
\Delta\,r_x
-
(x')^2z'
\Big[
r_x\,r_x\,r_{yy}
-
2\,r_x\,r_{xy}
+
r_{xx}
\Big]
-
\\
&
\ \ \ \ \ \ \ \ \ \ 
-
2\,x'(z')^2
\Big[
r_x\,r_z\,r_{yy}
-
r_x\,r_{yz}
+
r_{xz}
-
r_z\,r_{xy}
\Big]
-
\\
&
\ \ \ \ \ \ \ \ \ \ \ \ \ \ \ 
-
(z')^3
\Big[
r_z\,r_z\,r_{yy}
-
2\,r_z\,r_{yz}
+
r_{zz}
\Big].
\endaligned
\]
Divide both sides by:
\[
R_x
=
r_x\,R_y,
\]
and obtain:
\[
\aligned
\frac{y'}{R_x}
&
=
-\,\frac{x'}{R_y}
-
\frac{z'}{R_y}\,
\frac{r_z}{r_x},
\\
\frac{\square}{R_x}
&
=
\frac{\Delta}{R_y}
-
\frac{(x')^2z'}{R_y}
\Big[
r_x\,r_{yy}
-
2\,r_{xy}
+
\frac{r_{xx}}{r_x}
\Big]
-
\\
&
\ \ \ \ \ \ \ \ 
-
2\,\frac{x'(z')^2}{R_y}
\Big[
r_z\,r_{yy}
-
r_{yz}
+
\frac{r_{xz}}{r_x}
-
\frac{r_z\,r_{xy}}{r_x}
\Big]
-
\\
&
\ \ \ \ \ \ \ \ \
\frac{(z')^3}{R_y}
\Big[
\frac{r_z\,r_z\,r_{yy}}{r_x}
-
2\,\frac{r_z\,r_{yz}}{r_x}
+
\frac{r_{zz}}{r_x}
\Big].
\endaligned
\]
In the case $n = 2 = \kappa$, $c = 1$, the question formulated
above becomes:

\medskip\noindent{\bf Question.} 
{\em Do there exist nontrivial linear combinations of:}
\[
\aligned
&
\frac{
\big(-x'r_x-z'r_z)^j(z')^k
\left(
\aligned
\Delta r_x
&
-
(x')^2z'
\big[
r_xr_xr_{yy}-2r_xr_{xy}+r_{xx}
\big]
\\
&
-
2x'(z')^2\big[
r_xr_zr_{yy}
-
r_xr_{yz}
+
r_{xz}
-
r_zr_{xy}
\big]
-
\\
&
-
(z')^3\big[
rr_zr_zr_{yy}-2r_zr_{yz}+r_{zz}
\big]
\endaligned
\right)^l
}{
R_x\,r_x}\,
\times
\\
&
\times
\bigg(\!
\frac{1}{r_x}
\!\bigg)^a\,
\bigg(\!
\frac{r_z}{r_x}
\!\bigg)^b\,
\bigg(\!
\frac{r_{xx}}{r_x}
\!\bigg)^c\,
\bigg(\!
\frac{r_{yy}}{r_x}
\!\bigg)^d\,
\bigg(\!
\frac{r_{zz}}{r_x}
\!\bigg)^e\,
\bigg(\!
\frac{r_{xy}}{r_x}
\!\bigg)^f\,
\bigg(\!
\frac{r_{xz}}{r_x}
\!\bigg)^g\,
\bigg(\!
\frac{r_{yz}}{r_x}
\!\bigg)^h,
\endaligned
\]
{\em for some nonnegative integers:}
\[
j,\,k,\,l,
\ \ \ \ \
a,\,b,\,c,\,d,\,e,\,f,\,g,\,h,
\]
{\em in which any $\frac{ 1}{r_x}$ would have disappeared?}

\subsection{Slanted vector fields}

To construct holomorphic jet differentials on a hypersurface
$X^n \subset \P^{ n+1}$ defined as:
\[
0
=
R(z_1,\dots,z_n,z_{n+1})
=
\sum_{\alpha_1+\cdots+\alpha_n+\alpha_{n+1}\leqslant d}\,
a_{\alpha_1\dots\alpha_n\alpha_{n+1}}\,
z_1^{\alpha_1}
\cdots
z_n^{\alpha_n}
z_{n+1}^{\alpha_{n+1}},
\]
two strategies exist, the first one (still open) being to work
(only) in the ring (of fractions) of all partial derivatives of $R$:
\[
\Z
\Big[
\Big(
R_{z_1^{\beta_1}\cdots z_{n+1}^{\beta_{n+1}}}
\Big)_{\beta_1+\cdots+\beta_{n+1}\leqslant\kappa}
\Big],
\]
and the second one (currently active) 
being to work in the ring of coefficients:
\[
\Z
\Big[
\Big(
a_{\alpha_1\dots\alpha_{n+1}}
\Big)_{\alpha_1+\cdots+\alpha_{n+1}\leqslant d}
\Big].
\]

For instance, differentiate $0 = \sum_{\alpha}\, a_\alpha\, z^\alpha$
up to order, say, $4$:
\[
\aligned
0
&
=
\sum_\alpha\,
a_\alpha\,z^\alpha
\\
0
&
=
\sum_\alpha\,a_\alpha
\bigg(
\sum_{j_1}\,\frac{\partial (z^\alpha)}{\partial z_{j_1}}\,z_{j_1}'
\bigg)
\\
0
&
=
\sum_{\alpha}\,a_\alpha
\bigg(
\sum_{j_1}\,\frac{\partial (z^\alpha)}{\partial z_{j_1}}\,z_{j_1}''
+
\sum_{j_1,\,j_2}\,
\frac{\partial^2 (z^\alpha)}{\partial z_{j_1}\partial z_{j_2}}\,
z_{j_1}'z_{j_2}'
\bigg)
\\
0
&
=
\sum_\alpha\,a_\alpha\,
\bigg(
\sum_{j_1}\,\frac{\partial (z^\alpha)}{\partial z_{j_1}}\,z_{j_1}'''
+
\sum_{j_1,\,j_2}\,
\frac{\partial^2(z^\alpha)}{\partial z_{j_1}\partial z_{j_2}}\,
3\,z_{j_1}'z_{j_2}''
+
\sum_{j_1,\,j_2,\,j_3}\,
\frac{\partial^3(z^\alpha)}{\partial z_{j_1}\partial z_{j_2}\partial
z_{j_3}}\,
z_{j_1}'z_{j_2}'z_{j_3}'
\bigg)
\\
0
&
=
\sum_\alpha\,a_\alpha
\bigg(
\sum_{j_1}\,
\frac{\partial (z^\alpha)}{\partial z_{j_1}}\,z_{j_1}''''
+
\sum_{j_1,\,j_2}\,
\frac{\partial^2(z^\alpha)}{\partial z_{j_1}\partial z_{j_2}}
\big(
4\,z_{j_1}'z_{j_2}'''
+
3\,z_{j_1}''z_{j_2}''
\big)
+
\\
&
\ \ \ \ \ \ \ \ \ \ \ \ \ \ \ \ \
+
\sum_{j_1,\,j_2,\,j_3}\,
\frac{\partial^3(z^\alpha)}{\partial z_{j_1}
\partial z_{j_2}\partial z_{j_3}}\,
6\,z_{j_1}'z_{j_2}'z_{j_3}''
+
\sum_{j_1,\,j_2,\,j_3,\,j_4}\,
\frac{\partial^4(z^\alpha)}{\partial z_{j_1}\partial
z_{j_2}\partial z_{j_3}\partial z_{j_4}}\,
z_{j_1}'z_{j_2}'z_{j_3}'z_{j_4}'
\bigg).
\endaligned
\]

\begin{lemma}
{\rm (\cite{ Merker-2009})}
The equation obtained by differentiating the condition
$R\big( f(\zeta)\big) \equiv 0$ up to an arbitrary order
$\kappa \geqslant 1$ reads in closed form as follows:
\[
\aligned
\!\!\!0
&
=\!\!\!
\sum_{\alpha\in\N^{n+1}}\!\!
a_{\alpha}\,\,\,
\sum_{e=1}^\kappa\,
\sum_{1\leqslant\lambda_1<\cdots<\lambda_e\leqslant\kappa}\,
\sum_{\mu_1\geqslant 1,\dots,\mu_e\geqslant 1}\,
\sum_{\mu_1\lambda_1+\cdots+\mu_e\lambda_e=\kappa}\,
\frac{\kappa!}{(\lambda_1!)^{\mu_1}\mu_1!\cdots
(\lambda_e!)^{\mu_e}\mu_e!}\,
\\
\!\!\!
&
\sum_{j_1^1,\dots,j_{\mu_1}^1=1}^{n+1}\!\cdots\!
\sum_{j_1^e,\dots,j_{\mu_e}^e=1}^{n+1}
\frac{\partial^{\mu_1+\cdots+\mu_e}\big(z^\alpha\big)}{\partial
z_{j_1^1}\cdots\partial z_{j_{\mu_1}^1}\cdots
\partial z_{j_1^e}\cdots\partial z_{j_{\mu_e}^e}}\,
z_{j_1^1}^{(\lambda_1)}\cdots z_{j_{\mu_1}^1}^{(\lambda_1)}
\cdots
z_{j_1^e}^{(\lambda_e)}\cdots z_{j_{\mu_e}^e}^{(\lambda_e)}.
\endaligned
\]
\end{lemma}

These equations for $\kappa = 0, 1, \dots, \kappa$ define a certain
(projectivizable) subvariety: 
\[
J_{\rm vert}^\kappa \subset
\C_{(z_k)}^{n+1} 
\times 
\C_{(a_\alpha)}^{\frac{(n+1+d)!}{(n+1)!\,\,d!}},
\] 
complete intersection of codimension $\kappa + 1$
outside $\{ z_1' = \cdots = z_{n+1}' = 0\}$.
Vector fields tangent to $J_{\rm vert}^\kappa$
write under the general form:
\[
\aligned
{\sf T}
=
\sum_{i=1}^{n+1}\,
{\sf Z}_i\,\frac{\partial}{\partial z_i}
+
\sum_{\alpha\in\N^{n+1}\atop\vert\alpha\vert\leqslant d}\,
{\sf A}_\alpha\,
\frac{\partial}{\partial a_\alpha}
+
\sum_{k=1}^{n+1}\,{\sf Z}_k'\,\frac{\partial}{\partial z_k'}
+
\sum_{k=1}^{n+1}\,{\sf Z}_k''\,\frac{\partial}{\partial z_k''}
+\cdots+
\sum_{k=1}^{n+1}\,{\sf Z}_k^{(\kappa)}\,
\frac{\partial}{\partial z_k^{(\kappa)}}.
\endaligned
\]

Notably, the next theorem works with the quotient ring of 
$\Z [ a_\alpha]$, not of $\Z [ R_{ z^\beta}]$.

\begin{theorem}
{\rm (\cite{ Siu-2004, Merker-2009, Mourougane-2012, Siu-2012})}
\label{slanted-vf}
With $\kappa \leqslant d$, at every point of $J_{\sf vert}^\kappa \big
\backslash \{ z_i' = 0\}$, there exist $j_{n,\kappa}^d := \dim\, J_{\rm
vert}^\kappa$ global holomorphic sections ${\sf T}_1, \dots, {\sf
T}_{ j_{n,\kappa}^d}$ of the twisted tangent bundle:
\[
T_{J_{\rm vert}^\kappa}
\otimes
\mathcal{O}_{\P^{n+1}}
\big(
\kappa^2+2\,\kappa
\big)
\otimes
\mathcal{O}_{\P^{\frac{(n+1+d)!}{(n+1)!\,\,d!}-1}}
(1),
\]
which generate the tangent space:
\[
\C{\sf T}_1
\big\vert_p
\oplus\cdots\oplus
\C{\sf T}_{j_{n,\kappa}^d}
\big\vert_p
=
T_{J_{\sf vert}^n,\,p}.
\]
\end{theorem}

According to Siu (\cite{ Siu-2004, Siu-2012}), these fields can be used 
to show that, for $X$ generic,
entire curves  $f \colon \C \to X$ land in the base locus of {\em all} 
global algebraic jet differentials belonging to the space:
\begin{equation}
\label{h-0}
H^0\big(X,E_{n,m}^{\rm GG}T^*_X\otimes 
K_X^{-\delta m}\big)\neq 0,
\end{equation}
which is shown in~\cite{ Diverio-Merker-Rousseau-2010} 
to be nonzero for small enough $\delta \in \Q_{>
0}$, for $\kappa = n$, for $m \gg 1$, provided:
\[
d
=
{\sf deg}\,X
\,\geqslant\,
2^{n^5}.
\]
More precisely, by an abstract argument,
extend locally any such jet differential:
\[
{\sf P}\big(z,a\big)
=
\sum_{\vert\gamma_1\vert+\cdots+n\vert\gamma_n\vert=m}\,
{\sf p}_\gamma\big(z,a\big)\,
\big(z'\big)^{\gamma_1}\cdots(z^{(n)}\big)^{\gamma_n},
\]
for $a$ generic.
Use the vector fields of Theorem~\ref{slanted-vf} 
to eliminate $(z')^{\gamma_1} \dots
(z^{(n)})^{\gamma_n}$, and get the:

\begin{proposition}
\label{Y-locus}
Nonconstant entire curves algebraically degenerate inside:
\[
Y
:=
\big\{
z\in X\colon
\underbrace{
{\sf p}_\gamma\big(z,a\big)
=
0,
\ \
\forall\,
\vert\gamma_1\vert+\cdots+n\,\vert\gamma_n\vert
=
m}_{
\text{{\sf all coefficients, very numerous}}} 
\big\}.
\]
\end{proposition}

Here, the total number of algebraic equations
$p_\gamma (z, a) = 0$ is exponentially large
$\approx m^n \gg (2^{n^5})^n$.
Naturally, the common zero-set should conjecturally
be empty, whence Kobayashi's conjecture\,\,---\,\,not in 
optimal degree\,\,---\,\,seems
to be almost established.
However, all {\em intrinsic} techniques which 
provide global holomorphic sections like~\thetag{ \ref{h-0}} above,
namely either a 
decomposition of jet bundles in Schur bundles, or
asymptotic Morse inequalities, or else
probabilistic curvature estimates, 
are up to now unable to provide a partial explicit expression of
even a single algebraic coefficient ${\sf p}_\gamma ( z, a \big)$.

\smallskip

{\em This is why a refoundation towards rational effectiveness
is necessary.}

\smallskip

At least before refounding the construction of holomorphic jet
differentials, such intrinsic approaches may be pushed further to
improve the degree bound $d \geqslant 2^{ n^5}$, and to treat new
geometric situations.

Brotbek~(\cite{ Brotbek-2011}) produced holomorphic jet differentials
on general complete intersections $X^n \subset \P^{n+c}$ of
multidegrees $d_1, \dots, d_c \gg 1$.  Mourougane~(\cite{
Mourougane-2012}) showed that for general moving enough families of
high enough degree hypersurfaces in $\P^{ n+1}$, there is a proper
algebraic subset of the total space that contains the image of all
sections.

Yet getting information about `high enough' degrees represents
a substantial computational work.

The most substantial recent progress concerning degree bounds is
mainly due to Berczi (\cite{ Berczi-2010}), in the case of $f \colon
\C \to X^n \subset \P^{ n+1} ( \C)$, with $d \,\geqslant\, n^{8n}$,
instead of $d \geqslant 2^{n^5}$.  Using the above vector fields and
probabilistic curvature estimates for
Green-Griffiths jets, Demailly obtained in~\cite{ Demailly-2012}, 
still in the
case $f \colon \C \to X^n \subset \P^{ n+1} ( \C)$:
\[
d\geqslant
\frac{n^4}{3}\,
\big(n\,{\sf log}\,\big(n\,\big(\log(24\,n)\big)\big)^n.
\]

\begin{theorem}
{\rm ({\sc Darondeau}, \cite{ Darondeau-2014a})}
Suppose that the jet order $\kappa$ is 
Let $X^{ n -1} \subset \P^n ( \C)$ be a smooth complex projective
algebraic hypersurface of degree:
\[
d
\,\geqslant\,
\,(5\,n)^2\,n^n.
\]
If $X^{ n-1}$ is Zariski-generic, then there exists a proper algebraic
subvariety $Y \subsetneqq \P^n$ of codimension $\geqslant 2$ such that
every nonconstant entire holomorphic curve $f\colon \C \rightarrow
\P^n \backslash X$ actually lands in $Y$, namely $f(\C) \subset Y$.
\end{theorem}

Consider again the {\sl universal family} of degree $d$ hypersurfaces of
$\P^n$:
\[
\mathcal{H}
:=
\Big\{
\big([Z],[A]\big)
\in
\P^n\times\P^{\binom{n+d}{d}-1}
\colon\,
\sum\,A_\alpha\,Z^\alpha
=
0
\Big\}.
\] 
With an additional variable
$W \in \C$, introduce the family of hypersurfaces of $\P^{n+1}$:
\[
W^d
=
\sum_\alpha\,A_\alpha\,Z^\alpha.
\]
The space $J_{\rm vert}^\kappa( -{\sf log})$ of {\sl vertical
logarithmic $\kappa$-jets} is associated to jets of local
holomorphic maps $f\colon \D \to \P^n \backslash \mathcal{ H}_A$
valued in the complement of a hypersurface $\mathcal{ H}_A$
corresponding to a fixed $A$ and having a certain determined
behavior near $\{ W = 0\}$.  The counterpart of
Theorem~\ref{slanted-vf} useful {\em infra} is:

\begin{theorem}
{\rm ({\sc Darondeau}, \cite{Darondeau-2014b})}
\label{slanted-complement}
With $\kappa \leqslant d$, the twisted tangent bundle to the space of
logarithmic $\kappa$-jets:
\[
T_{J_{\rm vert}^\kappa(-{\sf log})}
\otimes
\mathcal{O}_{\P^n}\big(\kappa^2+2\kappa\big)
\otimes
\mathcal{O}_{\P^{\frac{(n+d)!}{n!\,d!}}}(1)
\]
is generated by its global holomorphic sections at every point
not in $\{ W = 0 \} \cup \{ Z_i' = 0\}$.
\end{theorem}

\subsection{Prescribing the Base locus of Siu-Yeung jet differentials}

In~\cite{ Siu-Yeung-1996, Merker-Siu-Yeung-2014}, it is shown that
a surface $X^2 \subset \P^3$ having affine equation: $z^d = R(x,y)$,
where $R \in \C[ x, y]$ is a generic 
polynomial of high enough degree $d \gg 1$, 
the following holds. For every collection of
polynomials $A_{j,k,p,q} \in \C[ x, y]$ 
having degrees $\deg\, A_{j,k,p,q} \leqslant d - 3\, m - 1$, 
the meromorphic jet differential:
\[
\footnotesize
\aligned
\frac{{\sf J}(x,y,x',y',x'',y'')}{
R_y\cdot z^{m(d-1)}}
\,=\,
\frac{1}{
R_y\cdot z^{m(d-1)}}
\sum_{j+k+p+3q=m}\,
A_{j,k,p,q}\,
(x')^j\,(y')^k\,
(R')^p\,
\left\vert\!
\begin{array}{cc}
x' & R'
\\
x'' & R''
\end{array}
\!\right\vert^q\,
(R)^{m-p-q},
\endaligned
\]
where:
\[
R'
:=
R_x\,x'
+
R_y\,y',
\ \ \ \ \ \ \
R''
:=
R_x\,x''
+
R_y\,y''
+
R_{xx}\,(x')^2
+
2\,R_{xy}\,x'\,y'
+
R_{yy}\,(y')^2,
\]
possesses a restriction to $X^2$
which is a {\em holomorphic} section of the bundle 
of the Green-Griffiths jet bundle $\mathcal{ E}_{2,m}^{\rm GG}T_X^*$,
provided only that the polynomial numerator:
\[
{\sf J}(x,y,x',y',x'',y'')
\,\equiv\,
R_y(x,y)\,
\widetilde{\sf J}(x,y,x',y',x'',y'')
\]
is divisible by $R_y$, which happens to be satisfiable 
for $m = 81$ and $d = 729$, and more generally whenever $m \geqslant 81$
and $d \geqslant 3\, m$.

An expansion yields: 
\[
{\sf J}
=
\sum_{\alpha+\beta+3\gamma=m}\,
\Lambda_{\alpha,\beta,\gamma}
\big(
A_{{\scriptscriptstyle{\bullet}}},
J_{x,y}^2R
\big)\,
(x')^\alpha\,(y')^\beta\,
\left\vert\!
\begin{array}{cc}
x' & y'
\\
x'' & y''
\end{array}
\!\right\vert^\gamma,
\]
in terms of some $\Lambda_{\scriptscriptstyle{\bullet}}$ that are linear in the 
$A_{\scriptscriptstyle{\bullet}}$ and polynomial in the $2$-jet $J_{x, y}^2 R$.

Since the vector fields of Theorem~\ref{slanted-complement} have a
maximal pole order $8$ here, lowering $\deg\, A_{ j, k, p,
q} \leqslant d - 11\, m - 1$ enables to conclude, as in
Proposition~\ref{Y-locus}, that for a generic curve $\{ R = 0\}
\subset \P^2$, nonconstant entire holomorphic maps $f \colon \C \to
\P^2 \backslash \{ R = 0\}$ land inside the common zero set of all the
$\Lambda_{\scriptscriptstyle{\bullet}}$, for $d \geqslant 2\,916$.

\medskip\noindent
{\bf Open Problem.}
{\sl Control or prescribe the base locus of coordinate jet differentials.}

\medskip

A conjecturally accessible strategy is as follows, of course
extendable to arbitrary dimensions.  For convenience,
replace $m \mapsto 3\, m$.  Decompose ${\sf J} = {\sf J}^{\sf top} +
{\sf J}_{\sf sub}^{\sf cor}$, where:
\[
\footnotesize
\aligned
{\sf J}^{\sf top}
&
:=
1\cdot
\left\vert\!
\begin{array}{cc}
x' & R'
\\
x'' & R''
\end{array}
\!\right\vert^m\,
(R)^{2m}
=
\left(
{\textstyle{
\left\vert\!
\begin{array}{cc}
x' & y'
\\
x'' & y''
\end{array}
\!\right\vert^m}}\,
R_y
+
(x')^3\,R_{xx}
+
2\,(x')^2y'\,R_{xy}
+
x'(y')^2\,R_{yy}
\right)^m\,
(R)^{2m},
\\
{\sf J}_{\sf sub}^{\sf cor}
&
:=
\sum_{j+k+p+3q=3m
\atop
q\leqslant m-1}\,
A_{j,k,p,q}(x,y)\,
(x')^{j}\,(y')^{k}\,
(R')^p\,
\left\vert\!
\begin{array}{cc}
x' & R'
\\
x'' & R''
\end{array}
\!\right\vert^q\,
(R)^{m-p-q}.
\endaligned
\]
Since $\big( (x'y''-y'x'')\,R_y \big)^m$ in 
${\sf J}^{\sf top}$ is divisible by $R_y$, 
Proposition~\ref{Y-locus} would show that entire curves land
in $\{ (R_y)^{m-1} = 0\}$, and exchanging $x \leftrightarrow y$,
in $\{ (R_x)^{m-1} = 0\}$, hence are constant because
$\emptyset = \{ 0 = R = R_x = R_y\}$ by smothness of $\{ R = 0\}$.

However, {\em all} $\Lambda_{ \alpha, \beta, \gamma}$, not just
$\Lambda_{ 0, 0, m}$, should be divisible by $R_y$ in order that the
restriction to the projectivization of $\{ z^d = R(x, y)\}$ of ${\sf
J} \big/ \big( R_y\, z^{3m(d-1)} \big)$ be a holomorphic jet
differential, because modulo $R_y$:
\[
{\sf J}^{\sf top}
\equiv
\Big(
(x')^3\,R_{xx}
+
2\,(x')^2y'\,R_{xy}
+
x'(y')^2\,R_{yy}
\Big)^m\,
(R)^{2m}
\]
is nonzero. The strategy is to use ${\sf J}_{\sf sub}^{\sf cor}$
in order to {\em correct} this remainder. Conjecturally,
the linear map which, to the $A_{{\scriptscriptstyle{\bullet}}}$ of
${\sf J}_{\sf sub}^{\sf cor}$, associates the 
coefficients of a basis $x^h y^i$ of $\C[ x, y] / \langle R_y \rangle$
in all the monomials $(x')^\alpha (y')^\beta$ with $\alpha + \beta = 3m$
is submersive, also in arbitrary dimension, which would hence
terminate.


\vfill\end{document}

%% file: u-v-x-y-z.pstex_t
\begin{picture}(0,0)%
\includegraphics{u-v-x-y-z.pstex}%
\end{picture}%
\setlength{\unitlength}{4144sp}%
\begingroup\makeatletter\ifx\SetFigFont\undefined%
\gdef\SetFigFont#1#2#3#4#5{%
  \reset@font\fontsize{#1}{#2pt}%
  \fontfamily{#3}\fontseries{#4}\fontshape{#5}%
  \selectfont}%
\fi\endgroup%
\begin{picture}(3488,782)(812,-497)
\end{picture}%

%% file: application-de-Gauss.pstex_t
\begin{picture}(0,0)%
\includegraphics{application-de-Gauss.pstex}%
\end{picture}%
\setlength{\unitlength}{4144sp}%
\begingroup\makeatletter\ifx\SetFigFont\undefined%
\gdef\SetFigFont#1#2#3#4#5{%
  \reset@font\fontsize{#1}{#2pt}%
  \fontfamily{#3}\fontseries{#4}\fontshape{#5}%
  \selectfont}%
\fi\endgroup%
\begin{picture}(5302,1980)(1274,-2081)
\put(1296,-1196){\makebox(0,0)[lb]{\smash{{\SetFigFont{9}{10.8}{\familydefault}{\mddefault}{\updefault}{\color[rgb]{0,0,0}$S$}%
}}}}
\put(4779,-1939){\makebox(0,0)[lb]{\smash{{\SetFigFont{9}{10.8}{\familydefault}{\mddefault}{\updefault}{\color[rgb]{0,0,0}$\Sigma$}%
}}}}
\put(2949,-922){\makebox(0,0)[lb]{\smash{{\SetFigFont{9}{10.8}{\familydefault}{\mddefault}{\updefault}{\color[rgb]{0,0,0}$\mathcal{A}_S$}%
}}}}
\put(4559,-1303){\makebox(0,0)[lb]{\smash{{\SetFigFont{9}{10.8}{\familydefault}{\mddefault}{\updefault}{\color[rgb]{0,0,0}$\mathcal{A}_\Sigma$}%
}}}}
\put(3083,-608){\makebox(0,0)[lb]{\smash{{\SetFigFont{7}{8.4}{\familydefault}{\mddefault}{\updefault}{\color[rgb]{0,0,0}$p$}%
}}}}
\put(3185,-690){\makebox(0,0)[lb]{\smash{{\SetFigFont{6}{7.2}{\familydefault}{\mddefault}{\updefault}{\color[rgb]{0,0,0}$q$}%
}}}}
\put(3237,-775){\makebox(0,0)[lb]{\smash{{\SetFigFont{6}{7.2}{\familydefault}{\mddefault}{\updefault}{\color[rgb]{0,0,0}$q$}%
}}}}
\put(2847,-685){\makebox(0,0)[lb]{\smash{{\SetFigFont{6}{7.2}{\familydefault}{\mddefault}{\updefault}{\color[rgb]{0,0,0}$q$}%
}}}}
\put(2955,-644){\makebox(0,0)[lb]{\smash{{\SetFigFont{6}{7.2}{\familydefault}{\mddefault}{\updefault}{\color[rgb]{0,0,0}$q$}%
}}}}
\put(4623,-795){\makebox(0,0)[lb]{\smash{{\SetFigFont{8}{9.6}{\familydefault}{\mddefault}{\updefault}{\color[rgb]{0,0,0}Auxiliary unit sphere}%
}}}}
\end{picture}%

%% file: M-h-polydiscs.pstex_t
\begin{picture}(0,0)%
\includegraphics{M-h-polydiscs.pstex}%
\end{picture}%
\setlength{\unitlength}{4144sp}%
\begingroup\makeatletter\ifx\SetFigFont\undefined%
\gdef\SetFigFont#1#2#3#4#5{%
  \reset@font\fontsize{#1}{#2pt}%
  \fontfamily{#3}\fontseries{#4}\fontshape{#5}%
  \selectfont}%
\fi\endgroup%
\begin{picture}(4693,1052)(874,-2319)
\put(2002,-1922){\makebox(0,0)[lb]{\smash{{\SetFigFont{10}{12.0}{\familydefault}{\mddefault}{\updefault}{\color[rgb]{0,0,0}$p$}%
}}}}
\put(4844,-2009){\makebox(0,0)[lb]{\smash{{\SetFigFont{10}{12.0}{\familydefault}{\mddefault}{\updefault}{\color[rgb]{0,0,0}$p'$}%
}}}}
\put(1612,-1507){\makebox(0,0)[lb]{\smash{{\SetFigFont{10}{12.0}{\familydefault}{\mddefault}{\updefault}{\color[rgb]{0,0,0}${\sf U}_p$}%
}}}}
\put(943,-1939){\makebox(0,0)[lb]{\smash{{\SetFigFont{10}{12.0}{\familydefault}{\mddefault}{\updefault}{\color[rgb]{0,0,0}$M$}%
}}}}
\put(4459,-1546){\makebox(0,0)[lb]{\smash{{\SetFigFont{10}{12.0}{\familydefault}{\mddefault}{\updefault}{\color[rgb]{0,0,0}$h({\sf U}_p)$}%
}}}}
\put(3339,-1656){\makebox(0,0)[lb]{\smash{{\SetFigFont{10}{12.0}{\familydefault}{\mddefault}{\updefault}{\color[rgb]{0,0,0}$h$}%
}}}}
\put(5117,-1833){\makebox(0,0)[lb]{\smash{{\SetFigFont{10}{12.0}{\familydefault}{\mddefault}{\updefault}{\color[rgb]{0,0,0}$M'$}%
}}}}
\put(902,-1414){\makebox(0,0)[lb]{\smash{{\SetFigFont{10}{12.0}{\familydefault}{\mddefault}{\updefault}{\color[rgb]{0,0,0}$\C^{n+c}$}%
}}}}
\put(5552,-1470){\makebox(0,0)[lb]{\smash{{\SetFigFont{10}{12.0}{\familydefault}{\mddefault}{\updefault}{\color[rgb]{0,0,0}${\C'}^{n+c}$}%
}}}}
\end{picture}%

%% file: 3-balls.pstex_t
\begin{picture}(0,0)%
\includegraphics{3-balls.pstex}%
\end{picture}%
\setlength{\unitlength}{4144sp}%
\begingroup\makeatletter\ifx\SetFigFont\undefined%
\gdef\SetFigFont#1#2#3#4#5{%
  \reset@font\fontsize{#1}{#2pt}%
  \fontfamily{#3}\fontseries{#4}\fontshape{#5}%
  \selectfont}%
\fi\endgroup%
\begin{picture}(4212,1357)(1181,-2510)
\put(1196,-1300){\makebox(0,0)[lb]{\smash{{\SetFigFont{10}{12.0}{\familydefault}{\mddefault}{\updefault}{\color[rgb]{0,0,0}$\C^{n+1}$}%
}}}}
\put(5242,-2273){\makebox(0,0)[lb]{\smash{{\SetFigFont{10}{12.0}{\familydefault}{\mddefault}{\updefault}{\color[rgb]{0,0,0}$M$}%
}}}}
\put(2867,-1426){\makebox(0,0)[lb]{\smash{{\SetFigFont{10}{12.0}{\familydefault}{\mddefault}{\updefault}{\color[rgb]{0,0,0}${\sf U}_p$}%
}}}}
\put(3665,-2036){\makebox(0,0)[lb]{\smash{{\SetFigFont{10}{12.0}{\familydefault}{\mddefault}{\updefault}{\color[rgb]{0,0,0}$p'$}%
}}}}
\put(2918,-2026){\makebox(0,0)[lb]{\smash{{\SetFigFont{10}{12.0}{\familydefault}{\mddefault}{\updefault}{\color[rgb]{0,0,0}$p$}%
}}}}
\put(3293,-2126){\makebox(0,0)[lb]{\smash{{\SetFigFont{10}{12.0}{\familydefault}{\mddefault}{\updefault}{\color[rgb]{0,0,0}$p_0$}%
}}}}
\put(3618,-1423){\makebox(0,0)[lb]{\smash{{\SetFigFont{10}{12.0}{\familydefault}{\mddefault}{\updefault}{\color[rgb]{0,0,0}${\sf U}_{p'}$}%
}}}}
\put(1227,-2259){\makebox(0,0)[lb]{\smash{{\SetFigFont{10}{12.0}{\familydefault}{\mddefault}{\updefault}{\color[rgb]{0,0,0}$M$}%
}}}}
\put(4827,-1468){\makebox(0,0)[lb]{\smash{{\SetFigFont{10}{12.0}{\familydefault}{\mddefault}{\updefault}{\color[rgb]{0,0,0}$\mathcal{ K} = a'\,\mathcal{K}'$}%
}}}}
\end{picture}%

%% file: faisceaux-infini-x-y.pstex_t
\begin{picture}(0,0)%
\includegraphics{faisceaux-infini-x-y.pstex}%
\end{picture}%
\setlength{\unitlength}{4144sp}%
\begingroup\makeatletter\ifx\SetFigFont\undefined%
\gdef\SetFigFont#1#2#3#4#5{%
  \reset@font\fontsize{#1}{#2pt}%
  \fontfamily{#3}\fontseries{#4}\fontshape{#5}%
  \selectfont}%
\fi\endgroup%
\begin{picture}(4986,1426)(363,-1470)
\put(1045,-169){\makebox(0,0)[lb]{\smash{{\SetFigFont{10}{12.0}{\familydefault}{\mddefault}{\updefault}{\color[rgb]{0,0,0}$\P^2$}%
}}}}
\put(786,-1208){\makebox(0,0)[lb]{\smash{{\SetFigFont{10}{12.0}{\familydefault}{\mddefault}{\updefault}{\color[rgb]{0,0,0}$\infty_y$}%
}}}}
\put(4965,-856){\makebox(0,0)[lb]{\smash{{\SetFigFont{10}{12.0}{\familydefault}{\mddefault}{\updefault}{\color[rgb]{0,0,0}$\infty_x$}%
}}}}
\put(3489,-214){\makebox(0,0)[lb]{\smash{{\SetFigFont{10}{12.0}{\familydefault}{\mddefault}{\updefault}{\color[rgb]{0,0,.69}$X^1$}%
}}}}
\end{picture}%

%% file: plongement-croisements.pstex_t
\begin{picture}(0,0)%
\includegraphics{plongement-croisements.pstex}%
\end{picture}%
\setlength{\unitlength}{4144sp}%
\begingroup\makeatletter\ifx\SetFigFont\undefined%
\gdef\SetFigFont#1#2#3#4#5{%
  \reset@font\fontsize{#1}{#2pt}%
  \fontfamily{#3}\fontseries{#4}\fontshape{#5}%
  \selectfont}%
\fi\endgroup%
\begin{picture}(2721,1814)(2403,-2818)
\put(2911,-1909){\makebox(0,0)[lb]{\smash{{\SetFigFont{11}{13.2}{\familydefault}{\mddefault}{\updefault}{\color[rgb]{0,0,0}$X^1$}%
}}}}
\put(4516,-2097){\makebox(0,0)[lb]{\smash{{\SetFigFont{11}{13.2}{\familydefault}{\mddefault}{\updefault}{\color[rgb]{0,0,0}$x$}%
}}}}
\put(3692,-1576){\makebox(0,0)[lb]{\smash{{\SetFigFont{11}{13.2}{\familydefault}{\mddefault}{\updefault}{\color[rgb]{0,0,0}$y$}%
}}}}
\put(3614,-1151){\makebox(0,0)[lb]{\smash{{\SetFigFont{11}{13.2}{\familydefault}{\mddefault}{\updefault}{\color[rgb]{.69,0,0}$\infty_y$}%
}}}}
\put(5109,-2147){\makebox(0,0)[lb]{\smash{{\SetFigFont{11}{13.2}{\familydefault}{\mddefault}{\updefault}{\color[rgb]{.69,0,0}$\infty_x$}%
}}}}
\put(5105,-1840){\makebox(0,0)[lb]{\smash{{\SetFigFont{11}{13.2}{\familydefault}{\mddefault}{\updefault}{\color[rgb]{.69,0,0}$\C_{\infty,x}^1$}%
}}}}
\put(2467,-2695){\makebox(0,0)[lb]{\smash{{\SetFigFont{11}{13.2}{\familydefault}{\mddefault}{\updefault}{\color[rgb]{0,0,.69}$\C^2$}%
}}}}
\end{picture}%